\documentclass[11pt]{article} 
\usepackage[top=99pt, bottom=75pt, left=72pt, right=70pt]{geometry}
\usepackage[nottoc]{tocbibind} 
\usepackage[utf8]{inputenc}
\usepackage{amssymb,amsmath,amscd,amsxtra}
\usepackage[all]{xy}
\usepackage{mathtools}
\usepackage{tensor}
\usepackage{faktor}
\usepackage{dsfont}
\usepackage{graphicx}
\usepackage[rgb]{xcolor}
\usepackage{tikz, tikz-cd}
\usetikzlibrary{arrows, matrix, intersections, math, calc, decorations.pathmorphing, decorations.markings, positioning}
\usepackage{enumitem}
\usepackage{hyperref}
\definecolor{R}{RGB}{255, 0, 34}
\definecolor{B}{RGB}{0, 85, 238}
\hypersetup{colorlinks = true,
            linkcolor = R,
            urlcolor  = B,
            citecolor = B,
            anchorcolor = B}
\newtheorem{theorem}{Theorem}[section]
\newtheorem{lemma}[theorem]{Lemma}

\newtheorem{proposition}[theorem]{Proposition}  
\newtheorem{corollary}[theorem]{Corollary} 

\newtheorem{remark}[theorem]{Remark}

\newtheorem{definition}[theorem]{Definition}

\newcommand{\beq}{\begin{equation}}
\newcommand{\eeq}{\end{equation}}
\newcommand{\beqa}{\begin{eqnarray}}
\newcommand{\eeqa}{\end{eqnarray}}
\newcommand{\beaa}{\begin{eqnarray*}}
\newcommand{\ben}{\begin{eqnarray*}}
\newcommand{\eaa}{\end{eqnarray*}}
\newcommand{\een}{\end{eqnarray*}}

\newcommand{\ZZ}{\mathbb{Z}}
\newcommand{\CC}{\mathbb{C}}
\newcommand{\PP}{\mathbb{P}}
\newcommand{\RR}{\mathbb{R}}
\newcommand{\ii}{\mathbf{i}}

\newcommand{\A}{\mathcal{A}}

\renewcommand{\O}{\mathcal{O}}

\newcommand{\M}{\mathcal{M}}

\newcommand*{\qed}{\hfill\ensuremath{\square}}

\def\proof{{\em Proof.\ }}

\DeclareFontFamily{U}{mathc}{}
\DeclareFontShape{U}{mathc}{m}{it}%
{<->s*[1.03] mathc10}{}
\DeclareMathAlphabet{\mathcal}{U}{mathc}{m}{it}

\def\lieg{{\mathfrak{g}}}

\title{Elliptic orbifold lines and integrable hierarchies}
\author{Todor Milanov}

\begin{document} 

\maketitle

\begin{abstract}
In 2003, Givental and the author constructed Hirota quadratic (or bilinear) equations for simple singularities -- a result later utilized by Fan, Jarvis, and Ruan to establish the generalized Witten conjecture. Whether this construction extends to broader classes of singularities has since remained an open question. In this paper, we resolve this question affirmatively for simple elliptic singularities. 
More precisely, using mirror symmetry, we construct Hirota quadratic equations for simple elliptic singularities by showing that the Gromov--Witten invariants of the elliptic orbifold lines satisfy a system of bilinear equations.

\end{abstract}

\setcounter{tocdepth}{2}
\tableofcontents

\section{Introduction}

Introduced by Witten \cite{Wi1991}, Gromov--Witten (GW) invariants
have driven rapid developments across algebraic (or complex) geometry,
symplectic geometry, and integrable systems. The deep connection to
integrable systems originated with Witten's conjecture \cite{Wi1991}
-- proved by Kontsevich \cite{Ko1992} -- stating that the generating
function of the GW invariants of a point is a tau-function of the KdV
hierarchy. Motivated by Witten's proposal, Dubrovin and Zhang
\cite{DZ2001} introduced a systematic formalism for constructing
bi-Hamiltonian integrable hierarchies. Assuming certain technical
conjectures regarding the polynomiality of the Poisson bracket, they
constructed a bi-Hamiltonian hierarchy associated with any semisimple
Frobenius manifold, thereby generalizing Witten's conjecture to target
manifolds with semisimple quantum cohomology. The Dubrovin--Zhang (DZ)
conjectures were subsequently established in \cite{BPS2012} and
\cite{LWZ2025}. Over the past decade, this construction was extended
by Buryak--Posthuma--Shadrin to drop the semisimplicity assumption by replacing GW invariants with partial Cohomological Field Theories (partial CohFTs); see the recent preprint \cite{BLS2026} for a comprehensive account.

The primary motivation for formulating 2D integrable systems in terms of Hirota bilinear equations (HBEs) is to uncover the geometric structure of their solution spaces. By the classical Liouville--Arnold theorem (see \cite[Chapter 10, Section 49]{Ar1989}), the phase space of an integrable system with finitely many degrees of freedom forms a Lagrangian fibration whose regular fibers are Abelian Lie groups. While no direct infinite-dimensional analogue of the Liouville--Arnold theorem exists for 2D field theories, the space of solutions in many key examples is parametrized by an infinite-dimensional homogeneous space. This structure was first discovered by Sato \cite{Sato1981} for the KP hierarchy and subsequently extended by Kac and Wakimoto \cite{KW1987}. Given that the Dubrovin--Zhang (DZ) hierarchy and its generalizations encompass a vast class of integrable systems, analyzing the geometric structure of their solutions is a natural and fundamental problem.

Although the Hamiltonian framework of Dubrovin and Zhang has proven highly successful, constructing HBEs within GW theory remains a major challenge. The principal difficulty lies in establishing a framework independent of GW theory, contrasting with the DZ approach where GW invariants are utilized directly to construct commuting flows. Our approach originates in \cite{Giv2003}, where Givental observed that the vertex operators of the $n$-KdV hierarchy can be expressed in terms of the period integrals for $A_{n-1}$-singularities. In a joint paper with Givental \cite{GM2005}, we extended this construction to all simple singularities. It was a natural next step to test these ideas within the framework of simple elliptic singularities $E_6^{(1,1)}, E_7^{(1,1)}, E_8^{(1,1)}$. In the classification of singularities proposed by Arnold and his school, simple elliptic singularities represent the most immediate generalization of simple singularities. While they have been extensively studied from diverse mathematical perspectives, our initial attempt to construct HBEs in the simple elliptic setting proved unsuccessful, indicating the need for novel conceptual frameworks. Ultimately, developments in mirror symmetry over the past two decades provided the necessary tools to resolve these difficulties. The primary goal of this paper is to describe our construction.

Let us now focus on the case of Frobenius manifolds arising from quantum cohomology. Let $X$ be a complex orbifold whose orbit space is a projective variety. Following Abramovich--Graber--Vistoli \cite{AGV2008} and Chen--Ruan \cite{CR2002}, one can define the orbifold GW invariants of $X$. In particular, the genus-0 GW invariants define the orbifold quantum cohomology. Since our approach relies on the Givental--Teleman higher-genus reconstruction, our primary interest lies in orbifolds $X$ that exhibit semisimple quantum cohomology. While the classification of such targets in higher dimensions remains a wide open problem, the classification in complex dimension 1 is fully established: $X$ must be an orbifold projective line $\mathbb{P}_{a_{1},\dots ,a_{r}}^{1}$, where $a_i \geq 2$ denote the orders of the isotropy groups of the orbifold points. According to Milanov--Ruan \cite{MiR} (see also \cite{IMRS}), simple elliptic singularities are mirror partners -- in the sense of Givental -- to the elliptic orbifold lines $\mathbb{P}_{3,3,3}^{1}$, $\mathbb{P}_{2,4,4}^{1}$, and $\mathbb{P}_{2,3,6}^{1}$. This mirror correspondence yields two crucial consequences. First, it allows us to reformulate our problem from singularity theory into the framework of GW theory. Second, it highlights an intermediate class between simple and simple elliptic singularities: the Fano orbifold lines, i.e., $\mathbb{P}_{a_{1},a_{2},a_{3}}^{1}$ satisfying $\tfrac{1}{a_1}+\tfrac{1}{a_2}+\tfrac{1}{a_3}-1>0$. The GW theory of Fano orbifold lines is structurally simpler because it is closer to the fully understood GW theory of the smooth projective line. This Fano setting \cite{MST2016} is precisely the framework in which most of the techniques for the present paper were originally developed.

Let us comment on three key developments that lead to the results in this paper. The first essential ingredient of our construction is the $\Gamma $-integral structure introduced by Iritani \cite{Iri2009}. Givental's original framework involved the Milnor ring and the vanishing cycles of the singularity. Under mirror symmetry, these are identified with the topological orbifold $K$-ring $K^0(\mathbb{P}^1_{a_1,a_2,a_3})$ and the so-called reflection vectors in quantum cohomology, the latter of which can be determined via the refined Dubrovin conjecture (see \cite{CM2024}). This $K$-theoretic framework successfully resolves several of the complications encountered in singularity theory. The second crucial development is our joint paper with Bojko Bakalov \cite{BM2013}. Although that work was set in the framework of simple singularities, our central observation -- that Givental’s vertex operators define a twisted representation of the lattice vertex algebra associated with the Milnor lattice -- generalizes to GW theory as well. More precisely, we discovered a sequence of products between twisted fields (cf. Section~\ref{sec:Bnp}) that satisfy the Borcherds axioms of a vertex algebra. The bilinear operator in this paper can be interpreted in the following way. Consider the lattice vertex algebra corresponding to $K^0(\mathbb{P}^1_{a_1,a_2,a_3})$. The states of conformal weight $1$ and the Borcherds $0$-product define a Lie algebra known as the Lie algebra of physical states. Our twisted representation provides a realization of this Lie algebra in which the bilinear operator takes the form of a Casimir operator. This representation-theoretic interpretation will be pursued in a separate paper, and we expect our construction to have interesting applications to conformal field theory. Finally, the third important development (see \cite{Mi2019, MilSa}) involves establishing certain analytic properties of the so-called propagators -- functions that arise as a byproduct when composing two vertex operators. We refer to Section~\ref{sec:ppp} for further details. The techniques presented in that section are essential for proving that the generating function of GW invariants satisfies the HBEs. Notably, the original argument in our joint work with Givental \cite{GM2005} relied on a specific result characterizing the kernel of the monodromy representation in vanishing cohomology. While this result is well known for simple singularities, its generalization to other singularity classes remains a major open problem.

The construction presented in this paper naturally applies to simple elliptic singularities, which correspond to the elliptic orbifold lines $\mathbb{P}_{3,3,3}^{1}$, $\mathbb{P}_{2,4,4}^{1}$, and $\mathbb{P}_{2,3,6}^{1}$. With minor modifications, our approach can also be extended to the remaining elliptic orbifold line, $\mathbb{P}_{2,2,2,2}^{1}$, provided that the refined Dubrovin conjecture is established for this case. As a byproduct of our proof, we establish a new property of the Kac–Wakimoto formulation of the KdV hierarchy. Specifically, the bilinear equations of the Kac–Wakimoto hierarchy yield an infinite sequence of bilinear relations corresponding to the higher modes of the Casimir operator (see Proposition~\ref{prop:kdv-kw}). This structural feature, originally observed by Bojko Bakalov, appears to hold generally for all Kac–Wakimoto hierarchies. Finally, we note that our integrable hierarchy does not include dynamics with respect to the variables $q_{k,0,0}$ ($k>1$). These variables, which correspond to descendant insertions of $\phi_{0,0}=1$, are treated as parameters throughout this work. Nevertheless, as a consequence of Proposition~\ref{prop:kdv-kw}, the framework introduced in \cite{Mi} can be leveraged to construct an extension of our HBEs that incorporates these dynamics. We leave the details of this extension to a subsequent paper.

\subsection{Elliptic orbifold lines}\label{sec:eol}
Let $a=(a_1,a_2,a_3)$ be a triple of positive integers satisfying 
$\tfrac{1}{a_1}+\tfrac{1}{a_2}+\tfrac{1}{a_3}=1$. There are only 3 such triples: $(3,3,3)$, $(2,4,4)$, and $(2,3,6)$.  Put  
\ben
G=\{t\in (\CC^*)^3\ |\ t_1^{a_1}=t_2^{a_2}=t_3^{a_3} \}.
\een
The orbifold $\PP^1_a:=\PP^1_{a_1,a_2,a_3}$ is defined to be the orbifold quotient $[Z/G]$ where 
\ben
Z:=\{z_1^{a_1}+z_2^{a_2}+z_3^{a_3}=0\}\subset (\CC\setminus{\{0\}})^3.
\een
The orbit space of $\PP^1_a$ is $Z/G\cong \PP^1$ where the
isomorphism is induced by the map
\ben
Z\to \PP^1,\quad (z_1,z_2,z_3)\mapsto [z_1^{a_1}:z_2^{a_2}].
\een
By definition the inertia orbifold
\ben
I\PP^1_{a}=\bigsqcup_{g\in G} [Z^g/G],\quad Z^g:=\{z\in Z\ |\ g z=z\}.
\een
The connected components of the inertia orbifold are non-empty only if
\ben
g\in G_1\cup G_2\cup G_3,
\een
where $G_i$ is the subgroup of $G$ consisting of $t=(t_1,t_2,t_3)$
such that $t_j=1$ for all $j\neq i$. Note that $G_i$ is a cyclic group
of order $a_i$.
The connected component for $g=1$ is $\PP^1_a$, while for
$g\neq 1$ it is isomorphic to the orbifold point $[{\rm pt}/G_i]$ where $i$ is such
that $g\in G_i$. The connected components for $g\neq 1$ are
known as {\em twisted sectors}. The {\em Chen--Ruan} or {\em orbifold} cohomology of $\PP^1_a$ is by definition
\ben
H:=H(I\PP^1_a,\CC)=
\bigoplus_{*\in \mathbb{Q}}\,
\bigoplus_{g\in G_1\cup G_2\cup G_3} 
H^{*-\iota(g)}(Z^g/G,\CC), 
\een
where $*$ denotes complex degree (i.e. half of the standard
degree) and the shift $\iota(g)$ is defined for all finite order elements
\ben
g=(e^{2\pi\,  \ii \,\alpha_1}, e^{2\pi\,  \ii \,\alpha_2},
e^{2\pi\,  \ii\, \alpha_3})\in G,
\quad
\alpha_i\in \mathbb{Q}\cap [0,1)
\een
by
\ben
\iota(g):=\alpha_1+\alpha_2+\alpha_3.
\een
We fix a basis of $H$ as follows
\ben
\phi_{0,0}=1,\quad \phi_{0,1}=P,
\een
and
\ben
\phi_{i,p}=1\in H^0(Z^{g_{i,p}}/G,\CC),\quad 1\leq i\leq 3,\quad 1\leq
p\leq a_i-1,
\een
where $P\in H^2(\PP^1,\CC)$ is the hyperplane class and $g_{i,p}\in
G_i$ is the element whose $i$th entry is 
$e^{2\pi\ii\, p/a_i}$. Let us define $a_0=1$, then the Chen--Ruan
degree of the above basis is given by
\ben
\operatorname{deg}_{\rm CR}(\phi_{i,p})=p/a_i,\quad 0\leq
i\leq 3.
\een
Let us denote by 
\ben
\mathcal{B}:=\{(0,0),(0,1)\}\sqcup\{(i,p)\, |\, 
1\leq i\leq 3, 1\leq p\leq a_i-1\}
\een
the set of indexes of the above basis. Sometimes, we will also write simply $\phi_{00}$ and $\phi_{01}$ instead of respectively $\phi_{0,0}$ and $\phi_{0,1}$. 

\medskip

The orbifold $K$-ring of $\PP^1_a$ has the following explicit description:
\ben
K^0(\PP^1_a)\cong \ZZ[L_1,L_2,L_3]/\langle L_i^{a_i}-L_j^{a_j}, (L_i-1)(L_j-1)\ |\
1\leq i\neq j\leq 3\rangle.
\een
The generators $L_i$ correspond to the orbifold line bundles
$[(Z\times \CC)/G]$, where $G$ acts on the fiber $\CC$ via the
character
\ben
G\to \CC^*, \quad g=(g_1,g_2,g_3)\mapsto g_i.  
\een
Following Iritani \cite{Iri2009}, we construct an integral structure in $H$, that is, an embedding 
$\Psi: K^0(\PP^1_a)\to H$ of the orbifold $K$-ring as a lattice in $H$. The definition of $\Psi$ will be recalled in Section \ref{sec:iis} (see formula \eqref{psi-map}). We have the following explicit formula:
\beq\label{psi_E}
\Psi(E) = \operatorname{rk}(E)\Big(1 -\log Q\, P\Big) +2\pi\ii
  \operatorname{deg}(E) P +
\sum_{j=1}^3\sum_{p=1}^{a_j-1} \Gamma(1-p/a_j) \chi_{j,p}(E) \phi_{j,p},
\eeq
where $\operatorname{rk}(E)$ is the rank of 
$E$, $\operatorname{deg}(E) :=\int_{|\PP^1_a|} c_1(E)$ is the
degree of $E$, and $\chi_{j,p}:K^{(0)}(\PP^1_a)\to \CC$ are {\em ring} homomorphisms. The rank is straightforward to compute, while the degree and the ring homomorphisms are uniquely determined by the following formulas:
\ben
\operatorname{deg}(L^p_i)=p/a_i,\quad 
\chi_{j,p}(L_i) :=e^{-\frac{2\pi\ii}{a_j} p \delta_{j,i}}
\een
and the properties respectively for $\operatorname{deg}:K^0(\PP^1_a)\to \CC$ to be $\ZZ$-linear and for $\chi_{j,p}:K^0(\PP^1_a)\to \CC$ to be a ring homomorphism. 

The vector space $H$ will be equipped with 3 bilinear pairings that will play very important role in this paper. The first one is the orbifold Poincar\'e pairing:
\beq\label{opp}
(\ ,\ ): H\otimes H\to \CC,\quad 
(\phi_i,\phi_j) = \int_{[I\PP^1_a]} \phi_i\cup \operatorname{inv}^*(\phi_j),
\eeq
where $\operatorname{inv}:I\PP^1_a \to I\PP^1_a$ is the involution that maps $Z^g/G\to Z^{g^{-1}}/G$ via the identity map and 
\ben
[I\PP^1_a]:= \sum_{g\in G_1\cup G_2\cup G_3} 
\frac{1}{|\operatorname{Stab}(Z^g)|}\, [Z^g/G]
\een
is the orbifold fundamental cycle of $I\PP^1_a$, where 
$\operatorname{Stab}(Z^g):=\{t\in G\ |\ t\cdot z=z\ \forall z\in Z^g\}$ and $[Z^g/G]$ is the fundamental cycle of $Z^g/G$ as a variety. Using the above definition, it is straightforward to find 
\ben
(\phi_{i,p},\phi_{j,q})= \frac{1}{a_i}\, \delta_{i,j}\, \delta_{p+q,a_i},\quad 
\forall (i,p), (j,q)\in \mathcal{B},
\een
where recall that $a_0:=1.$ The second pairing is the {\em Euler pairing}
\beq\label{euf}
\langle a,b\rangle := \frac{1}{2\pi} (a,
e^{\pi\ii \theta} b), \quad a,b\in H
\eeq
where 
\ben
\theta: H\to H,\quad 
\theta(\phi_{i,p}) = \Big(\tfrac{1}{2}-\tfrac{p}{a_i}\Big)\,\phi_{i,p}
\een
is the {\em grading operator}. Using the Kawasaki--Riemann--Roch formula, one can prove that the map $\Psi$ intertwines the pairing
$\langle\ ,\ \rangle$ and the Euler pairing on $K^0(X)$, that is,
\beq\label{euler-pairing}
\chi(E_1^\vee\otimes E_2) =
\langle \Psi(E_1),  \Psi(E_2)\rangle .
\eeq
We can visualize the Euler pairing via the diagram on Figure \ref{fig:ep}. The vertices are labeled with $1$, $L:=L_1^{a_1}=L_2^{a_2}=L_3^{a_3}$, $\gamma_{i,p}:=L_i^{-p}-L_i^{-p+1}$ ($1\leq i\leq 3$, $1\leq p\leq a_i-1$). First of all $\langle a, a\rangle=1$ for all vertices $a$. The arrows between two {\em different} vertices $a$ and $b$ are drawn according to the following rules: if $\langle a,b\rangle=i$, then we draw $|i|$ arrows from $a$ to $b$, the arrows are either dashed or solid depending on whether respectively $i>0$ or $i<0$, and no arrows from $a$ to $b$ means that $\langle a,b\rangle=0$. 
\medskip

\begin{figure}
\begin{tikzpicture}[>=Stealth]
\draw[gray,thick,->] (0,0) -- (1.4,0);
\draw[gray,thick,->] (1.5,0) -- (3.4,0);
\draw[gray,thick,->] (3.5,0) -- (5.4,0);
\draw[gray,thick,->] (5.5,0) -- (7.4,0);
\filldraw[black] (1.5,0) circle (4pt); 
\filldraw[black] (3.5,0) circle (4pt); 
\filldraw[black] (5.5,0) circle (4pt); 
\filldraw[black] (7.5,0) circle (4pt); 
\draw[gray,thick] (0,0) -- (0.8,0.908);
\draw[gray,thick,->] (0.88,1) -- (1.4,1.57);
\draw[gray,thick,->] (1.5,1.7) --(2.4,2.72);
\draw[gray,thick,-] (2.5,2.83) -- (3,3.4);
\filldraw[black] (3.125,3.542) circle (1pt);
\filldraw[black] (3.25,3.683) circle (1pt);
\filldraw[black] (3.375,3.825) circle (1pt); 
\draw[gray,thick,->] (3.5,3.967) -- (3.9,4.42);
\filldraw[black] (1.5,1.7) circle (4pt); 
\filldraw[black] (2.5,2.83) circle (4pt); 
\filldraw[black] (4,4.53) circle (4pt); 
\draw (3,1.7) node {$\gamma_{i,1}=L_i^{-1}-1$};
\draw (4.3,2.83) node {$\gamma_{i,2}=L_i^{-2}-L_i^{-1}$};
\draw (6.5,4.53) node {$\gamma_{i,a_i-1}=L_i^{-a_i+1}-L_i^{-a_i+2}$};
\draw[gray,thick,->] (0,0) -- (-1.4,0);
\draw[gray,thick,->] (-1.5,0) -- (-3.4,0);
\draw[gray,thick,->] (-3.5,0) -- (-5.4,0);
\filldraw[black] (-1.5,0) circle (4pt); 
\filldraw[black] (-3.5,0) circle (4pt); 
\filldraw[black] (-5.5,0) circle (4pt); 
\draw[gray, thick,dashed,->] (0.1,0) -- (0.1,1.9);
\draw[gray, thick,dashed,->] (-0.1,0) -- (-0.1,1.9);
\draw[gray,thick,->] (0,2) -- (1.4,1.72);
\draw[gray,thick,->] (0,2) -- (1.49,0.13);
\draw[gray,thick,->] (0,2) -- (-1.49,0.13);
\filldraw[black] (0,2) circle (4pt); 
\filldraw[black] (0,0) circle (4pt); 
\draw (0,-0.5) node {$1$};
\draw (0,2.5) node {$L$};
\end{tikzpicture}
\caption{
Solid arrows are assigned weight $-1$ and dashed arrows -- weight $+1$.
The Euler pairing between two different vertices is the weighted count of arrows between them. }
\label{fig:ep}
\end{figure}
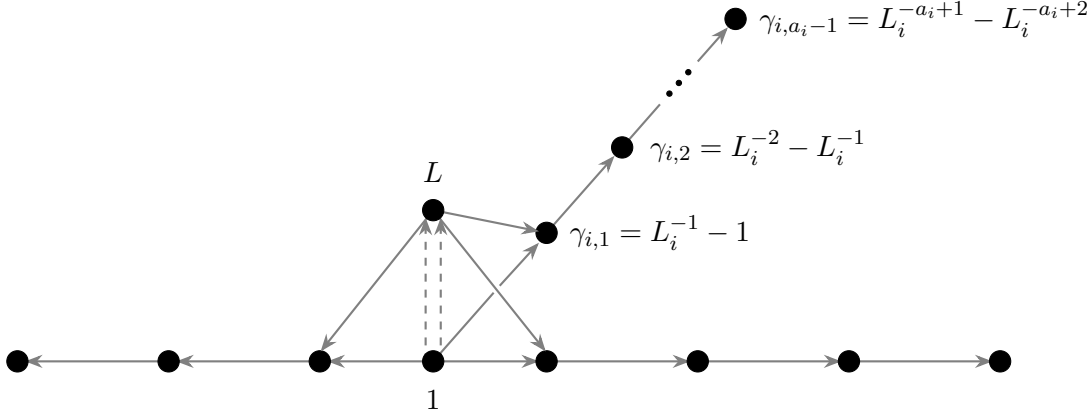

Finally, the third pairing 
\ben
(a|b):=\langle a,b\rangle + \langle b,a\rangle,\quad a,b\in H
\een
is called the {\em intersection pairing}. 
From now on we will use $\Psi$ to identify $H$ and $K^0(\PP^1_a)\otimes
\CC$. The above formulas allow us to say that the
intersection pairing is the symmetrization of the Euler pairing.

\subsection{Gromov--Witten invariants, vertex operators, and elliptic root systems}
\label{sec:GVE}
Recall that the descendant orbifold Gromov--Witten invariants of
$\PP^1_a$ are defined as intersection numbers
\beq\label{GW}
\langle \phi_1\psi_1^{k_1},\dots,\phi_r\psi_r^{k_r}\rangle_{g,r,d} :=
\int_{[\overline{\M}_{g,r}(\PP^1_a,d)]^{\rm virt}}
\operatorname{ev}^*(\phi_1\otimes \cdots\otimes \phi_r) \psi_1^{k_1}\cdots \psi_r^{k_r},
\eeq
where $\phi_i\in H$, $g,r,d\in \ZZ_{\geq 0}$, and the integral is interpreted as cap product with
the virtual fundamental cycle on the moduli space
$\overline{\M}_{g,r}(\PP^1_a,d)$ of orbifold stable 
maps (see \cite{AGV2008,CR2002}). The following generating function
\ben
\mathcal{D}(\hbar,\mathbf{t}) = \exp\Big(\sum_{g,r,d=0}^\infty
\hbar^{g-1}\frac{Q^d}{r!}
\langle \mathbf{t}(\psi_1),\dots,\mathbf{t}(\psi_r)\rangle_{g,r,d}\Big)
\een
is called the {\em total descendant potential}. Here $Q\in \CC^*$ is a
complex parameter called the {\em Novikov variable}, $\hbar$ is a formal variable, 
$t_0,t_1,\dots\in H$ are formal vector variables, and
$\mathbf{t}(z)=\sum_{k=0}^\infty t_k z^k$. The components of the vector
variable $t_k$ with respect to the basis $\phi_{i,p}$ ($(i,p)\in \mathcal{B}$) will be denoted by $t_{k,i,p}$. The total descendant potential is a formal power series belonging to the following {\em Fock space}:
\beq\label{fock}
\CC[t_{0,0,0}](\!(\hbar)\!)[\![
t_{k,i,p}\ : \ (k,i,p)\neq(0,0,0) ]\!],
\eeq
that is, formal power series in all $t_{k,i,p}$ with $(i,p)\in \mathcal{B}$, $k\geq 0$, and  $(k,i,p)\neq (0,0,0)$  whose coefficients are formal Laurent series in $\hbar$ whose coefficients are polynomials in $t_{0,0,0}$. The polynomial dependence on $t_{0,0,0}$ is a consequence of the so-called {\em string equation}. 

Let us define vertex operators acting on the Fock space \eqref{fock}. Put  
\beq\label{cper}
I_E^{(m)}(\lambda):= \Big(
\frac{
  \lambda^{\theta-m-\tfrac{1}{2}} }{
  \Gamma(\theta-m+\tfrac{1}{2}) }
\Big)\Psi(E),\quad E\in K^0(\PP^1_a),\quad m\in \CC.
\eeq
We refer to \eqref{cper} as the {\em calibrated periods} of $\PP^1_a$. Most of the time we will be interested in calibrated periods with $m\in \ZZ$, but we will need the general case too. 
Let us introduce the vertex operators
\beq\label{vop}
\Gamma^E(\lambda):=\prod_{(j,p)\in \mathcal{B}}\, \exp\left(
\sum_{k=0}^\infty 
\Big(
I_E^{(-k-1)}(\lambda),\phi_{j,p}
\Big) \,
\frac{  t_{k,j,p}  }{  \sqrt{\hbar}   }
\right)
\exp\left(
\sum_{k=0}^\infty 
\Big(
I_E^{(k)}(\lambda),\phi^{j,p} \Big) \,
(-1)^{k+1} \sqrt{\hbar}\partial_{k,j,p}
\right),
\eeq
where $\partial_{k,j,p} = \tfrac{\partial}{\partial t_{k,j,p}}$ and 
$\{\phi^{j,p}\}$ is the dual basis with respect to the orbifold  Poincar\'e pairing:
\ben
\phi^{0,0}= P,\quad
\phi^{0,1}=1,\quad
\phi^{j,p} = a_j \phi_{j,a_j-p}\quad
(1\leq j\leq 3,\ 1\leq p\leq a_j-1).
\een
Let us denote by $\Gamma_{jp}^E(\lambda)$ the vertex operator
corresponding to the $(j,p)$-term in the product defining
$\Gamma^E(\lambda)$. In particular, $\Gamma^E(\lambda)=\prod_{j,p}
\Gamma^E_{jp}(\lambda)$. Recalling formulas \eqref{psi_E} and \eqref{cper},
after a straightforward computation, we get 
\begin{align*}
\Gamma^E_{0,0}(\lambda)  = & \exp\left(
  \sum_{k=0}^\infty
  \frac{\lambda^k}{k!} \Big(
  2\pi\ii \operatorname{deg}(E) -
  \operatorname{rk}(E)\, \log Q  \Big) \,
\frac{  t_{k,0,0}  }{  \sqrt{\hbar}   }
\right)
\exp\left(
  -\operatorname{rk}(E)\, \sqrt{\hbar}\partial_{0,0,0}
\right), \\
  \Gamma^E_{0,1}(\lambda)  = &
\exp\left(
\operatorname{rk}(E)\,            
\sum_{k=0}^\infty
\frac{\lambda^{k+1}}{(k+1)!} \,
\frac{  t_{k,0,1}  }{  \sqrt{\hbar}   }
\right), \\
\Gamma^E_{j,p}(\lambda) = & \exp\left(
\frac{1}{a_j}
  \sum_{k=0}^\infty
                            \frac{\lambda^{k+p/a_j}}{
\frac{p}{a_j} (1+\frac{p}{a_j})\cdots (k+ \frac{p}{a_j})  } \, 
  \chi_{j,-p}(E)\,
\frac{  t_{k,j,p}  }{ \sqrt{\hbar}   }
                            \right) \times \\
  &
                          \exp\left(
                          -
                          \sum_{k=0}^\infty
                          \tfrac{p}{a_j}
                          \Big(\tfrac{p}{a_j}+1\Big)\cdots
                          \Big(\tfrac{p}{a_j}+k-1\Big)\,
                          \lambda^{-k-p/a_j}\, \chi_{j,p}(E)\,
                          \sqrt{\hbar}\partial_{k,j,p}
                          \right),
\end{align*}
where $1\leq j\leq 3$ and $1\leq p\leq a_j-1$. The following automorphism plays a central role in our construction:
\beq\label{cm}
\tau: K^0(\PP^1_a)\to K^0(\PP^1_a),\quad \tau(E)=T\PP^1_a\otimes E,
\eeq
where $T\PP^1_a$ is the orbifold tangent bundle. The vertex operators are $\tau$-twisted, that is, the analytic continuation along a counter-clock-wise oriented loop around $\lambda=0$ transforms $\Gamma^E(\lambda)$ into $\Gamma^{\tau(E)}(\lambda).$ We will refer to $\tau$ as the {\em classical monodromy} operator because under mirror symmetry it represents the big loop around the discriminant. Our construction of HQEs relies on the spectral decomposition
\beq\label{sd_tau}
K(\PP^1_a,\CC)=K_0(\PP^1_a,\CC)\oplus K_\perp (\PP^1_a,\CC),
\eeq
where $K(\PP^1_a,\CC) := K^0(\PP^1_a)\otimes \CC$, $K_0(\PP^1_a,\CC)$ is spanned by the eigenvectors of $\tau$ with eigenvalue 1, and $K_\perp(\PP^1_a,\CC)$ is the subspace spanned by the eigenvectors of eigenvalue $\neq 1$. There is a close relation between $\tau$ and the grading operator $\theta$. Namely, 
Iritani's map $\Psi$ induces an isomorphism $K(\PP^1_a,\CC)\cong H$ and $\tau=\Psi^{-1} e^{2\pi\ii \, (\theta-\tfrac{1}{2})}\,\Psi$. In particular, we get from here that $\tau$ is a finite order automorphism of order $l=\operatorname{lcm}(a_1,a_2,a_3)$. Note that if $a_1\leq a_2\leq a_3$, then $l=a_3$. It is also easy to check that
\ben
K_0(\PP^1_a,\CC)=\operatorname{Ker} (\ |\ )=
\{x\in K(\PP^1_a,\CC)\ |\ (x|y)=0\ \forall y\in K(\PP^1_a,\CC)\}
\een
and that the restriction of $(\ |\ )$ to $K_\perp(\PP^1_a,\CC)$ is non-degenerate. 
Let $\Phi:=\{E\in K^0(\PP^1_a)\ |\ (E|E)=2\}$. The set $\Phi$ coincides with the elliptic root system of type $E_{N-2}^{(1,1)}$ where $N=a_1+a_2+a_3-1$ (see \cite{Sa} for more details on elliptic root systems). The structure of $\Phi$ can be understood as follows. Put 
\begin{align*}
\Phi_{\rm aff} & :=\{ \alpha\in \Phi\ |\ \langle \alpha, L^{-1}\rangle =0\},\\
\Phi_\perp & :=
\{ \alpha\in \Phi\ |\ 
\langle \alpha, L^{-1}\rangle = \langle \alpha, L_3^{-a_3+1}\rangle =0\}.
\end{align*}
Then $\Phi_\perp$ is a root system of type $E_{N-2}$ and 
the subset (see Figure \ref{fig:ep}) 
\ben
\{1,\, 
\gamma_{i,p}\ (1\leq i\leq 2, 1\leq p\leq a_i-1),\, 
\gamma_{3,p} (1\leq p\leq a_3-2)\}\quad \subset \quad \Phi_\perp
\een
is a set of simple roots for which the corresponding Dynkin diagram coincides with the Dynkin diagram of type $E_{N-2}$. Furthermore, $\Phi_{\rm aff}$ is an affine root system of type $E_{N-2}^{(1)}$ and the subset $\{1, \gamma_{i,p}\ (1\leq i\leq 3, 1\leq p\leq a_i-1)\}$ is a set of simple roots for which the corresponding Dynkin diagram coincides with the affine Dynkin diagram of type $E_{N-2}^{(1)}$. Finally, we have 
\begin{align}
\label{ell-s}
\Phi  =\Phi_{\rm aff} + \ZZ\, (L-1)   = 
\{\alpha+m (L-1)\ |\ \alpha \in \Phi_{\rm aff},\ m\in \ZZ\} 
\end{align}
and 
\begin{align}
\label{aff-s}
\Phi_{\rm aff} = \Phi_\perp + \ZZ\,\delta  = 
\{\alpha+n\delta\ |\ \alpha\in \Phi_\perp,\ n\in \ZZ\}
\end{align}
where 
\beq\label{delta}
\delta:= l\, \Big(-2+
\sum_{i=1}^3 \frac{1}{a_i} (1+L_i^{-1}+\cdots + L_i^{-a_i+1})
\Big)=
l+\sum_{j=1}^3\sum_{p=1}^{a_j-1} 
l \,(1-p/a_j) \gamma_{j,p}.
\eeq
In other words, $\Phi_{\rm aff}^{\rm im}:=\ZZ\, \delta$ is the set of imaginary affine roots. 

\subsection{The bilinear operator}\label{sec:bilo}
The main ingredient in the definition of the HQEs for the total descendant potential of $\PP^1_a$ will be a bilinear operator, acting on the tensor square of the Fock space \eqref{fock}, of the following form:
\beq\label{affine_casimir}
\Omega_{\rm aff}(\lambda):=
\sum_{E\in \Phi_{\rm aff}} b_E(\lambda) \Gamma^E(\lambda)\otimes
\Gamma^{-E}(\lambda)-
\Big(
\sum_{F\in \Phi_{\rm aff}^{\rm im}} b_F(\lambda) 
\Gamma^F(\lambda)\otimes \Gamma^{-F}(\lambda)
\Big)
L(\lambda).
\eeq
Let us define the coefficients $b_E(\lambda)$ and $b_F(\lambda)$ and the operator $L(\lambda)$. 

Let us define $B^{E,F}(\lambda,\mu)$ by taking the product of vertex operators
\beq\label{product_f}
\Gamma^E(\lambda)\Gamma^F(\mu) =
B^{E,F}(\lambda,\mu) 
: \Gamma^E(\lambda)\Gamma^F(\mu)  :\ .
\eeq
In other words, $B^{E,F}(\lambda,\mu)$ is the factor that one gains by moving the annihilation part of $\Gamma^E(\lambda)$ through the creation part of $\Gamma^F(\mu)$. Since we have explicit formulas for the vertex operators, it is straightforward to obtain an explicit formula for $B^{E,F}(\lambda,\mu)$ -- see Section \ref{sec:phf}. We have
\beq\label{phase_B}
B^{E,F}(\lambda,\mu)=
e^{
-2\pi\ii\, 
\operatorname{rk}(E)
\operatorname{deg}(F) }\, 
Q^{\operatorname{rk}(E)\operatorname{rk}(F)}\,
\Big( 1- (\mu/\lambda)^{1/l}  \Big)^{(E|F)} \,
\prod_{k=1}^{l-1}
\Big( 1- \xi^k (\mu/\lambda)^{1/l}  \Big)^{(\tau^k(E)|F)}\, ,
\eeq
where $\xi=e^{2\pi\ii/l}$. 
The coefficients in front of the vertex operators are defined as follows:
\beq\label{b_E}
b_E(\lambda)=
\lim_{\mu\to \lambda} 
\frac{B^{E,E}(\lambda,\mu)}{(\lambda-\mu)^{(E|E)}},\quad 
E\in \Phi_{\rm aff}\cup \Phi_{\rm aff}^{\rm im}.
\eeq
Using formula \eqref{phase_B} we get 
\beq\label{bE}
b_E(\lambda) = \frac{1}{l^2} \,
\prod_{k=1}^{l-1} (1-\xi^k)^{(\tau^k(E)|E)}\,
Q^{\operatorname{rk}(E)^2}
e^{-2\pi\ii \operatorname{rk}(E) \operatorname{deg}(E)}\, \lambda^{-2},
\quad E\in \Phi_{\rm aff},
\eeq
and 
\ben
b_F(\lambda) = e^{-2\pi\ii 
\operatorname{rk}(F)\, 
\operatorname{deg}(F)} \, Q^{\operatorname{rk}(F)^2} = 
Q^{\operatorname{rk}(F)^2},\quad 
F\in \Phi_{\rm aff}^{\rm im},
\een
where we used that $\operatorname{rk}(F), \operatorname{deg}(F)\in \ZZ$ for all $F\in\Phi_{\rm aff}^{\rm im}$.
\begin{remark}
The coefficients in front of the vertex operators in \eqref{affine_casimir} are defined in terms of the product formula \eqref{product_f}. This principle was first observed in \cite{FGM2010} in the settings of simple singularities. The same formula works in our case too although the relation to representation theory here is not clear yet.  \qed
\end{remark}
Let us define the operator $L(\lambda)$. We split the definition into two parts. Let us write $L(\lambda)=L_0(\lambda)+L_\perp(\lambda)$. The component $L_\perp(\lambda)$ is constructed in a standard way by using some variation of the {\em Sugawara construction}. Namely, let  
$\{\alpha_i\}_{i=1}^{N-2}$ and $\{\beta_i\}_{i=1}^{N-2}$ be two bases of $K_\perp(\PP^1_a,\CC)$ dual with respect to the intersection pairing: $(\alpha_i|\beta_j)=\delta_{i,j}$ for $1\leq i,j\leq N-2$. For example, let us choose $\{\alpha_i\}_{i=1}^{N-2}$ to be the eigen-basis of $\tau$ corresponding via the isomorphism $\Psi$ to 
$\{\phi_{i,p}\}_{1\leq i\leq 3, 1\leq p\leq a_i-1}$. The dual basis is uniquely determined by such a choice. Put
\beq\label{vir-field_1}
L_\perp(\lambda) := 
\frac{1}{2}\, \operatorname{tr}\Big(
\frac{1}{4} + \theta\,\theta^T \Big) \lambda^{-2}+
\frac{1}{2}
\sum_{i=1}^{N-2}
:\Big(
\phi_{\alpha_i}(\lambda)\otimes 1- 1\otimes
\phi_{\alpha_i}(\lambda)
\Big)\, 
\Big(
\phi_{\beta_i}(\lambda)\otimes 1- 1\otimes
\phi_{\beta_i}(\lambda)
\Big):
\eeq
where $\phi_\alpha(\lambda):=\partial_\lambda \log \Gamma^\alpha(\lambda)$, that is, 
\ben
\phi_\alpha(\lambda)=
\sum_{k=0}^\infty 
\sum_{(j,p)\in \mathcal{B}}
\left(
\Big(
I_\alpha^{(-k)}(\lambda),\phi_{j,p}
\Big) \,
\frac{  t_{k,j,p}  }{  \sqrt{\hbar}   } +
\Big(
I_\alpha^{(k+1)}(\lambda),\phi^{j,p} \Big) \,
(-1)^{k+1} \sqrt{\hbar}\frac{\partial}{\partial t_{k,j,p}}\right).
\een
Let us define $L_0(\lambda)$. Let us fix bases $\{\alpha_{N-1},\alpha_N\}$ and $\{\beta_{N-1},\beta_N\}$ of $K_0(\PP^1_a,\CC)$ dual with respect to the Euler pairing: $\langle \alpha_i,\beta_j\rangle =\delta_{i,j}$ for $i,j\in \{N-1,N\}$. Note that the restriction of the Euler pairing to $K_0(\PP^1_a,\CC)$ is a symplectic form and that $\{\delta/l, L-1\}$ is a Darboux basis: 
\ben
\langle \delta/l, L-1\rangle = -\langle L-1,\delta/l\rangle =1.
\een
Therefore, we may choose $\alpha_{N-1}=\delta/l$, $\alpha_N=L-1$ and 
$\beta_{N-1}=L-1$, $\beta_N=-\delta/l$. Put
\ben
L_0(\lambda)= \tfrac{1}{2}
\sum_{i=N-1,N}
\tfrac{1}{2\pi\ii}\, \left.\tfrac{d}{d\nu}
:\Big(
\phi^\nu_{\alpha_i}(\lambda)\otimes 1- 1\otimes 
\phi^\nu_{\alpha_i}(\lambda)\Big)\, 
\Big(
\phi^{-\nu}_{\beta_i}(\lambda)\otimes 1- 1\otimes 
\phi^{-\nu}_{\beta_i}(\lambda)\Big):\,\right|_{\nu=0} ,
\een
where 
\ben
\phi_\alpha^\nu(\lambda):=
\sum_{k=0}^\infty 
\sum_{(j,p)\in \mathcal{B}}
\left(
\Big(
I_\alpha^{(-k+\nu)}(\lambda),\phi_{j,p}
\Big) \,
\frac{  t_{k,j,p}  }{  \sqrt{\hbar}   } +
\Big(
I_\alpha^{(k+1+\nu)}(\lambda),\phi^{j,p} \Big) \,
(-1)^{k+1} \sqrt{\hbar}\frac{\partial}{\partial t_{k,j,p}}\right).
\een
Recalling the definition \eqref{cper} of the calibrated periods we get that for every fixed $\lambda\neq 0$ the coefficients of the differential operator $\phi^\nu_\alpha(\lambda)$ are entire functions in $\nu\in \CC$. In particular, the $\nu$-derivative in the definition of $L_0(\lambda)$ exists. This completes the construction of the bilinear operator \eqref{affine_casimir}.
\begin{remark}
The definition of $L(\lambda)$ is closely related to the free field realization of the Virasoro operators introduced by Dubrovin and Zhang -- see formula (2.37) in \cite{DZ1999}.
\end{remark}

\subsection{Main result}
\begin{definition}
\label{def:aff_HQE}
Let $\mathcal{D}(\hbar,\mathbf{t})$ be a formal power series in the variables 
\ben
t_{k,0,0}\quad (k>0),\quad 
t_{k,0,1} \quad (k\geq 0),
\quad t_{k,j,p}\quad (k\geq 0, 1\leq j\leq 3, 1\leq p\leq a_j-1 )
\een
with coefficients in $\CC[t_{0,0,0}](\!(\hbar)\!)$. We say that 
$\mathcal{D}(\hbar,\mathbf{t})$ satisfies the descendant HQEs of $\PP^1_a$ if for every integer $m\in \ZZ$, the formal power series 
\beq\label{hqe-desc}
\left.
\Omega_{\rm aff}(\lambda) (
\mathcal{D}(\hbar,\mathbf{t}')
\mathcal{D}(\hbar,\mathbf{t}''))
\right|_{t_{0,0,0}'-t_{0,0,0}''=m\sqrt{\hbar}, t_{k,0,0}'=t_{k,0,0}'' (k>0)}
\eeq
is regular in $\lambda$, that is, the coefficients in front of the monomials in $\mathbf{t}'$, $\mathbf{t}''$, and $\hbar$ are polynomials in $\lambda.$  
\end{definition}
The substitution in \eqref{hqe-desc} eliminates the dynamics in the variables $t_{k,0,0}$ ($k>0$), that is, these variables should be viewed as parameters. Our hierarchy can be extended to include differential equations in $t_{k,0,0}$ ($k>0$) too but in this paper, for the sake of simplicity, we decided not to deal with the extended flows. Furthermore, let us point out that it seems that the constant function $\mathcal{D}(\hbar,\mathbf{t})=1$ is a solution to \eqref{hqe-desc}. Although we do not have a conceptual proof of this fact, we made several numerical computations with Mathematica for the case $a=(3,3,3)$.  There are only finitely many identities to be checked and in all cases we got a perfect match. 
After these preliminary remarks, let us state our main result. 
\begin{theorem}\label{thm:HQE_d}
The total descendant potential $\mathcal{D}(\hbar,\mathbf{q})$ of $\PP^1_a$ satisfies the descendant HQEs of $\PP^1_a$.     
\end{theorem}
The main feature of our HQEs \eqref{hqe-desc} is that they involve an infinite sum of vertex operators. We will return to this point again in Section \ref{sec:convergence}. For now, let us outline the main steps in proving that the infinite sums in our HQEs are convergent. The bilinear operator factorizes as follows:
\beq\label{cfac}
\Omega_{\rm aff}(\lambda) =  
\left(
\sum_{n\in \ZZ}
Q^{n^2l^2} \Gamma^{n\delta}(\lambda)\otimes \Gamma^{-n\delta}(\lambda) 
\right)\,
\left(- 
L(\lambda)  + 
\sum_{\alpha\in \Phi_\perp} b_\alpha(\lambda) 
\Gamma^{\alpha}(\lambda)\otimes 
\Gamma^{-\alpha}(\lambda) \right).
\eeq
The first factor on the RHS is obtained by formally substituting $x= \Gamma^{\delta/l}\otimes \Gamma^{-\delta/l}$ in the following function: 
\beq\label{theta}
\theta(Q,x)=\sum_{n\in \ZZ} Q^{n^2l^2} x^{n l},
\eeq
where the various powers of $x= \Gamma^{\delta/l}\otimes \Gamma^{-\delta/l}$ should be computed by using the normally ordered product. 
Note that under the substitutions $Q=e^{\pi\ii\tau/l^2}$ and $x=e^{2\pi \ii z/l}$, the function $\theta(Q,x)$ becomes the Riemann theta function of an elliptic curve (see \cite{Mu}). 
In particular, $\theta(Q,x)$ is an analytic function in two variables in the domain $\{(Q,x)\in \CC\times \CC^*\ |\ |Q|<1\} $. Quite remarkably, the GW invariants of $\PP^1_a$ are convergent for all $|Q|<1$. More precisely, it was proved by Milanov-Ruan in \cite{MiR} (see also \cite{IMRS}) that under the substitution $Q=e^{2\pi\ii\tau/l}$, the Gromov--Witten invariants of $\PP^1_a$ become quasi-modular forms in $\tau\in \mathbb{H}$ for the modular group $\Gamma(l)\subset \operatorname{SL}(2,\ZZ)$ (see \cite{MiSh}).  
Therefore, by requiring that $|Q|<1$, we resolve the divergence issues in both the generating function of GW invariants of $\PP^1_a$ and the HQEs \eqref{hqe-desc}. Note that the polynomial dependence of $\mathcal{D}(\hbar,\mathbf{t})$ on $t_{0,0,0}$ is very important because the vertex operator
\ben
\Gamma^{\delta/l}(\lambda) = 
\exp\Big(- 
(\log Q + 2\pi\ii) \,
\frac{t_{0,0,0}}{\sqrt{\hbar}}-
\sum_{k=1}^\infty 
\frac{\lambda^k}{k!} \, 
((\log Q+2\pi\ii) \, 
\frac{t_{k  ,0,0}}{\sqrt{\hbar}} - 
\frac{t_{k-1,0,1}}{\sqrt{\hbar}}  ) \Big)
\exp\Big(- 
\sqrt{\hbar}\partial_{0,0,0}\Big),
\een
involves a shift in the variable $t_{0,0,0}$. 

\subsection{Structure of the paper}
Section \ref{sec:gis} can be viewed as an extension of the introduction. We introduce the necessary background on gamma integral structures and work out the explicit formula for the bilinear operator $\Omega_{\rm aff}$. Givental's formula for the total descendant potential is recalled in Section \ref{sec:gf}. The formula is introduced in two steps. First, in Section \ref{sec:da}, by using the so-called {\em calibration} operator $S(t,z)$ and its quantization $\widehat{S}(t)$, the total descendant potential is transformed into the so called total {\em ancestor} potential. The later has a geometric interpretation -- it is the potential of the cohomological field theory constructed via the Gromov--Witten theory of $\PP^1_a$. In the second step, thanks to Teleman's theorem (see \cite{Te2012}), the total ancestor potential is transformed to a product of tau-functions of the KdV hierarchy (see Section \ref{sec:hgr}). The material in Sections \ref{sec:gis} and \ref{sec:gf} is well known. The proof of Theorem \ref{thm:HQE_d} starts in Section \ref{sec:vop}. Conjugating by $\widehat{S}(t)$, we transform the HQEs for the total descendant potential to HQEs for the total ancestor potential. This step is more or less the same as in the original paper by Givental (see \cite{Giv2003}). Nevertheless, now we have a better formula for the coefficients of the bilinear operator which was discovered in \cite{FGM2010} -- see formulas \eqref{b_E} and \eqref{b_Et}. Also we found an interpretation of the Virasoro operator $L(\lambda)$ in terms of the Borcherds $(-1)$-product (see \cite{Bo1986}). These two improvements simplify the conjugation to some extent. The core of our proof is in Section \ref{sec:aa_hqe}. 
The coefficients of the vertex operators of the ancestor HQEs are the so-called periods of the Frobenius structure.
Similarly to what was done in \cite{Giv2003}, using the tameness of the total ancestor potential, we prove that the coefficients of the ancestor HQEs can be expressed in terms of polynomials and theta series in the periods. This is done in Sections \ref{sec:apoly} and \ref{sec:ai_hqe}. The next step is to prove that coefficients of the ancestor HQEs are single valued analytic functions with possible poles only along the discriminant of the Frobenius manifold. This is the first place where the techniques from \cite{Giv2003} and \cite{GM2005} can not be used. There are two new ingredients: the first one is the connection formula for the propagator series -- see Proposition \ref{prop:conf} and the second one is the locality (in the sense of vertex algebras \cite{FB2001,Kac1998}) of the vertex operators which leads to the formula in Proposition \ref{prop:ph_period}. The locality is not a deep result but it requires some long computations (see \cite{Mi2014,MilSa}). The proof of the connection formula is based on the so-called Painleve property of a semi-simple Frobenius manifold. It is a rather involved argument making use of deep results from the theory of isomonodromic deformations. The proofs of both propositions are available in the draft of our book \cite{MilSa} written jointly with Kyoji Saito. Using these results, we prove that the ancestor HQEs are equivalent to proving that the coefficients of the ancestor HQEs extend analytically across the discriminant. This is a local computation done in the last section of the paper, i.e., Section \ref{sec:an_kdv}. Again compared to the original work of Givental \cite{Giv2003}, the local computation can be simplified by using the properties of the propagator series studied in Section \ref{sec:aa_hqe}. 

\subsection{Acknowledgments}
I would like to thank Bojko Bakalov for many useful discussions on Kac--Wakimoto hierarchies and lattice vertex algebras. 
I am also thankful to David Klompenhouwer for an useful discussion on the DR hierarchy and to Sergey Shadrin for very useful comments about the generalization of the DZ hierarchy.
This work is supported by the World Premier International Research Center Initiative (WPI Initiative), MEXT, Japan and by JSPS Kakenhi Grant Number JP22K03265.

\section{Gamma integral structure}\label{sec:gis}
The $\Gamma$-integral structure in quantum cohomology was introduced by Iritani in \cite{Iri2009}. We made several statements without proofs in the introduction about the $\Gamma$-integral structure of $\PP^1_a$. We would like to recall the necessary background on gamma integral structures and to prove the statements used in formulating Theorem \ref{thm:HQE_d}. 

\subsection{The integral lattice of Iritani}\label{sec:iis}

Let $IX$ be the {\em inertia orbifold} of $X$, that is, as a groupoid the points of $IX$ are 
\ben
(IX)_0=\{(x,g)\ |\ x\in X_0,\  g\in \operatorname{Aut}(x)\}
\een
while the arrows from $(x',g')$ to $(x'',g'')$ consists of all arrows
$g\in X_1$ from $x'$ to $x''$ such that 
$g''\circ g = g\circ g'$. It is known that $IX$ is an orbifold
consisting of several connected components $X_v$, $v\in
T:=\pi_0(|IX|)$.

We define a linear map 
\ben
\Psi: K^0(X)\to H^*(IX;\CC)=\oplus_{v\in T} H^*(X_v;\CC)
\een
via 
\beq\label{psi-map}
\Psi(E)=(2\pi)^{(1-\operatorname{dim}  X)/2}\ 
\Big(\widehat{\Gamma}(X) e^{-\sum_{i=1}^rP_i\log Q_i}\Big)\cup
(2\pi\ii)^{\operatorname{deg}}
\operatorname{inv}^* \widetilde{\operatorname{ch}}(E),
\eeq
where $\{P_i\}_{i=1}^r\subset H^2(|X|,\ZZ)$ is an ample basis and
$Q_i$ are the corresponding Novikov variables. 
Let us recall the rest of the notation. The linear operator 
\ben
\operatorname{deg}: H^*(IX;\CC)\to H^*(IX;\CC)
\een 
is defined by $\operatorname{deg} (\phi)=i \phi$ if $\phi \in H^{2i}(IX;\CC)$. 
The involution $\operatorname{inv}:IX\to IX$ inverts all arrows while
on the points it acts as $(x,g)\mapsto (x,g^{-1}).$ If $E$ is an
orbifold vector bundle, then we have an eigenbasis decomposition 
\ben
\operatorname{pr}^*(E) = 
\bigsqcup_{v\in T} E_v=
\bigsqcup_{v\in T}\oplus_{0\leq f<1} E_{v,f},
\een
where $\operatorname{pr}:IX\to X$ is the forgetful map $(x,g)\mapsto
x$ and $E_{v,f}$ is the subbundle of
$E_v:=\operatorname{pr}^*(E)|_{X_v}$ whose fiber over a point
$(x,g)\in (IX)_0$ is the eigenspace of $g$ corresponding to the
eigenvalue $e^{2\pi\sqrt{-1} f}$. Let us denote by $\delta_{v,f,i}$
$(1\leq i\leq l_{v,f}:={\rm rk}(E_{v,f}))$ the Chern roots of
$E_{v,f}$, then the Chern character and the $\Gamma$-class of $E$ are
defined by 
\ben
\widetilde{\operatorname{ch}}(E) = 
\sum_{v\in T} \sum_{0\leq f<1} e^{2\pi\ii f}
\sum_{i=1}^{l_{v,f}} e^{\delta_{v,f,i}}
\een
\ben
\widehat{\Gamma}(E) = 
\sum_{v\in T} 
\prod_{0\leq f<1} 
\prod_{i=1}^{l_{v,f}} 
\Gamma(1-f+\delta_{v,f,i}),
\een
where the value of the $\Gamma$-function $\Gamma(1-f+y)$  at
$y=\delta_{v,f,i}$ is obtained by first expanding in Taylor's series
at $y=0$ and then formally substituting $y=\delta_{v,f,i}$. By
definition $\widehat{\Gamma}(X):=\widehat{\Gamma}(TX)$. The cup
product in \eqref{psi-map} is the usual topological cup product in
$|IX|$.   

\subsection{The Euler pairing}
Let us recall the settings of Section \ref{sec:eol}. Using formulas \eqref{euf} and \eqref{euler-pairing}, we would like to find an explicit formula for the Euler pairing of $\PP^1_a$. 
\begin{proposition}\label{prop:otb}
The $K$-theoretic class of the orbifold tangent bundle of $\PP^1_a$ is 
\ben
T\PP^1_a=L_1+L_2+L_3-L-1=L_1L_2L_3L^{-1}.
\een
\end{proposition}
\proof
To begin with, let us recall that the orbifold $K$-ring of $\PP^1_a$ is naturally identified with the $G$-equivariant $K$-ring $K^0_G(Z)$ of $Z$. The relation that we want to prove will be interpreted as a relation in $K^0_G(Z)$. 
The orbifold tangent bundle $T\PP^1_a$ can be identified with $[Q/G]$
where $Q$ is the $G$-equivariant bundle on $Z$ defined by the following short exact sequence of $G$-equivariant vector bundles on $Z$:
\ben
\xymatrix{
0\ar[r] &  Z\times \lieg \ar[r] &  TZ \ar[r] & Q\ar[r] & 0,
}
\een
where $\lieg=\CC$ is the Lie algebra of $G$ and the map $Z\times \lieg \to
TZ$ is defined by
\ben
(z,\xi)\mapsto \left.\frac{d}{d\epsilon} \Big(
e^{\epsilon \xi} z\Big)\right|_{\epsilon=0} \in T_zZ. 
\een
The above exact sequence implies that $Q=TZ-1$ in $K_G(Z)$. 
Furthermore, the tangent bundle $TZ$ can be computed from the following short exact sequence: 
\ben
\xymatrix{
0\ar[r] & TZ \ar[r] & L_1\oplus L_2\oplus L_3 \ar[r] & L\ar[r] & 0
}
\een
where the map to $L$ is defined by $(z,s_1,s_2,s_3)\mapsto \sum_{i=1}^3 a_i z_i^{a_i-1} s_i$.
We get $TZ=L_1+L_2+L_3-L$ in $K^0_G(Z)$. 
Finally, we get 
\ben
T\PP^1_a = Q=  TZ-1 = L_1+L_2+L_3-L-1 = L_1L_2L_3L^{-1},
\een
where the last equality follows easily from the relations in the orbifold $K$-ring.
\qed

Let us prove formula \eqref{psi_E}. The functions $\operatorname{rk}(E)$, $\operatorname{deg}(E)$, and $\chi_{j,p}(E)$ introduced in formula \eqref{psi_E} are in fact the components of the orbifold Chern character map, that is, we have 
\ben
\operatorname{inv}^*
\widetilde{\operatorname{ch}}(E) = 
\operatorname{rk}(E)\phi_{0,0}+ 
\operatorname{deg}(E)\phi_{0,1} + \sum_{j=1}^3\sum_{p=1}^{a_j-1}
\chi_{j,p}(E) \phi_{j,p}.
\een
Let us compute the $\Gamma$-class of $\PP^1_a$. 
The component of the $\Gamma$-class corresponding to the twisted sector 
$[Z^g/G]$ where $g=g_{j,p}$ is $\Gamma(1-p/a_j)\phi_{j,p}$. Indeed, we have 
$T\PP^1_a=L_1L_2L_3L^{-1}$ and 
by definition $L_i$ is a trivial bundle with $g=(g_1,g_2,g_3)$ acting on the fiber by multiplication by 
$g_i= e^{2\pi\ii\delta_{i,j} p/a_j}$. Using that $T\PP^1_a=L_1+L_2+L_3-L-1$, we see that the component of the $\Gamma$-class corresponding to the non-twisted sector must be
\ben
\frac{\Gamma(1+P/a_1)\Gamma(1+P/a_2)\Gamma(1+P/a_3)}{\Gamma(1+P)} =
1+\Gamma'(1) \chi \,P= 1,
\een
where we used that $\chi:=\tfrac{1}{a_1}+\tfrac{1}{a_2}+\tfrac{1}{a_3}-1=0$. 
In other words, 
\ben
\widehat{\Gamma}(\PP^1_a)= 
1+ \sum_{j=1}^3\sum_{p=1}^{a_j-1} 
\Gamma(1-p/a_j)\phi_{j,p}.
\een 
On the other hand, we have
\ben
(2\pi\ii)^{\operatorname{deg}}\operatorname{inv}^*\widetilde{\operatorname{ch}}(E)=
\operatorname{rk}(E)\phi_{0,0}+ 
\operatorname{deg}(E)\,(2\pi\ii)\, P + 
\sum_{j=1}^3\sum_{p=1}^{a_j-1} \chi_{j,p}(E) \phi_{j,p}.
\een
Let us recall the definition \eqref{psi-map} of $\Psi$. Using the above two formulas and $Q^{-P}=1-P\, \log Q$, we get formula \eqref{psi_E}. 
Similarly, we compute that 
$e^{\pi\ii\theta}\Psi(F)$ coincides with 
\ben
\ii \operatorname{rk}(F)\, (1 +\log Q\, P) +
2\pi\,\operatorname{deg}(F) \, P +
\sum_{j=1}^3\sum_{p=1}^{a_j-1} 
\Gamma(1-p/a_j)\, \chi_{j,p}(F)\, e^{-\pi\ii \,p/a_j}\, \ii \phi_{j,p}.
\een
Recalling formulas \eqref{euf}, \eqref{euler-pairing}, and \eqref{psi_E}, we get 
\begin{align*}
\langle E,F\rangle  =
&
\operatorname{rk}(E)\operatorname{deg}(F) - 
\operatorname{rk}(F)\operatorname{deg}(E) + \\
&
+
\sum_{j=1}^3\sum_{p=1}^{a_j-1} 
\Gamma(p/a_j) \, 
\Gamma(1-p/a_j)\, 
\chi_{j,-p}(E)
\chi_{j,p}(F)\, e^{-\pi\ii \,p/a_j}\, \frac{\ii}{2\pi a_j} .
\end{align*}
We have
\ben
\Gamma(p/a_j) \, \Gamma(1-p/a_j)\, 
e^{-\pi\ii \,p/a_j}\, \frac{\ii}{2\pi a_j}=
\frac{1}{a_j}\, \frac{1}{1-\eta_j^p},
\een
where $\eta_j:=e^{2\pi\ii/a_j}$. The formula for the Euler pairing takes the following form:
\begin{align}
\label{euler-form}
\langle E,F\rangle = 
\operatorname{rk}(E)\operatorname{deg}(F) - 
\operatorname{rk}(F)\operatorname{deg}(E) + 
\sum_{j=1}^3
\frac{1}{a_j}
\sum_{p=1}^{a_j-1} 
\frac{\chi_{j,p}(E^\vee \otimes F)}{1-\eta_j^p}.
\end{align}
In case $E=L_1^{k_1}L_2^{k_2}L_3^{k_3}$ and
$F=L_1^{l_1}L_2^{l_2}L_3^{l_3}$ are line bundles, the above formula
can be simplified.
\begin{lemma}
\label{le:identity}
Let $\eta=e^{2\pi\ii/a}$ be a root of unity. 
\begin{enumerate}
\item[a)]
If $1\leq k\leq a$, then the following identity holds:
\ben
\sum_{p=1}^{a-1}
\frac{\eta^{pk}}{1-x\eta^p} = 
\frac{a x^{a-k}}{1-x^a} - 
\frac{1}{1-x}.
\een
\item[b)]
The following identity holds: 
\ben
\sum_{p=1}^{a-1} 
\frac{\eta^{pk}}{1-\eta^p} =
\operatorname{Mod}[k-1,a] - \frac{a-1}{2},\quad k\in \ZZ,
\een
where $\operatorname{Mod}[k-1,a]\in [0,a-1)$ is the remainder on the division of
$k-1$ by $a$.
\end{enumerate}
\end{lemma}
\proof
a) Expanding the denominator into a geometric series we get 
$\sum_{m=0}^\infty \sum_{p=1}^{a-1}\eta^{p(k+m)} x^m $. Since
\ben
\sum_{p=1}^{a-1} \eta^{p(k+m)} = 
\begin{cases}
    a-1 & \mbox{ if } k+m \equiv 0 (\operatorname{mod} \ a),\\
    -1 & \mbox{ otherwise },
\end{cases}
\een
the coefficient in front of $x^m$ is $-1$ plus an extra contribution of $a$ for $m=a-k+a s$ ($s\geq 0$). The formula that we want to prove follows. 

b) Let us first prove the formula in the special case when $1\leq k\leq a$. This is straightforward because we need only to compute the limit when $x\to 1$ in the formula in part a). We get $k-1-\tfrac{a-1}{2}$. For the general case, put $k=q \, a + k'$ where $1\leq k'\leq a$. Since the LHS of the identity does not change when we add to $k$ a multiple of $a$, we may replace $k$ by $k'$. According to the previous computation, the sum should be $k'-1-\tfrac{a-1}{2}$. It remains only to notice that $k'-1=\operatorname{Mod}[k-1,a]$.\qed

\begin{proposition}\label{prop:ep}
Let $E=L_1^{k_1}L_2^{k_2}L_3^{k_3}$ and
$F=L_1^{l_1}L_2^{l_2}L_3^{l_3}$ be line bundles. Then
\ben
\langle E, F\rangle = 1-
\sum_{j=1}^3 \left\lceil
\frac{k_j-l_j}{a_j}
\right \rceil ,
\een
where $\lceil x\rceil$ denotes the ceiling of the real number $x$.
\end{proposition}
\proof
We have $E^\vee \otimes F = L_1^{l_1-k_1}L_2^{l_2-k_2} L_3^{l_3-k_3}$
and
\ben
\chi_{j,p}(E^\vee \otimes F ) = e^{2\pi \ii p(k_j-l_j)/a_j} =  \eta_j^{p(k_j-l_j)}.
\een
Using Lemma \ref{le:identity}, we get
\ben
\frac{1}{a_j}\sum_{p=1}^{a_j-1}
\frac{\chi_{j,p}(E^\vee\otimes F)}{1-\eta_j^p} =
\frac{1}{a_j}\sum_{p=1}^{a_j-1}
\frac{\eta_j^{p(k_j-l_j)}}{1-\eta_j^p} =
\frac{1}{a_j}
\operatorname{Mod}[k_j-l_j-1, a_j]-\frac{1}{2} + \frac{1}{2 a_j}.
\een
The formula for the Euler pairing becomes
\ben
-2 +
\sum_{j=1}^3 \frac{1+l_j-k_j}{a_j} +
\sum_{j=1}^3 \frac{\operatorname{Mod}[k_j-l_j-1, a_j]}{a_j} =
-2-\sum_{j=1}^3 \left[  \frac{k_j-l_j-1}{a_j} \right]=
1-\sum_{j=1}^3 \left\lceil
\frac{k_j-l_j}{a_j}
\right \rceil,
\een
where we used that
\ben
\frac{n}{a} - \frac{\operatorname{Mod}[n,a]}{a} =
\left[\frac{n}{a}\right],\quad
\left[\frac{n-1}{a}\right] +1 =\left\lceil \frac{n}{a} \right\rceil. \qed
\een
\begin{corollary}
The following formulas hold:
\begin{align*}
  \langle L_i^{-p},L_j^{-q}\rangle & =0,\quad
                                     i\neq j ,\quad
                                     1\leq p\leq a_i-1,\quad
                                     1\leq q\leq a_j-1 \\
  \langle L_i^{-p},L_i^{-q}\rangle & =\delta_{p\geq q}+\delta_{p-q,a_i},\quad
                                     0\leq p,q\leq a_i.
\end{align*}
where $\delta_{p\geq q}$ is $1$ if $p\geq q$ and $0$ otherwise.
\end{corollary}

\medskip
\noindent
The orbifold K-ring $K^0(\PP^1_a)$ has the following integral basis:
\begin{align*}
\gamma_{0,0} & =1,\\
\gamma_{0,1} & = L,\\
\gamma_{j,p} & = L_j^{-p}-L_j^{-p+1},\quad 1\leq p\leq a_j-1. 
\end{align*}
The Euler pairing in this basis takes the following form:
\begin{align*}
  \langle 1, L\rangle & =2 \\
  \langle L,L\rangle =\langle 1,1\rangle & = 1\\
  \langle L,\gamma_{i,1} \rangle & = -1\\
  \langle \gamma_{i,p-1},\gamma_{i,p}\rangle &=-1,\quad 1\leq p\leq a_i-1 
\end{align*}
and all other pairings are $0$. Now it is easy to check that the Euler pairing can be described by the quiver on Figure \ref{fig:ep}. 
\subsection{Classical monodromy operator}
The properties of the intersection pairing, Euler pairing, and the classical monodromy operator \eqref{cm} are intertwined in a very interesting way. Now that we have a good understanding of the Euler pairing of $\PP^1_a$, let us explain the role of $\tau$ and use it to make further conclusions about the intersection pairing. 
\begin{proposition}\label{prop:rad}
The intersection form $(\ |\ )$ has the following properties.
\begin{enumerate}
\item[a)]
$(E|F)=\langle (1-\tau)E,F\rangle$. 
\item[b)]
The radical of the intersection form 
\ben
\operatorname{rad}(\ |\ ):= \{ E\in K^0(\PP^1_a)\ |\ (E|F)=0\ \forall F\in K^0(\PP^1_a)\}
\een
is a lattice of rank 2 and $\{L-1,\delta\}$ is a $\ZZ$-basis where $\delta$ is defined by \eqref{delta}.
\end{enumerate}
\end{proposition}
\proof
a)
By Serre duality, we have 
\ben
\operatorname{dim}\, H^i(\PP^1_a,E\otimes F^\vee) =
\operatorname{dim}\, H^{1-i}(\PP^1_a,F\otimes E^\vee \otimes T^\vee),
\een
where $T=T\PP^1_a$ is the tangent bundle and $T^\vee$ is the canonical bundle. We get 
\ben
\langle F,E\rangle = 
\sum_{i=0}^1 (-1)^i 
\operatorname{dim}\, H^i(\PP^1_a,E\otimes F^\vee) =-
\sum_{j=0}^1 (-1)^j
H^{j}(\PP^1_a,F\otimes E^\vee \otimes T^\vee) = -
\langle E\otimes T,F\rangle .
\een
The formula in part a) follows easily from the definition $(E|F)=\langle E, F\rangle + \langle F, E\rangle$. 

b) Since the Euler pairing is non-degenerate, 
the radical of $(\ |\ )$ coincides with the kernel of the operator $1-\tau$, that is, the operator of multiplication by 
\ben
1-T\PP^1_a = 1-L_1L_2L_3 L^{-1} =  1-L_1+1-L_2+1-L_3+1-L^{-1}.
\een
Recalling the relations in the K-ring $K^0(\PP^1_a)$, it is straightforward to check that $L-1$ and $\delta$ belong to the kernel of $1-\tau$, that is, the radical of the intersection pairing contains the lattice $\ZZ\, (L-1) + \ZZ\, \delta$. 

Suppose that 
\ben
E=c_{0,0} + c_{0,1}(L-1) +\sum_{j=1}^3\sum_{p=1}^{a_j-1} c_{j,p}\gamma_{j,p}
\een
is a fixed point for $\tau$. Note that $\tau(\gamma_{j,p})=\gamma_{j,p-1}$ for all $1\leq p\leq a_j-1$ where $\gamma_{j,0} := 1-L_j$. Using $T\PP^1_a=1-(L-1)+ \sum_{j=1}^3 (L_j-1)$ and 
\ben
L_j-1 = (L_j-1)L^{-1}=L_j^{-a_j+1} - L^{-1} = L_j^{-a_j+1}-1 + L-1,
\een
we find 
\ben
\tau(E)= c_{0,0} + (c_{0,1}+ 2c_{0,0}-c_{1,1}-c_{2,1}-c_{3,1}) \, (L-1) + 
\sum_{j=1}^3\sum_{p=1}^{a_j-1}
(c_{0,0}-c_{j,1}+c_{j,p+1})\,\gamma_{j,p},
\een
where $c_{j,a_j}:=0$.
Comparing the coefficients in front of $\gamma_{j,p}$ in $\tau(E)=E$, we get 
\ben
c_{j,p+1}-c_{j,p}= c_{j,1}-c_{0,0},\quad \forall 1\leq p\leq a_j-1,
\een
Put $\Delta_j:=c_{j,1}-c_{0,0}$. Then $c_{j,p}=c_{0,0}+p\, \Delta_j$ 
for $1\leq p\leq a_j-1$. On the other hand, the above equality for $p=a_j-1$ yields $c_{j,a_j-1}= -\Delta_j$. We get $-\Delta_j=c_{0,0}+(a_j-1) \Delta_j$, that is, $c_{0,0}=-a_j\Delta_j$. In particular, $a_j$ divides $c_{0,0}$ for all $j$. Hence, $l$ divides $c_{0,0}$. 
Comparing the coefficients in front of $L-1$, we get 
\ben
2c_{0,0}=c_{1,1}+c_{2,1}+c_{3,1}= 3c_{0,0} + \Delta_1+\Delta_2+\Delta_3=
3c_{0,0} - c_{0,0} \Big(\frac{1}{a_1}+\frac{1}{a_2}+\frac{1}{a_3}\Big). 
\een
The above identity is automatically satisfied because 
$\tfrac{1}{a_1}+\tfrac{1}{a_2} +\tfrac{1}{a_3}=1$. 
We get 
\ben
E=c_{0,0} +c_{0,1} (L-1) + 
\sum_{j=1}^3 \sum_{p=1}^{a_j-1} 
c_{0,0} (1-p/a_j) \gamma_{j,p} = c_{0,1}\, (L-1) + (c_{0,0}/l) \, \delta.
\qed
\een

\medskip
\noindent
Let us recall the spectral decomposition of the classical monodromy operator $\tau$:
\ben
K(\PP^1_a, \CC) =
\bigoplus_{j=0}^{l-1} K_{j/l}(\PP^1_a,\CC), 
\een
where $K(\PP^1_a,\CC):=K^{(0)}(\PP^1_a)\otimes \CC$ and 
\ben
K_{j/l}(\PP^1_a,\CC):=
\{ x\in K(\PP^1_a,\CC) \ |\ \tau(x)=e^{2\pi\ii j/l}x\}
\een
are the eigen-subspaces of $\tau$ (some of them could be trivial). 
\begin{proposition}\label{prop:sd}
The spectral decomposition of $\tau$ has the following properties.
\begin{enumerate}
\item[a)] 
The eigen-subspace $K_0(\PP^1_a,\CC)$ coincides with the radical of $(\ |\ )$, that is, 
\ben
K_0(\PP^1_a,\CC)=\{x\in K(\PP^1_a,\CC)\ |\ (x|y)=0\ \forall y\in K(\PP^1_a,\CC)\}.
\een
\item[b)]
The restriction of the intersection form $(\ |\ )$ to the subspace
\ben
K_\perp(\PP^1_a,\CC):=\oplus_{j=1}^{l-1} K_{j/l}(\PP^1_a,\CC)
\een
is non-degenerate.
\item[c)]
The following formula holds:
$\Psi \circ \tau \circ \Psi^{-1} = 
e^{2\pi\ii \left(\theta-\tfrac{1}{2} \right)}$
where $\Psi$ is the Iritani map \eqref{psi_E}. In particular, $\{\Psi^{-1}(\phi_{i,p})\}_{(i,p)\in \mathcal{B}}$ is an eigenbasis of $\tau.$
\end{enumerate}
\end{proposition}
\proof
a) Let us recall part a) of Proposition \ref{prop:rad}: $(E|F)=\langle (1-\tau)E|F\rangle$. Since the Euler pairing is non-degenerate, the claim in part a) is obvious.

b) This part is clearly a corollary of a) because $K_\perp(\PP^1_a,\CC) \cap K_0(\PP^1_a,\CC)=\{0\}$. 

c) Let us recall formula \eqref{psi_E}. 
Since $T\PP^1_a=L_1L_2L_3L^{-1}$, we have 
\begin{align*}
\operatorname{rk}(E\otimes T\PP^1_a) & =\operatorname{rk}(E),\\ 
\operatorname{deg}(E\otimes T\PP^1_a) & =\operatorname{deg}(E),\\ 
\chi_{j,p}(E\otimes T\PP^1_a) & = \chi_{j,p}(E) e^{-2\pi\ii p /a_j}.
\end{align*}
Using the above formulas, we get 
\ben
\Psi (\tau (E)) = e^{2\pi\ii \left(
\theta-\tfrac{1}{2}\right)} \Psi(E).
\een
This is equivalent to what we have to prove.
\qed

\subsection{Elliptic root system}
Let us recall the set $\Phi:=\{E\in K^0(\PP^1_a)\ |\ (E|E)=2\}$ of reflection vectors. We would like to prove that $\Phi$ is an elliptic root system in the sense of Kyoji Saito and to prove formulas \eqref{ell-s} and \eqref{aff-s}. 
\begin{proposition}
\label{prop:slice}
Let $E\in K^0(\PP^1_a)$ and 
\beq\label{c-coeff}
E=c_{0,0}+c_{0,1}\, (L-1)+\sum_{j=1}^3\sum_{p=1}^{a_j-1} c_{j,p} \gamma_{j,p},
\eeq
where $\gamma_{j,p}=L_j^{-p}-L_j^{-p+1}$. The following formulas hold:
\ben
\langle E, L^{-1} \rangle = - c_{0,1},\quad 
\langle E, L_i^{-a_i+1} \rangle = c_{i,a_i-1}-c_{0,1}\quad (1\leq i\leq 3).
\een
\end{proposition}
\proof
Note that $\langle E, L^{-1} \rangle =\langle E\, L,1 \rangle $. On the other hand, since $(L-1)L=L-1$ and $\gamma_{j,p} L=\gamma_{j,p}$, we have 
\ben
E\, L= c_{0,0}\, L+c_{0,1}\, (L-1)+\sum_{j=1}^3\sum_{p=1}^{a_j-1} c_{j,p} \gamma_{j,p}.
\een
Recalling the diagram on Figure \ref{fig:ep}, we find $\langle L, 1\rangle = \langle\gamma_{j,p},1\rangle =0$. Since $\langle 1,1\rangle =1$, we get $\langle E\, L,1\rangle = -c_{0,1}$. This proves the first formula. For the second one, we argue in a similar fashion. First, 
$\langle E, L_i^{-a_i+1} \rangle =\langle E\, L\, L_i^{-1},1 \rangle $. Since $(L-1)(L_i^{-1}-1)=0$, $\gamma_{j,p} (L_i^{-1}-1)=0$ for $j\neq i$, and
$\gamma_{i,p}\, (L_i^{-1}-1)=\gamma_{i,p+1}-\gamma_{i,p}$ for $1\leq p\leq a_i-2$, we get
\ben
E\, L\, (L_i^{-1}-1)=
c_{0,0} \,\gamma_{i,1} +  \sum_{p=1}^{a_i-2} c_{i,p} (\gamma_{i,p+1}-\gamma_{i,p})+
c_{i,a_i-1} \, \gamma_{i,a_i-1}\, (L_i^{-1}-1).
\een
Recalling again Figure \ref{fig:ep}, since $\langle \gamma_{j,p},1\rangle=0 $, we get 
\ben
\langle E\, L\, (L_i^{-1}-1), 1\rangle = 
c_{i,a_i-1}
\langle \gamma_{i,a_i-1}\, (L_i^{-1}-1),1\rangle.
\een
On the other hand, we have 
\ben
\gamma_{i,a_i-1}\, (L_i^{-1}-1)=  
L^{-1} - L_i^{-a_i+1}-\gamma_{i,a_i-1}=
1-L - (\gamma_{i,1}+\cdots +\gamma_{i,a_i-1}+2\gamma_{i,a_i-1}),
\een
where we used that $L^{-1}=2-L$ and $L_i^{-p}=1+\gamma_{i,1}+\cdots+\gamma_{i,p}$.
Recalling again Figure \ref{fig:ep}, since $\langle 1 - L, 1\rangle = 1-0=1$ and $\langle \gamma_{j,p},1\rangle=0$, we get 
\ben
\langle \gamma_{i,a_i-1}\, (L_i^{-1}-1),1\rangle= 1.
\een
Substituting in the above formula, we get $\langle E\, L\, (L_i^{-1}-1), 1\rangle =c_{i,a_i-1}$ and  
\ben
\langle E, L_i^{-a_i+1}\rangle = 
\langle E\, L\, L_i^{-1},1\rangle = 
\langle E\, L\, (L_i^{-1}-1),1\rangle + 
\langle E\, L,1\rangle = c_{i,a_i-1} -c_{0,1}.
\qed
\een

\medskip
Let us introduce the vector space $H_{N-2}^{(1,1)}:=K^0(\PP^1_a)\otimes \RR$, that is, the topological orbifold $K$-ring with real coefficients. Let us define the following two real vector subspaces of $H_{N-2}^{(1,1)}$:
\begin{align*}
H_{N-2}^{(1)} := & \{ E \in H_{N-2}^{(1,1)}\ |\ \langle E , L^{-1}\rangle =0\} ,\\
H_{N-2} := & \{ E \in H_{N-2}^{(1,1)}\ |\ 
\langle E , L_3^{-a_3}\rangle = 
\langle E , L_3^{-a_3+1}\rangle  = 0\}.
\end{align*}
Since $L_3^{a_3}=L$ and the Euler pairing is non-degenerate, we have that $H_{N-2}$ is a hyperplane in $H_{N-2}^{(1)}$ and $H_{N-2}^{(1)}$ is a hyperplane in $H_{N-2}^{(1,1)}$. Recalling Proposition \ref{prop:slice}, we get that in terms of the decomposition \eqref{c-coeff}, the subspace $H_{N-2}^{(1)}$ is cut out by the equation $c_{0,1}=0$ while $H_{N-2}$ -- by $c_{0,1}=c_{3,a_3-1}=0$. 

Slightly abusing the notation we will continue to denote by $(\ |\ )$ the restriction of the intersection pairing to $H_{N-1}^{(1)}$ and $H_{N-2}$. Let us recall that for a real vector space $F$ and a quadratic form $I$ on $F$, Kyoji Saito introduced the concept of a {\em root system belonging to $I$} -- see Section 1, Definition 1 in \cite{Sa} for more details. If the quadratic form is positive definite, then Saito's definition is equivalent to the classical definition of a root system. If the quadratic form is semi-positive definite and the radical has dimension 1, then the root system is equivalent to an affine root system. Finally, if the quadratic form is semi-positive definite and its radical has dimension 2, then the root system is called {\em elliptic root system}. 

\begin{proposition}
\label{prop:ers-slicing}
Let $H_{N-2}\subset H_{N-2}^{(1)}\subset H_{N-2}^{(1,1)}$ be the vector spaces introduced above and equipped with quadratic forms induced by the intersection pairing $(\ |\ )$. 
\begin{enumerate}
\item[a)] 
The subset $\Phi_{\rm aff}:=\Phi\cap H_{N-2}^{(1)}$ is an affine root system in $H_{N-2}^{(1)}$ belonging to the quadratic form $(\ |\ )$. Moreover, the subset $\{1, \gamma_{i,p}\ (1\leq i\leq 3, 1\leq p\leq a_i-1)\}$ is a set of simple roots for which the corresponding Dynkin diagram coincides with the affine Dynkin diagram of type $E_{N-2}^{(1)}$.
\item[b)] 
The subset $\Phi_\perp:=\Phi\cap H_{N-2}$ is a finite root system in $H_{N-2}$ belonging to the quadratic form $(\ |\ )$. Moreover, the subset 
$\{1,\, \gamma_{i,p}\ (1\leq i\leq 2, 1\leq p\leq a_i-1),\, \gamma_{3,p} (1\leq p\leq a_3-2)\}$ 
is a set of simple roots for which the corresponding Dynkin diagram coincides with the Dynkin diagram of type $E_{N-2}$.
\item[c)]
The set $\Phi$ splits in the following way
\begin{align*}
\Phi= & \{E+m \, (L-1) + n\, \delta\ |\ E\in \Phi_\perp,\ (m,n)\in \ZZ^2\}\cong 
\Phi_\perp\times \ZZ^2, \\
\Phi= & \{E+m \, (L-1)\ |\ E\in \Phi_{\rm aff},\ m\in \ZZ\}\cong \Phi_{\rm aff}\times \ZZ,
\end{align*}
where $\delta$ is defined by \eqref{delta} and the isomorphisms $\Phi_\perp\times \ZZ^2\cong \Phi$ and $\Phi_{\rm aff}\times \ZZ\cong \Phi$ are given respectively by 
$(E,m,n)\mapsto E+m \, (L-1) + n\, \delta$ and 
$(E,m)\mapsto E+m \, (L-1)$.
\end{enumerate}
\end{proposition}
\proof
Parts a) and b) are straightforward corollaries of Proposition \ref{prop:slice}. Indeed, for part a), let us decompose $E$ as in formula \eqref{c-coeff}. Then $E\in H_{N-2}^{(1)}$ iff $c_{0,1}=0$, that is, $E$ is an integer linear combination of the vectors in the set $\mathcal{P}_{\rm aff}:=\{1, \gamma_{i,p}\ (1\leq i\leq 3, 1\leq p\leq a_i-1)\}$. The Euler pairings $\langle a, b\rangle $ for  $a,b\in \mathcal{P}_{\rm aff}$ can be identified with the Euler form of the quiver $Q$ obtained from the quiver on Figure \ref{fig:ep} by chopping the vertex $L$. Clearly, $Q$ is a quiver of the affine Dynkin diagram of type $E_{N-2}^{(1)}$. Since $E\in \Phi_{\rm aff}$ iff and only if $(E|E)=2$, we get that if we identify the set $\mathcal{P}_{\rm aff}$ with the set of simple roots for the affine root system of type $E_{N-2}^{(1)}$, then $E\in \Phi_{\rm aff}$ iff $E$ is an affine (real) root. The argument for part b) is similar. The condition $E\in H_{N-2}$ is equivalent to $c_{0,1}=c_{3,a_3-1}=0$, that is, we have to chop off from the quiver on Figure \ref{fig:ep} the vertices $L$ and $\gamma_{3,a_3-1}$. The resulting quiver is a Dynkin quiver of type $E_{N-2}$. Therefore, we can proceed in the same way as in a). 

For part c), let us prove the first isomorphism, that is,
\beq\label{2-splitting}
\Phi_\perp\times \ZZ^2\cong \Phi,\quad 
(E,m,n)\mapsto E+m \, (L-1) + n\, \delta.
\eeq
The argument for the other case is similar. First, note that the map lands in $\Phi$ because $L-1$ and $\delta$ are in the radical of $(\ |\ )$. Furthermore, let us prove that the map is one-to-one. Suppose that
$E'+m' \, (L-1) + n'\, \delta = E''+m'' \, (L-1) + n''\, \delta$. Then 
\ben
(m''-m')\, (L-1) + (n''-n')\, \delta = E'-E''\quad \in H_{N-2}.
\een
Therefore, the Euler pairing of $(m''-m')\, (L-1) + (n''-n')\, \delta $ with $L^{-1}$ and $L_3^{-a_3+1}$ must be $0$. On the other hand, according to Proposition \ref{prop:slice}, we have 
\beq\label{div_ep}
\langle L-1, L^{-1}\rangle = -1,\quad 
\langle \delta, L^{-1}\rangle = 0,
\eeq
and 
\beq\label{delta_ep}
\langle L-1, L_3^{-a_3+1}\rangle = -1,\quad 
\langle \delta , L_3^{-a_3+1}\rangle = 1,
\eeq
where we used that in the decomposition \eqref{c-coeff} of $E=\delta$, the coefficient $c_{0,1}=0$ and $c_{3,a_3-1}= l (1-(a_3-1)/a_3) = 1$. We get $-(m''-m')=0$ and $(n''-n')=0$, that is, $m''=m'$ and $n''=n'$. 

It remains to prove that the map is onto. To this end, let $E\in \Phi$ be arbitrary. Let us decompose $E$ as in \eqref{c-coeff}. Then combining Proposition \ref{prop:slice} and formulas \eqref{div_ep} and \eqref{delta_ep}, we get that $E':=E-m(L-1)-n\delta \in H_{N-2}$ iff
\begin{align*}
-c_{0,1}+m = & 0, \\
c_{3,a_3-1}-c_{0,1}+m -n =& 0.
\end{align*}
The above system has a unique integral solution $m=c_{0,1}$ and $n=c_{3,a_3-1}$. Clearly, the map \eqref{2-splitting} maps $(E',m,n)$ to $E$. 
\qed

\medskip

\begin{remark}
Recalling Proposition \ref{prop:rad}, we get $\operatorname{rad}\cap H_{N-2}^{(1)}=\ZZ\,\delta$. In the terminology of Kyoji Saito (see \cite{Sa}), the pair $(\Phi,\ZZ\,\delta)$ is an example of a {\em marked extended affine root system}. \qed
\end{remark}
\begin{remark}
Under mirror symmetry, the K-theoretic lattice $K^0(\PP^1_a)$ can be identified with the Milnor lattice of the so-called {\em cusp polynomials} $f(x)=x_1^{a_1}+x_2^{a_2}+x_3^{a_3}-\tfrac{1}{Q} x_1 x_2 x_3$. If $\tfrac{1}{a_1}+\tfrac{1}{a_2}+\tfrac{1}{a_3}=1$, then the cusp polynomial defines a simple elliptic singularity and the set $\Phi$ can be identified with the set of {\em vanishing cycles}.  The elliptic root systems corresponding to simple elliptic singularities are the main motivation of Kyoji Saito for introducing the theory of elliptic root systems. \qed
\end{remark}

\subsection{Phase factors}\label{sec:phf}
Recalling the explicit formula for the Iritani map \eqref{psi_E} and the definition of the calibrated periods \eqref{cper}, we get the following explicit formula for the calibrated periods:
\begin{proposition}
\label{prop:cp}
Let $E\in K^0(\PP^1_a)$ be a K-theoretic
vector bundle and $m\in \CC$. Then 
\begin{align*}
  I^{(-m)}_E(\lambda) = & \
\operatorname{rk}(E)
\frac{\lambda^m}{\Gamma(m+1)}+
\frac{\lambda^{m-1}}{\Gamma(m)}\, \Big(
2\pi\ii\, \operatorname{deg}(E)-
\operatorname{rk}(E)\, \log Q 
\Big)\,P + \\
&
+
    \sum_{j=1}^3 \sum_{p=1}^{a_j-1}
\chi_{j,p}(E) \,
\frac{\Gamma(1-p/a_j)}{\Gamma(m+1-p/a_j)}\, \lambda^{m-p/a_j}\, 
\phi_{j,p}.
\end{align*}
\end{proposition}
Using the formula in Proposition \ref{prop:cp}, it is straightforward to obtain the explicit formulas for the vertex operators stated in the introduction.

We would like to prove formula \eqref{phase_B}. The factor 
$B^{E,F}(\lambda,\mu)=B_0^{E,F}(\lambda,\mu)\, B_\perp^{E,F}(\lambda,\mu)$ where  the factors $B_0^{E,F}(\lambda,\mu)$ and $B_\perp^{E,F}(\lambda,\mu)$ are defined by the following formulas:
\begin{align*}
\Gamma_0^E(\lambda)\Gamma_0^F(\mu) & = 
B_0^{E,F}(\lambda,\mu)\, 
:\Gamma_0^E(\lambda)\Gamma_0^F(\mu): \, \\
\Gamma_\perp^E(\lambda)\Gamma_\perp^F(\mu) & = 
B_\perp^{E,F}(\lambda,\mu)\, 
:\Gamma_\perp^E(\lambda)\Gamma_\perp^F(\mu):\ ,
\end{align*}
where $\Gamma_0^A=\Gamma_{0,0}^A \Gamma_{0,1}^A$ and 
$\Gamma_\perp^A=\prod_{j=1}^3 \prod_{p=1}^{a_j-1} \Gamma_{j,p}^A$. 
The $B_0$-factor is straightforward to compute. We have
\beq\label{B_0}
B_0^{E,F}(\lambda,\mu)=
e^{
-2\pi\ii\, 
\operatorname{rk}(E)
\operatorname{deg}(F) }\, 
Q^{\operatorname{rk}(E)\operatorname{rk}(F)}.
\eeq
Let us compute $B_\perp^{E,F}(\lambda,\mu)$. 
\begin{proposition}\label{prop:B_perp}
The following formula holds:
\ben
B_\perp^{E,F}(\lambda,\mu)=
\prod_{k=1}^l
\Big(
1-\xi^k (\mu/\lambda)^{1/l} 
\Big)^{(\tau^k(E)|F)},
\een
where $\xi:=e^{2\pi\ii/l}$.
\end{proposition}
\proof
We have 
$
\Gamma_{j,p}^E(\lambda) \Gamma_{j,p}^F(\mu) =
B_{j,p}(\lambda,\mu) 
: \Gamma_{j,p}^E(\lambda) \Gamma_{j,p}^F(\mu):
$, 
where 
\ben
B_{j,p}(\lambda,\mu)=
\exp\Big(
-\sum_{k=0}^\infty 
\frac{(x^{1/a_j})^{p+k a_j}}{p+k a_j}\, 
\chi_{j,-p}(E^\vee\otimes F)
\Big),
\een
where $x:=\mu/\lambda$. Taking the product over all $p$ ($1\leq p\leq a_j-1$), we get 
\ben
B_j(\lambda,\mu):=
\prod_{p=1}^{a_j-1} B_{j,p}(\lambda,\mu)=
\exp\Big(-
\sum_{n=1}^\infty \frac{(x^{1/a_j})^n}{n}\, 
\chi_{j,-n}(E^\vee\otimes F) + 
\frac{\operatorname{rk}(E)\operatorname{rk}(F)}{a_j} 
\sum_{n=1}^\infty \frac{x^n}{n} 
\Big).
\een
Let us decompose 
\ben
E^\vee\otimes F=
m_{0,0} +m_{0,1} L^{-1} + 
\sum_{j=1}^3 \sum_{p=1}^{a_j-1} m_{j,p} L_j^{-p},
\een
where the coefficients $m_{i,q}\in \ZZ$. Note that 
\ben
\chi_{j,-n}(E^\vee\otimes F) = 
m_{j,0}+ 
\sum_{p=1}^{a_j-1} m_{j,p} \eta_j^{-np} ,
\een
where $\eta_j=e^{2\pi\ii/a_j}$ and $m_{j,0}$ is the sum of all 
$m_{i,q}$ such that $i\neq j$, that is,  
\ben
m_{j,0}= m_{0,0}+m_{0,1}+
\sum_{i:i\neq j}
\sum_{q=1}^{a_i-1} m_{i,q}=
\operatorname{rk}(E)\operatorname{rk}(F) -
\sum_{p=1}^{a_j-1} m_{j,p}.
\een
The above formula takes the following form:
\ben
B_j(\lambda,\mu) = 
(1-x)^{-\operatorname{rk}(E)\operatorname{rk}(F)/a_j} \, 
\prod_{p=0}^{a_j-1} 
\Big(
1-\eta_j^{-p} x^{1/a_j}
\Big)^{m_{j,p}}.
\een
We get the following formula for the orthogonal component of the phase factor 
\ben
B_\perp^{E,F}(\lambda,\mu)=
(1-x)^{-\operatorname{rk}(E)\operatorname{rk}(F)} \, 
\prod_{j=1}^3\prod_{p=0}^{a_j-1} 
\Big(
1-\eta_j^{-p} x^{1/a_j}
\Big)^{m_{j,p}}.
\een
Let us rewrite each factor in the above product in the following way:
\ben
1-\eta_j^{-p} x^{1/a_j} = 
1-\xi^{-pl/a_j} (x^{1/l})^{l/a_j} =
\prod_{s=0}^{\tfrac{l}{a_j}-1} (1-\xi_j^{-s} \xi^{-p} x^{1/l}),
\een
where $\xi_j=\xi^{a_j}$ is a primitive $\tfrac{l}{a_j}$-root of $1$. Since $\xi_j^{-s}\xi^{-p} =\xi^{-s a_j -p}$, we can identify the above product with the product of all $1-\xi^{-k} x^{1/l}$ such that $1\leq k\leq l$ and the remainder of the divison of $k$ on $a_j$ is $p$. It is convenient to extend the definition of $m_{j,p}$ by defining $m_{j,k}:=m_{j,p}$ where $k\in \ZZ$ and $p$ is the remainder of the division of $k$ on $a_j$. We get 
\ben
\prod_{p=0}^{a_j-1} 
\Big(
1-\eta_j^{-p} x^{1/a_j}
\Big)^{m_{j,p}}=
\prod_{k=1}^{l} 
\Big(
1-\xi^{-k} x^{1/l}
\Big)^{m_{j,k}}
\een
and 
\beq\label{Bm}
B_\perp^{E,F}(\lambda,\mu) =
(1-x)^{-\operatorname{rk}(E)\operatorname{rk}(F)} \, 
\prod_{k=1}^{l} 
\Big(
1-\xi^{-k} x^{1/l}
\Big)^{m_{1,k}+m_{2,k}+m_{3,k}}.
\eeq
It remains only to compute the sum $m_{1,k}+m_{2,k}+m_{3,k}$. To this
end, let us examine the pairing
\ben
(\tau^k(E)|F)=(\tau^k(E^\vee\otimes F)|1)=
m_{0,0} (\tau^k|1)+
m_{0,1}(\tau^k|L^{-1}) + 
\sum_{j=1}^3 \sum_{p=1}^{a_j-1}
m_{j,p} (\tau^k|L_j^{-p}),
\een
where slightly abusing the notation we denoted by $\tau$ the orbifold
tangent bundle, that is, $\tau=L_1L_2L_3L^{-1}$.  Recalling Proposition
\ref{prop:ep}, we get
\ben
(\tau^k|L_j^{-p}) =
(L_1^kL_2^kL_3^k L_j^{p}|1) =2-
\left\lceil\frac{k+p}{a_j} \right\rceil-
\left\lceil\frac{-k-p}{a_j} \right\rceil-
\sum_{i:i\neq j}
\Big(
\left\lceil\frac{k}{a_i} \right\rceil+
\left\lceil\frac{-k}{a_i} \right\rceil
\Big),
\een
where for the first equality we set $L=1$ because $L-1$ is in the radical of $(\ |\ )$. Note that
\ben
\lceil x\rceil + \lceil -x\rceil =
\begin{cases}
  0 & \mbox{ if } x\in \ZZ,\\
  1& \mbox{ if } x\notin \ZZ.
\end{cases}
\een
Therefore,
\ben
(\tau^k|L_j^{-p}) = -1+\delta_{a_j|p+k} + \sum_{i:i\neq j}
\delta_{a_i|k},\quad \forall p\in \ZZ,
\een
where we introduced the boolean function $\delta_{a|m}$ which is $1$
if $a|m$ and $0$ otherwise. Using the above formula, we get 
\ben
(\tau^k(E)|F) + \operatorname{rk}(E)\operatorname{rk}(F)=
\operatorname{rk}(E)\operatorname{rk}(F) \, \sum_{i=1}^3 \delta_{a_i|k}  +
\sum_{j=1}^3 \sum_{p=1}^{a_j-1}
m_{j,p} (\delta_{a_j|p+k} -\delta_{a_j|k}),
\een
where we used that
$\operatorname{rk}(E)\operatorname{rk}(F)=
m_{0,0}+m_{0,1}+\sum_{j=1}^3 \sum_{p=1}^{a_j-1} m_{j,p}.$ The RHS of
the above formula can be rewritten as
\ben
\sum_{j=1}^3 \Big(\Big(
\operatorname{rk}(E)\operatorname{rk}(F) - \sum_{p=1}^{a_j-1} m_{j,p}\Big)
\delta_{a_j|k} +
\sum_{p=1}^{a_j-1}
m_{j,p}  \delta_{a_j|p+k} \Big)=
\sum_{j=1}^3 \sum_{p=0}^{a_j-1} m_{j,p} \delta_{a_j|p+k} =
\sum_{j=1}^3 m_{j,-k}.
\een
This proves that
$m_{1,-k}+m_{2,-k}+m_{3,-k} =
(\tau^k(E)|F) +
\operatorname{rk}(E)\operatorname{rk}(F)$. Let us switch the index
$k$ in the product \eqref{Bm} by $k\mapsto -k$. Using the above
formula, we get
\ben
B_\perp^{E,F}(\lambda,\mu) =
(1-x)^{-\operatorname{rk}(E)\operatorname{rk}(F)} \, 
\prod_{k=1}^{l} 
\Big(
1-\xi^{k} x^{1/l}
\Big)^{
(\tau^k(E)|F) +
\operatorname{rk}(E)\operatorname{rk}(F)
}.
\een
To complete the proof, we need only to note
\ben
\prod_{k=1}^l (1-\xi^k
x^{1/l})^{\operatorname{rk}(E)\operatorname{rk}(F)}=
(1-x)^{\operatorname{rk}(E)\operatorname{rk}(F)}.
\qed
\een
Combining formula \eqref{B_0} and Proposition \ref{prop:B_perp} we get \eqref{phase_B}. 

\subsection{Convergence of the bilinear operator}
\label{sec:convergence}
Let us first prove formula \eqref{cfac}. 
\begin{lemma}\label{le:vo-splitting}
The following formula holds:
\ben
\Gamma^{\alpha+n\delta}(\lambda) = 
Q^{-n\,l\, \operatorname{rk}(\alpha) }
\Gamma^{n\delta}(\lambda)
\Gamma^{\alpha}(\lambda)=
Q^{-n\,l\, \operatorname{rk}(\alpha) }
\Gamma^{\alpha}(\lambda)
\Gamma^{n\delta}(\lambda),
\een
where $\alpha\in \Phi_\perp$ and $n\in \ZZ$.
\end{lemma}
\proof
 It is straightforward to check that $\Gamma_{j,p}^{\alpha+ n\delta}(\lambda)= \Gamma_{j,p}^\alpha(\lambda)$. Note that 
\ben
\operatorname{rk}(\delta)=l,\quad 
\operatorname{deg}(\delta)=-l.
\een
Recalling the explicit formulas for the vertex operators $\Gamma_0^E(\lambda)=\Gamma_{0,0}^E(\lambda) \Gamma_{0,1}^E(\lambda)$, we get 
\ben
\Gamma_0^{\alpha + n\delta}(\lambda) = 
e^{2\pi\ii \,n\,l\,\operatorname{deg}(\alpha)}
Q^{-n\,l\,\operatorname{rk(\alpha)}}\, 
\Gamma_0^{n\delta}(\lambda) \,
\Gamma_0^{\alpha }(\lambda) =
Q^{-n\,l\,\operatorname{rk(\alpha)}}\, 
\Gamma_0^{n\delta}(\lambda) \,
\Gamma_0^{\alpha }(\lambda),
\een
where we used that $l\,\operatorname{deg}(\alpha)\in \ZZ$.
\qed

\begin{lemma}\label{le:bE-splitting}
We have: 
$b_{\alpha+n\delta}(\lambda)= Q^{n^2l^2+2nl\operatorname{rk}(\alpha)}\, b_\alpha(\lambda)$ for all
$\alpha\in \Phi$ and $n\in \ZZ$. 
\end{lemma}
\proof
This follows immediately from the explicit formula \eqref{bE}, the fact that $\delta$ is in the radical of the form $(\ |\ )$, and $\tau(\delta)=\delta$. \qed

Combining Lemma \ref{le:vo-splitting} and Lemma \ref{le:bE-splitting}, we get 
\ben
b_{\alpha+n\delta}(\lambda) 
\Gamma^{\alpha+n\delta}(\lambda)\otimes 
\Gamma^{-\alpha-n\delta}(\lambda) = 
Q^{n^2l^2} \Gamma^{n\delta}(\lambda)\otimes \Gamma^{-n\delta}(\lambda) \ 
b_\alpha(\lambda) \Gamma^{\alpha}(\lambda)\otimes 
\Gamma^{-\alpha}(\lambda),
\een
where $\alpha\in \Phi_\perp$ and $n\in \ZZ$. Since the vertex operators $\Gamma^\alpha(\lambda)$ and $\Gamma^{n\delta}(\lambda)$ commute, we get \eqref{cfac}.

Let us examine more carefully the action of the bilinear operator in \eqref{hqe-desc}. As we already observed in the introduction, the second factor on the RHS of \eqref{cfac} can be interpreted as
\beq\label{theta_sum}
\theta(Q,\Gamma^{\delta/l}\otimes \Gamma^{-\delta/l}):= 
\sum_{n\in \ZZ}
Q^{n^2l^2} \Gamma^{n\delta}(\lambda)\otimes \Gamma^{-n\delta}(\lambda),
\eeq
where the operator series \eqref{theta_sum} is obtained by formally substituting $x= \Gamma^{\delta/l}\otimes \Gamma^{-\delta/l}$ in the theta series \eqref{theta} and the various powers of $x$ are computed by using the normally ordered product. The vertex operator $\Gamma^{\delta/l}$ has the following form:
\ben
\Gamma^{\delta/l}(\lambda) = e^{- s(\mathbf{t},\lambda)/\sqrt{\hbar}}\, e^{- \sqrt{\hbar}\partial_{0,0,0}},
\een
where 
\ben
s(\mathbf{t},\lambda) := t \,t_{0,0,0}+\sum_{k=1}^\infty \frac{\lambda^k}{k!} \, (t\, t_{k,0,0} - t_{k-1,0,1}  ) ,
\een
where $t:=\log Q + 2\pi\ii$ is the K\"ahler parameter of $\PP^1_a$. Here we are choosing $Q\in (0,1)$ to be a real number and $\log Q = \ln Q$ is defined via the principal branch of the logarithm.  
The operator series \eqref{theta_sum} takes the form
\ben
\theta(Q,\Gamma^{\delta/l}\otimes \Gamma^{-\delta/l}):=  
\sum_{k=0}^\infty 
\frac{1}{k!}
\sum_{n\in \ZZ} Q^{n^2l^2} 
\exp\Big(
n l \, s(\mathbf{t}''-\mathbf{t}',\lambda)/\sqrt{\hbar}
\Big)\, (nl)^k
\Big( \sqrt{\hbar}\, 
(\partial_{0,0,0}''-\partial_{0,0,0}')\Big)^k.
\een
Put 
\ben
\theta^{(k)}(Q,x) = \frac{1}{k!} (x\partial_x)^k \,\theta(Q,x) = \frac{1}{k!}\sum_{n\in \ZZ} Q^{n^2l^2 } \, x^{n l} \, (nl)^k.
\een
Then we have 
\beq\label{theta_do}
\theta(Q,\Gamma^{\delta/l}\otimes \Gamma^{-\delta/l})=  
\sum_{k=0}^\infty 
\theta^{(k)}(Q, e^{s(\mathbf{t}''-\mathbf{t}',\lambda)/\sqrt{\hbar} } ) \, 
\Big(\sqrt{\hbar}\,
(\partial_{0,0,0}'-\partial_{0,0,0}'')\Big)^k.
\eeq
Note that the theta function derivative $\theta^{(k)}(Q,e^s)$ is an entire function in $s$. Therefore, we can expand it in a Taylor series at $s=0$. Substituting 
$s= s(\mathbf{t}''-\mathbf{t}',\lambda)/\sqrt{\hbar}$ yields an infinite series convergent in the formal topology of the Fock space. Thanks to the polynomiality of $\mathcal{D}(\hbar,\mathbf{t}')\mathcal{D}(\hbar,\mathbf{t}'')$ in $t_{0,0,0}'$ and $t_{0,0,0}''$, the action of \eqref{theta_do} on $\mathcal{D}(\hbar,\mathbf{t}')\mathcal{D}(\hbar,\mathbf{t}'')$ is convergent in the formal topology of the Fock space.

\section{Givental's formula}\label{sec:gf}
Our proof of Theorem \ref{thm:HQE_d} uses the higher genus formula for the GW invariants of $\PP^1_a$. The formula was conjectured by Givental in \cite{Giv2001} and proved by Teleman in \cite{Te2012} for any target orbifold $X$ with generically semi-simple quantum cohomology. 

\subsection{Quantum cohomology}
Let us recall the quantum cup product of $\PP^1_a$. Given two cohomology 
classes $\phi_1,\phi_2\in H$, since the orbifold Poincar\'e pairing is non-degenerate, the quantum cup product $\phi_1\bullet \phi_2$ is uniquely determined by the following definition
\beq\label{qcp}
(\phi_1\bullet \phi_2,\phi_3)=
\sum_{l,d=0}^\infty \frac{Q^d}{l!}
\langle \phi_1,\phi_2,\phi_3,t,\dots,t\rangle_{0,3+l,d},
\eeq
where $\phi_3\in H$ is an arbitrary cohomology class and 
$t=t_{0,0} +t_{0,1}P+ \sum_{i=1}^3 \sum_{p=1}^{a_i-1} t_{i,p} \phi_{i,p}\in H$. Note that the quantum cup product is independent of $t_{0,0}$ thanks to the string equation. In addition, the dependence on $t_{0,1}$ is governed by the {\em divisor equation}, that is, the RHS is annihilated by the differential operator $Q\tfrac{\partial}{\partial Q} - \tfrac{\partial}{\partial t_{0,1}}$. In other words, the RHS is a function in $Qe^{t_{0,1}}$. Apriori the expression in \eqref{qcp} is a formal power series in the variables $Q e^{t_{0,1}}$ and $t_{i,p}$ ($1\leq i\leq 3$, $1\leq p\leq a_i-1$). Thanks to mirror symmetry (see \cite{IMRS}), the series is convergent and it extends analytically for all $(Q,t)$ such that $|Q e^{t_{0,1}}|<1$. 

Let us define a family of Frobenius manifolds $M_Q$ depending on the parameter $Q\in \CC^*$.  As a complex manifold 
\ben
M_Q:=\{t\in H\ |\ |Q e^{t_{0,1}}|<1 \} \quad \subset \quad H.
\een
For the ease of the notation, we will drop the index $Q$ and denote $M_Q$ simply by $M$. 
Using the linear structure of $H$ we can trivialize the holomorphic tangent bundle $TM\cong M\times H$: if we interpret the holomorphic tangent space $T_x M$ as the vector space of derivations $\operatorname{Der}(\O_{M,x},\CC)$, then the derivation corresponding to $v\in H$ is just the directional derivative 
\ben
D_v( f)=\lim_{\epsilon\to 0} \frac{f(x+\epsilon v)-f(x)}{\epsilon}.
\een
Using the isomorphism $TM\cong M\times H$, we get that the quantum cup product defines a holomorphic multiplication on $TM$. The unit of this multiplication is $\phi_{0,0}=1\in H$ and it corresponds to the vector field $e:=\tfrac{\partial}{\partial t_{0,1}}$. Furthermore, the cohomology classes $\phi_{i,p}=\tfrac{\partial}{\partial t_{i,p}}$. The orbifold Poincar\'e pairing $(\ ,\ )$ defines a symmetric flat non-degenerate complex bilinear form on $TM$. Finally, the following vector field 
\ben
E:= t_{0,0} \frac{\partial}{\partial t_{0,0}} + 
\sum_{i=1}^3\sum_{p=1}^{a_i-1} 
\Big(1- \frac{p}{a_i} \Big) \, t_{i,p} \frac{\partial}{\partial t_{i,p}}
\een
is an Euler vector field. It is well known that the data $(M,\bullet, (\ ,\ ), e, E)$ is a Frobenius structure of conformal dimension $D=1$. In other words, the following connection is flat: 
\begin{align}\label{qc_1}
    \nabla_{\partial/\partial t_{i}} & =\frac{\partial}{\partial t_{i}} - z^{-1}\, \phi_{i}\bullet \ 
    \quad (i\in \mathcal{B}),\\
    \label{qc_2}
    \nabla_{\partial/\partial z} & =\frac{\partial}{\partial z} - z^{-1} \theta + z^{-2} E\bullet.
\end{align}
The above formulas define a connection on the vector bundle $TM\times \CC^*\to M\times \CC^*$ known as the {\em quantum connection} or {\em Dubrovin connection}. Under the mirror symmetry isomorphism (see \cite{IMRS}), the above connection can be identified with the Gauss-Manin connection on the twisted de Rham cohomology of a simple elliptic singularity and the Frobenius structure is identified with Kyoji Saito's {\em flat structure}. Note that the Hodge grading operator $\theta$ corresponds to the grading operator of the Frobenius structure:
\ben
\theta(v) = \nabla^{\rm L.C.}_v(E) - \frac{2-D}{2} \, v,\quad v\in H,
\een
where $v$ is viewed as a flat vector field on $M$ and $\nabla^{\rm L.C.}$ is the Levi-Civita connection of the Poincar\'e pairing.

\subsection{From descendants to ancestors}\label{sec:da}
The {\em ancestor} GW invariants of $\PP^1_a$ are defined by the following formula:
\beq\label{AGW}
\langle 
\phi_1\overline{\psi}_1^{k_1},\dots 
\phi_n\overline{\psi}_n^{k_n} 
\rangle_{g,n}(t)=\sum_{l,d} \frac{Q^d}{l!} 
\int_{[\overline{\mathcal{M}}_{g,n+l}(\PP^1_a,d)]^{\rm virt}}
\operatorname{ev}^*(\phi_1\otimes \cdots \otimes \phi_n \otimes t^{\otimes l}) 
\overline{\psi}_1^{k_1}\cdots \overline{\psi}_n^{k_n},
\eeq
where $\overline{\psi}_i$ are the pullbacks of the $\psi_i$ classes on $\overline{\mathcal{M}}_{g,n}$ via the forgetful morphism $\overline{\mathcal{M}}_{g,n+l}(\PP^1_a,d)\to \overline{\mathcal{M}}_{g,n}$, that is, the morphism forgetting the holomorphic map and the last $l$ marked point (and contracting the resulting unstable components). 

The ancestor GW invariants can be expressed in terms of the descendant ones. In order to explain the relation, it is convenient to make the so-called {\em dilaton shift}: $\mathbf{t}(z)=\mathbf{q}(z)+z$ where $\mathbf{q}(z)=\sum_{k=0}^\infty q_k z^k$ and $q_k=\sum_{i\in \mathcal{B}} q_{k,i} \phi_i\in H$ are formal vector variables. Under this substitution the Fock space \eqref{fock} becomes 
\beq\label{dfock}
\CC[q_{0,0,0}](\!(\hbar)\!)[\![
q_{k,i,p}+\delta_{(k,i,p),(1,0,0)}\ :\ (k,i,p)\neq (0,0,0)]\!].
\eeq
The total descendant potential $\mathcal{D}(\hbar,\mathbf{t})$ is identified with an element in the Fock space \eqref{dfock}. Slightly abusing the notation we will denote this element by 
$ \mathcal{D}(\hbar,\mathbf{q})$.  
Let us introduce the so-called {\em total ancestor potential}
\ben
\mathcal{A}_t(\hbar,\mathbf{q}):= \exp\Big(
\sum_{n,g=0}^\infty \frac{\hbar^{g-1}}{n!}\, 
\langle 
\mathbf{q}(\overline{\psi}_1) +\overline{\psi}_1,\dots, 
\mathbf{q}(\overline{\psi}_n) +\overline{\psi}_n
\rangle_{g,n}(t)
\Big).
\een
The total ancestor potential is a formal power series in $t$ and $q_0, q_1+1,q_2,\dots$ whose coefficients are in $\CC(\!(\hbar)\!)$.
Furthermore, let us introduce the {\em calibration operator} 
$S(t,z)=1+S_1(t) z^{-1}+S_2(t) z^{-2}+\cdots $ where $S_k(t)\in \operatorname{End}(H)$ is defined by 
\ben
(S_k(t)\phi_1,\phi_2):= 
\sum_{l,d=0}^\infty \frac{Q^d}{l!}
\langle \phi_1 \psi^{k-1},\phi_2,t,\dots,t\rangle_{0,2+l,d},\quad 
\forall \phi_1,\phi_2\in H. 
\een
It is known that the calibration provides a fundamental solution to the Dubrovin connection in a neighborhood of $z=\infty$. More precisely, according to Givental (see \cite{Giv2001}) $S(t,z) z^{\theta}$ is a solution to the Dubrovin connection. Moreover, the operator series satisfies the following symplectic condition: $S(t,z)S(t,-z)^T=1$ where ${}^T$ denotes transposition with respect to the Poincar\'e pairing. Following Givental, let us introduced a linear operator $\widehat{S}(t)^{-1}$ acting on the Fock space as follows:
\beq\label{S-action:fock}
\widehat{S}(t)^{-1}\, \mathcal{F}(\mathbf{q}):= 
e^{\frac{1}{2\hbar}W(\mathbf{q},\mathbf{q})} \mathcal{F}([S(t,z)\mathbf{q}(z)]_+),
\eeq
where the notation is as follows. We have a sequence of formal variables $\mathbf{q}=(q_{k,i,p})$ where $k\geq 0$ and $(i,p)\in \mathcal{B}$. We identify $\mathbf{q}$ with the formal series $\mathbf{q}(z)=\sum_{k=0}^\infty \sum_{(i,p)\in \mathcal{B}} q_{k,i,p} \phi_{i,p} z^k$. This allows us to define $[S(t,z)\mathbf{q}(z)]_+$ where $[\ ]_+$ is just the operation truncating all terms involving negative powers of $z$. Note that the coefficient in front of $\phi_{a} z^k$ in $S(t,z)\mathbf{q}(z)$ is 
\beq\label{ka-coeff}
q_{k,a} + \sum_{b\in \mathcal{B}}\Big(S_{1,ab}(t) q_{k+1,b} + S_{2,ab}(t) q_{k+2,b}  + \cdots \Big),
\eeq
where $S_{l,ab}(t)$ are the entries of the matrix of the linear operator $S_l(t)$ with respect to the basis $\{\phi_{a}\}_{a\in \mathcal{B}}$, that is, $S_l(t) \phi_{b} =: \sum_{a\in \mathcal{B}} \phi_a\, S_{l,ab}(t) $. 
The formal series $\mathcal{F}([S(t,z)\mathbf{q}(z)]_+)$ is obtained from $\mathcal{F}(\mathbf{q})$ by substituting the series \eqref{ka-coeff} for $q_{k,a}$. The formal series in the exponent is a quadratic form 
\beq\label{W_form_1}
W(\mathbf{q},\mathbf{q})=
\sum_{k,l=0}^\infty (W_{kl}(t)\,q_l,q_k) =
\sum_{k,l=0}^\infty \sum_{a,b\in \mathcal{B}} (W_{kl}(t)\,\phi_{b},\phi_{a}) q_{k,a} q_{l,b}
\eeq
where $W_{kl}(t)\in \operatorname{End}(H)$ are defined by the following identity:
\beq\label{W_form_2}
\sum_{k,l=0}^\infty W_{kl}(t) z^{-k} w^{-l} =
\frac{S^T(t,z) S(t,w)-1}{z^{-1}+w^{-1}}.
\eeq
The coefficients of the quadratic form \eqref{W_form_1} can be expressed in terms of GW invariants too
\ben
(W_{kl}(t)\,\phi_{b},\phi_{a}) = 
\sum_{l,d=0}^\infty 
\frac{Q^d}{l!}
\langle \phi_a \psi^k,\phi_b \psi^l,t,\dots,t\rangle_{0,2+l,d}.
\een
We have the following relation between ancestor and descendant GW invariants:
\beq\label{da}
\mathcal{D}(\hbar,\mathbf{q}) = e^{F^{(1)}(t)} \widehat{S}(t)^{-1} \mathcal{A}_t(\hbar,\mathbf{q}),
\eeq
where $F^{(1)}(t)$ is the generating function of genus-1 primary (i.e. no descendants) GW invariants
\ben
F^{(1)}(t)=\sum_{n,d=0}^\infty \frac{Q^d}{n!} 
\langle t,\dots,t\rangle_{1,n,d} ,\quad t\in H.
\een
Due to the dilaton shift, the action of $\widehat{S}^{-1}$ changes the Fock space, that is, the RHS of \eqref{da} is a formal power series in $q_0-t, q_1+1,q_2,\dots$. The LHS in \eqref{da}, although it is independent of $t$, should be expanded formally into a Taylor series at $q_0=t$ so that it becomes an element of the space of formal series to which the RHS belongs. 

\subsection{Higher-genus reconstruction}\label{sec:hgr}
The Frobenius manifold $M$ is semi-simple, that is, there exists a point $t\in M$ and a local coordinate system $(u_1,\dots,u_N)$, called {\em canonical coordinate system}, in a neighborhood of $t$ such that 
\ben
\frac{\partial}{\partial u_i} \bullet 
\frac{\partial}{\partial u_j} = \delta_{i,j} \frac{\partial}{\partial u_j},\quad 
\Big( \frac{\partial}{\partial u_i}, \frac{\partial}{\partial u_j}\Big) = 
\frac{\delta_{i,j}}{\Delta_j},\quad 1\leq i,j\leq N,
\een
for some holomorphic functions $\Delta_j\in \O_{M,t}$ such that $\Delta_j(t)\neq 0$. Such points $t$ are said to be semi-simple. The set $\mathcal{K}$ of points $t\in M$ that are not semi-simple is either empty or an analytic hypersurface in $M$ called the {\em caustic} of $M$. Both the canonical coordinates $u_j$ and $\Delta_j$ ($1\leq j\leq N$) extend analytically along any path avoiding the caustic. 

In the semi-simple settings, it was conjectured by Givental (see \cite{Giv2001}) and proved by Teleman (see \cite{Te2012}) that the total ancestor potential can be reconstructed from the Frobenius structure. Let us recall the reconstruction formalism. The main ingredient of the reconstruction is the asymptotic solution $\Psi R^{\rm can}(t,z) e^{U/z}$ to the Dubrovin connection near $z=0.$ Here $\Psi$ is an $N\times N$-matrix with entries $\Psi_{ai}=\sqrt{\Delta_i}\tfrac{\partial t_a}{\partial u_i}$ where the index $a\in \mathcal{B}$, $U=\operatorname{diag}(u_1,\dots,u_N)$ is a diagonal matrix with entries the canonical coordinates, and $R^{\rm can}(t,z)=1+R^{\rm can}_1(t) z+R^{\rm can}_2(t) z^2+\cdots$ is a formal power series in $z$ with coefficients $N\times N$-matrices $R^{\rm can}_k(t)$. Let us identify the matrix $\Psi$ with the following linear map:
\ben
\Psi :\CC^N \to H,\quad e_i\mapsto \sqrt{\Delta_i} \frac{\partial}{\partial u_i} = 
\sum_{a\in \mathcal{B}}  \phi_{a}\, \Psi_{ai},
\een
where $\{e_i\}_{i=1}^N$ is the standard basis of $\CC^N$. In other words, $\Psi$ is the transition map between two trivializing frames of the tangent bundle $TM$ (in a neighborhood of a semi-simple point), that is, the frame of normalized idempotents $\sqrt{\Delta_i} \tfrac{\partial}{\partial u_i}$ ($1\leq i\leq N$) and the frame of flat vector fields $\tfrac{\partial}{\partial t_a}$ ($a\in \mathcal{B}$). 
Substituting the ansatz $\Psi R^{\rm can}(t,z) e^{U/z}$ into the Dubrovin connection, we obtain the following set of equations:
\begin{align}
    \label{rec:R-1}
    dR^{\rm can}_k+ (\Psi^{-1}d\Psi)R^{\rm can}_k & = 
    [dU,R^{\rm can}_{k+1}]. \\
    \label{rec:R-2}
    kR^{\rm can}_k+[U,R^{\rm can}_{k+1}] & =-(\Psi^{-1}\theta \Psi) R^{\rm can}_k.
\end{align}
Starting with $R^{\rm can}_0=1$, the above equations determine uniquely
$R^{\rm can}_k(t)$. Clearly, the matrices $R^{\rm can}_k(t)$ depend analytically on $t\in M\setminus{\mathcal{K}}$. Let us denote by $R_k(t)\in \operatorname{End}(H)$ the linear operator whose matrix with respect to the basis $\{\Psi(e_i)\}_{i=1}^N$ of $H$ is $R^{\rm can}_k(t)$. It can be proved that the operator series $R(t,z):=1+R_1(t)z+R_2(t)z^2+\cdots $ satisfies the same symplectic condition as the calibration, that is, $R(t,z)R^{T}(t,-z)=1$. 

The quantization formalism of symplectic transformations can be applied to $R(t,z)$ too. We obtain a linear operator $\widehat{R}(t)$ acting on the so-called {\em tame} elements of the ancestor Fock space. The action can be described as follows. Let us define the linear operators $V_{kl}\in \operatorname{End}(H)$ by the following identity:
\beq\label{Vkl}
\sum_{k,l=0}^\infty
V_{kl} (-z)^k (-w)^l =\frac{R^T(t,z)R(t,w)-1}{z+w}.
\eeq
Then 
\beq\label{R-action}
\widehat{R}(t)\mathcal{F}(\mathbf{q})= 
\left(
e^{
\frac{\hbar}{2} V(
\partial_{\mathbf{q}},
\partial_{\mathbf{q}})}
\mathcal{F} \right)(R^{-1}\mathbf{q}),
\eeq
where 
\ben
V(\partial_{\mathbf{q}},\partial_{\mathbf{q}})=
\sum_{k,l=0}^\infty
\sum_{a,b\in \mathcal{B}} (\phi^a,V_{kl}\phi^b) 
\frac{\partial^2}{\partial q_{k,a}\partial q_{l,b}},
\een
where $\{\phi^a\}$ is a basis of $H$ dual to $\{\phi_a\}$ with respect to the Poincare pairing. Let us emphasize the structure of the RHS of \eqref{R-action}. First, we apply the differential operator $e^{\frac{\hbar}{2} V}$ to the formal series $\mathcal{F}$. Then we make a substitution $\mathbf{q}\mapsto R(t,z)^{-1} \mathbf{q}=R^T(t,-z)\mathbf{q}$, that is, 
\ben
q_k\mapsto q_k -R_1^T(t) q_{k-1}+R_2^T(t) q_{k-2} - \cdots +
(-1)^k R_k^T(t) q_{0},
\een
where $q_k=\sum_{a\in \mathcal{B}} q_{k,a}\phi_a$. Finally, the RHS of \eqref{R-action} is convergent in the formal topology if $\mathcal{F}$ is a tame series. We postpone the definition of a tame series until Section \ref{sec:aa_hqe} where we start discussing the analytic properties of the HQEs. 

Let us recall the Witten--Kontsevich tau-function, that is, the total descendant potential of the point
\beq\label{tdp:pt}
\mathcal{D}_{\rm pt}(\hbar,\mathbf{t})= 
\exp\Big(
\sum_{g,n=0}^\infty \frac{\hbar^{g-1}}{n!} \, 
\int_{\overline{\mathcal{M}}_{g,n}} 
\mathbf{t}(\psi_1)\cdots \mathbf{t}(\psi_n)
\Big),
\eeq
where $\mathbf{t}=(t_0,t_1,\dots)$  and $\mathbf{t}(z)=t_0+t_1 z+t_2 z^2+\cdots$.
Let us introduce the following formal power series:
\beq\label{tdp:point-N}
\prod_{i=1}^N \mathcal{D}_{\rm pt}(
\hbar \Delta_i,
q_0(u_i),q_1(u_i)+1,q_2(u_i),\dots ),
\eeq
where 
\ben
q_k=
\sum_{a\in \mathcal{B}} q_{k,a} \phi_a = 
\sum_{a\in \mathcal{B}} q_{k,a} \frac{\partial}{\partial t_a}
\een
is interpreted as a flat vector field on $M$ and $q_k(u_i)$ is the natural derivation of a vector field on a function, that is, the coordinate function $u_i$. 
Note that since $\tfrac{\partial u_i}{\partial t_{0,0}}=1$, the product \eqref{tdp:point-N} is a formal series in $q_0, q_1+1,q_2,\dots$. The higher-genus reconstruction formula can be written in the following way:
\beq\label{hgr}
\mathcal{A}_t(\hbar,\mathbf{q})= \widehat{R}(t)\, 
\prod_{i=1}^N \mathcal{D}_{\rm pt}(
\hbar \Delta_i,
q_0(u_i),q_1(u_i)+1,q_2(u_i),\dots ),
\eeq
where the action of $\widehat{R}(t)$ is defined by formula \eqref{R-action}.

\subsection{Fock spaces for the descendant and the ancestor GW invariants}
Let us recall that the total descendant potential belongs to the Fock space \eqref{dfock} while the total ancestor potential belongs to 
$ 
\CC[\![t]\!](\!(\hbar)\!)[\![q_0,q_1+1,q_2,\dots]\!].
$ 
Both Fock spaces can be replaced by smaller ones because the total descendant and ancestor potentials are in some sense analytic in the non-jet variable $q_0$. Let us start with the ancestor potential. By mirror symmetry, we can interpret the RHS of \eqref{hgr} as the total ancestor potential of a simple elliptic singularity. It is known (see \cite{Mi2014}) that the total ancestor potential in singularity theory extends analytically across the caustic, that is, the ancestor GW invariants \eqref{AGW} are convergent power series in $t$ that extend to analytic functions $\forall t\in M$. Note also that by definition the ancestor GW invariant \eqref{AGW} is $0$ if $2g-2+n\leq 0$, that is, if the moduli space $\overline{\mathcal{M}}_{g,n}$ is empty.  We conclude that both $\log \mathcal{A}_t(\hbar,\mathbf{q})$ and $\mathcal{A}_t(\hbar,\mathbf{q})$  belong to 
$ 
\O(M)(\!(\hbar)\!)[\![q_0,q_1+1,q_2,\dots]\!]
$ 
where $\O(M)$ is the algebra of holomorphic functions on $M.$ For degree reasons, the ancestor GW invariants \eqref{AGW} must be polynomial in the variables $t_a$ ($a\neq (0,1)$).
Let us introduce the ring 
\ben
\CC[M]:=\O(\mathbb{D})[
t_{0,0}, t_{i,1},\dots,t_{i,a_i-1}(1\leq i\leq 3)],
\een
where $\mathbb{D}:=\{t_{0,1}\ |\ |Qe^{t_{0,1}}|<1\}$. Then 
$\mathcal{A}_t(\hbar,\mathbf{q})$ belongs to 
\beq\label{AFS}
\CC[M](\!(\hbar)\!)[\![q_1+1,q_2,\dots]\!].
\eeq
From now on, we will refer to \eqref{AFS} as the {\em ancestor Fock space}. 

Now let us examine the total descendant potential. Recall formula \eqref{da} and the definition \eqref{S-action:fock} of $\widehat{S}(t)^{-1}$. If $\mathcal{F}(\mathbf{q})$ is any element of the ancestor Fock space and $t\in M$ is such that $t_{0,1}=0$, then $\mathcal{F}([S(t,z)\mathbf{q}]_+)$ is still an element of the ancestor Fock space \eqref{AFS}. Indeed, the only subtle point in proving this statement is the substitution
\ben
q_0\mapsto q_0 + S_1(t) q_1+S_2(t) q_2+\cdots =
q_0- S_1(t) 1 + S_1(t)(q_1+1) +S_2(t) q_2+\cdots .
\een
However, $S_1(t)1=t$ and since we assumed that $t_{0,1}=0$, we get that the point $q_0-S_1(t)1=q_0-t$ is in the domain of analyticity of $\mathcal{F}$ provided that $q_0$ is. 
We get that the Taylor series expansion of $\mathcal{D}(\hbar,\mathbf{q})$ at $q_0=t$ belongs to the following space 
\beq\label{M-fock}
e^{\tfrac{1}{2\hbar}W(\mathbf{q},\mathbf{q})+F^{(1)}(t)}
\CC[M](\!(\hbar)\!)[\![q_1+1,q_2,\dots]\!],
\eeq
where the non-jet variable $q_0$ is identified with the flat coordinate $t$ on $M.$ We will refer to \eqref{M-fock} as the {\em descendant Fock space}.

\section{Hirota quadratic equations for the ancestors}\label{sec:vop}

The quantized operator $\widehat{S}(t)^{-1}$ gives an isomorphism between the ancestor and the descendant Fock spaces.
Therefore, the descendant HQEs of $\PP^1_a$ are equivalent to a system of HQEs for the total ancestor potential. We simply have to conjugate the bilinear operator $\Omega_{\rm aff}(\lambda)$ by $\widehat{S}(t)^{-1}\otimes \widehat{S}(t)^{-1}$, that is, let us define $\Omega_{\rm aff}(t,\lambda)$ by 
\beq\label{abio}
\Omega_{\rm aff}(\lambda)\, 
\Big(
\widehat{S}(t)^{-1}\otimes \widehat{S}(t)^{-1} \Big)= 
\Big(
\widehat{S}(t)^{-1}\otimes \widehat{S}(t)^{-1}\Big)\, 
\Omega_{\rm aff}(t,\lambda).
\eeq
Using Givental's quantization formalism we will obtain a formula for $\Omega_{\rm aff}(t,\lambda)$ similar to \eqref{affine_casimir}. The conjugation of the vertex operators will be done in the same way as in \cite{Giv2003}. Let us advertise that for the conjugation of the Virasoro term we will make use of {\em Borcherds $n$-products} (see \cite{Bo1986}).
\subsection{Quantization of linear Hamiltonians}
Following Givental (see \cite{Giv2001}), let us introduce the following symplectic form:
\ben
\Omega(f(z),g(z)):=\operatorname{Res}_{z=0} (f(-z),g(z))dz,
\quad
f(z),g(z)\in H[z,z^{-1}],
\een
where $H[z,z^{-1}]$ is the vector space of Laurent polynomials with coefficients in $H$. Note that the symplectic form $\Omega$ extends to both completions $H(\!(z^{-1})\!)$ and $H(\!(z)\!)$ of $H[z,z^{-1}]$. The calibration $S(t,z)$ and the asymptotic operator $R(t,z)$ are symplectic transformations of respectively $H(\!(z^{-1})\!)$ and $H(\!(z)\!)$. 
The symplectic vector space $H[z,z^{-1}]$ has an associated Heisenberg Lie algebra. As a vector space it is just $H[z,z^{-1}]\oplus \CC$. The Lie bracket is defined by the following formulas:
\ben
[f,g]:= \Omega (f,g),\quad 
[f,1]:=0,\quad 
\forall \, f,g\in H[z,z^{-1}].
\een
The symplectic vector space has also a natural Poisson manifold structure.
Namely, for $f\in H[z,z^{-1}]$, using the symplectic form we define the linear function $h_f:=\Omega(f,\ )$. There exists a unique Poisson bracket on $H[z,z^{-1}]$ such that 
\ben
\{h_f,h_g\}=\Omega(f,g),\quad \forall f,g\in H[z,z^{-1}].
\een
Let us recall the basis $\{\phi_a\}_{a\in \mathcal{B}}$ of $H$ and denote by $\{\phi^a\}_{a\in \mathcal{B}}$ the dual basis with respect to the Poincar\'e pairing. 
The Heisenberg Lie algebra has a natural representation $f\mapsto \widehat{f}$ on the Fock space \eqref{fock} such that the central element acts by multiplication by 1 and 
\beq\label{qlh}
(\phi^a (-z)^{-k-1})\sphat := q_{k,a}/\sqrt{\hbar},\quad 
(\phi_a z^{k})\sphat:= 
- \sqrt{\hbar} 
\frac{\partial}{\partial q_{k,a}},\quad 
k\geq 0,\quad a\in \mathcal{B}.
\eeq
The quantization rules \eqref{qlh} can be interpreted in terms of quantization of linear Hamiltonians.  Put
\beq\label{Darboux}
q_{k,a}:=\Omega(\phi^a(-z)^{-k-1},\ ),\quad 
p_{k,a}:=-\Omega(\phi_a z^k,\ ) =\Omega(\ ,\phi_a z^k).
\eeq
It is straightforward to verify that the set of all linear functions $p_{k,a}$ and $q_{k,a}$ ($k\geq 0$, $a\in \mathcal{B}$) is a Darboux coordinate system, that is, 
$ 
\{p_{k,a},q_{l,b}\}=\delta_{k,l}\delta_{a,b}.
$ 
The quantization rules \eqref{qlh} turn into  
\beq\label{qqp}
\widehat{q}_{k,a}= q_{k,a}/\sqrt{\hbar},\quad  
\widehat{p}_{k,a}=\sqrt{\hbar} \frac{\partial}{\partial q_{k,a}}.
\eeq
The substitutions $q'_{0,0,0}-q''_{0,0,0}=m\sqrt{\hbar}$ and $q'_{k,0,0}-q''_{k,0,0}=0$ for $k>0$ can be written conveniently in the following way: 
\ben
\widehat{\mathbf{f}}_{L-1}(\lambda)\otimes 1 - 
1\otimes \widehat{\mathbf{f}}_{L-1}(\lambda) = 2\pi \ii \, m,
\een
where 
\ben
\mathbf{f}_{L-1}(\lambda,z):=
\sum_{n\in \ZZ} 
I^{(n)}_{L-1}(\lambda) (-z)^n = 
2\pi\ii\sum_{k=0}^\infty \frac{\lambda^k}{k!}\, P\, (-z)^{-k-1}.
\een
Put $\mathbf{f}_{L-1}(t,\lambda,z):=S(t,z)\mathbf{f}_{L-1}(\lambda,z)$ and 
let us recall the bilinear operator $\Omega_{\rm aff}(t,\lambda)$ defined by \eqref{abio}. 
\begin{definition}
\label{def:aff_HQE_anc}
Let $F(\hbar,\mathbf{q})\in \CC[M](\!(\hbar)\!)[\![q_1+1,q_2,\dots]\!]$ be a formal power series. We say that $F$ satisfies the ancestor HQEs of $\PP^1_a$ if for every integer $m\in \ZZ$, the formal power series 
\beq\label{hqe-anc}
\left.
\Omega_{\rm aff}(t,\lambda) (F\otimes F)
\right|_{
\widehat{\mathbf{f}}_{L-1}(t,\lambda)\otimes 1 - 
1\otimes \widehat{\mathbf{f}}_{L-1}(t,\lambda) = 2\pi \ii \, m
}
\eeq
is regular in $\lambda$, that is, the coefficients in front of the monomials in $\mathbf{q}'$, $\mathbf{q}''$, and $\hbar$ are polynomials in $\lambda.$  
\end{definition}

\begin{proposition}\label{prop:HQE_a}
    The total descendant potential $\mathcal{D}(\hbar,\mathbf{q})$ is a solution to the HQEs in Definition \ref{def:aff_HQE} if and only if the total ancestor potential $\mathcal{A}_t(\hbar,\mathbf{q})$ is a solution to the HQEs in Definition \ref{def:aff_HQE_anc}.
\end{proposition}
\proof
The only thing that we have to prove is that the substitutions that we used to define the two sets of HQEs are intertwined by the operator $\widehat{S}^{-1}\otimes \widehat{S}^{-1}$. Indeed, using formulas \eqref{da} and \eqref{S-action:fock}, we get that the HQEs for the total descendant potential are equivalent to the regularity of the following formal power series:
\beq\label{CalF}
\left.\left( 
 \mathcal{F}([S(t,z)\mathbf{q}']_+,[S(t,z)\mathbf{q}'']_+)
\right)
\right|_{
\widehat{\mathbf{f}}_{L-1}(\lambda)\otimes 1 - 
1\otimes \widehat{\mathbf{f}}_{L-1}(\lambda) = 2\pi \ii \, m
}
\eeq
where
\ben
\mathcal{F}(\mathbf{q}',\mathbf{q}''):=
\Omega_{\rm aff}(t,\lambda) (
\mathcal{A}_t(\hbar,\mathbf{q}')\mathcal{A}_t(\hbar,\mathbf{q}''))
\een
and we dropped the exponential terms from formulas \eqref{da} and \eqref{S-action:fock} 
because they do not affect the regularity condition. Put 
$\overline{\mathbf{q}}'=[S(t,z)\mathbf{q}']_+$ and 
$\overline{\mathbf{q}}''=[S(t,z)\mathbf{q}'']_+$. We have 
\ben
\widehat{\mathbf{f}}_{L-1}(\lambda)(\mathbf{q}') - 
\widehat{\mathbf{f}}_{L-1}(\lambda)(\mathbf{q}'') =
\frac{1}{\sqrt{\hbar}}
\Omega(\mathbf{f}_{L-1}(\lambda,z), \mathbf{q}'-\mathbf{q}'').
\een
and 
\ben
\widehat{\mathbf{f}}_{L-1}(t,\lambda)(\overline{\mathbf{q}}') - 
\widehat{\mathbf{f}}_{L-1}(t,\lambda)(\overline{\mathbf{q}}'')=
\frac{1}{\sqrt{\hbar}}
\Omega(\mathbf{f}_{L-1}(t,\lambda,z), \overline{\mathbf{q}}'-\overline{\mathbf{q}}'').
\een
Since $S(t,z)$ is a symplectic transformation and 
by definition $\mathbf{f}_{L-1}(t,\lambda,z)=S(t,z)\mathbf{f}_{L-1}(\lambda,z)$ and 
$\overline{\mathbf{q}}=[S(t,z)\mathbf{q}]_+$ we get
\ben
\widehat{\mathbf{f}}_{L-1}(\lambda)(\mathbf{q}') - 
\widehat{\mathbf{f}}_{L-1}(\lambda)(\mathbf{q}'') =
\widehat{\mathbf{f}}_{L-1}(t,\lambda)(\overline{\mathbf{q}}') - 
\widehat{\mathbf{f}}_{L-1}(t,\lambda)(\overline{\mathbf{q}}'').
\een
Therefore the series \eqref{CalF} is regular in $\lambda$ if and only if the following series 
\ben
\left.\left( 
 \mathcal{F}(\overline{\mathbf{q}}',\overline{\mathbf{q}}'')
\right)
\right|_{
\widehat{\mathbf{f}}_{L-1}(t,\lambda)\otimes 1 - 
1\otimes \widehat{\mathbf{f}}_{L-1}(t,\lambda) = 2\pi \ii \, m
}
\een
is regular in $\lambda$ which is precisely the HQEs for the total ancestor potential $\mathcal{A}_t(\hbar,\mathbf{q}).$
\qed

\subsection{Quantization of quadratic Hamiltonians}
The operators $\widehat{S}(t)^{-1}$ and 
$\widehat{R}(t)$ can be interpreted via quantization of quadratic Hamiltonians (see \cite{Giv2001}). Namely, if $A(z)$ is an infinitesimal symplectic transformation of $H(\!(z^{-1})\!)$ or $H(\!(z)\!)$, then $h_A(f)=\frac{1}{2}\Omega(Af,f)$ is a quadratic function and we define $\widehat{A}:=\widehat{h}_A$ where the quantization of quadratic Hamiltonians is defined as follows. Let $p_{k,a}$ and $q_{k,a}$ be the Darboux coordinates \eqref{Darboux}. The Hamiltonian $h_A$ is a sum of quadratic Darboux monomials.  Indeed, every $f\in H[z,z^{-1}]$ can be decomposed as 
\ben
f=\sum_{k=0}^\infty \sum_{a\in \mathcal{B}} 
\Big(
p_{k,a} \phi^a (-z)^{-k-1} + q_{k,a} \phi_a z^k\Big).
\een
Substituting this decomposition in the definition of $h_A$ we get that $h_A$ is a sum of quadratic Darboux monomials. The quantization $\widehat{h}_A$ is defined by the rules \eqref{qlh} and normal ordering -- differentiation should be applied first. For example, $(p_{k,a} q_{l,b})\sphat = q_{l,b}\tfrac{\partial}{\partial q_{k,a}}$. 
If we have a symplectic transformation $T$ such that $T=e^A$ for some infinitesimal symplectic transformation $A$, then we define $\widehat{T}:= e^{\widehat{h}_A}$. It can be proved that by applying this construction to $S(t,z)^{-1}$ and $R(t,z)$ we obtain linear operators whose action on the Fock space is given respectively by formulas \eqref{S-action:fock} and \eqref{R-action}.
 
The quantization formalism of linear and quadratic Hamiltonians allows us to give a very elegant description of the commutation relations between quantized symplectic transformations and vertex operators. This was first observed in \cite{Giv2003}. Let us focus on the two cases of interest to us, that is, the case of a {\em lower-triangular} symplectic transformation $S(z)=1+S_1 z^{-1}+S_2 z^{-2}+\cdots$, $S_k\in \operatorname{End}(H)$ and the case of an {\em upper-triangular} symplectic transformation 
$R(z)=1+R_1 z+R_2 z^2+\cdots$, $R_k\in \operatorname{End}(H)$. To begin with, let us introduce the following notation: if $\mathbf{f}(z)=\sum_{n\in \ZZ} I^{(n)}(-z)^n\in H[\![z,z^{-1}]\!]$ is any series, then we define the following two truncations of $\mathbf{f}(z)$:
\ben
\mathbf{f}^+(z):=\sum_{k=0}^\infty I^{(k)}(-z)^k,\quad 
\mathbf{f}^-(z):=\sum_{k=0}^\infty I^{(-k-1)}(-z)^{-k-1}.
\een
A vertex operator is by definition an element of the Heisenberg group of the form 
$:e^{\widehat{\mathbf{f}}}:= e^{\widehat{\mathbf{f}}^-} e^{\widehat{\mathbf{f}}^+}$ where $:\ :$ is the normal ordering. 
We have 
\beq\label{VOS}
e^{\widehat{\mathbf{f}}^-} e^{\widehat{\mathbf{f}}^+} \, 
\widehat{S}^{-1} = 
e^{\tfrac{1}{2}\Omega((S\mathbf{f}^+)^+,(S\mathbf{f}^+)^-)}\, 
\widehat{S}^{-1}\, 
e^{\widehat{(S\mathbf{f}})^-} e^{\widehat{(S\mathbf{f})}^+} 
\eeq
and 
\beq\label{VOR}
e^{\widehat{\mathbf{f}}^-} e^{\widehat{\mathbf{f}}^+} \, 
\widehat{R} = 
e^{\tfrac{1}{2}\Omega((R^{-1}\mathbf{f}^-)^+,(R^{-1}\mathbf{f}^-)^-)}\, 
\widehat{R}\, 
e^{\widehat{(R^{-1}\mathbf{f}})^-} e^{\widehat{(R^{-1}\mathbf{f})}^+} \, .
\eeq
Let us outline the proof of \eqref{VOS}. The proof of \eqref{VOR} is similar. 
It is easy to check that if $A(z)$ is an infinitesimal symplectic transformation, then 
\beq\label{qhh}
[\widehat{A}, \widehat{\mathbf{f}}] = (A(z)\mathbf{f})\sphat,\quad 
\forall \mathbf{f}\in H[z,z^{-1}]. 
\eeq
Using \eqref{qhh} we get that the LHS of \eqref{VOS} is 
$\widehat{S}^{-1} e^{\widehat{(S\mathbf{f}^-)}} e^{\widehat{(S\mathbf{f}^+)}}.$
Note that $S\mathbf{f}^+ = (S\mathbf{f}^+)^- + (S\mathbf{f})^+$ and that 
$S\mathbf{f}^-+ (S\mathbf{f}^+)^- = (S\mathbf{f})^-$. Therefore, in order to obtain the RHS of \eqref{VOS} we need only to recall the Baker--Campbell--Hausdorff formula
$e^{A+B}= e^{-\tfrac{1}{2}[A,B]}e^A e^B  $ for $A=(S\mathbf{f}^+)^-$ and $B=(S\mathbf{f})^+$.   
\begin{remark}
When applying formulas \eqref{VOS} and \eqref{VOR} we have to justify why we can exponentiate the action of the element $\mathbf{f}$ of the Heisenberg Lie algebra and why the composition with the quantized sympletcic transformation makes sense. In our case, $\mathbf{f}$ depends on an auxiliary variable $\lambda$ and we can address both questions by using an appropriate $\lambda$-adic topology.  
\qed
\end{remark}

\subsection{Vertex operators for the ancestors}
Note that the vertex operators \eqref{vop} can be written in the following way. Put 
\ben
\mathbf{f}_E (\lambda,z):= 
\sum_{n\in \ZZ}
I_E^{(n)}(\lambda)\, (-z)^n
\een
and let $\widehat{\mathbf{f}}_E(\lambda):=(\mathbf{f}_E (\lambda,z))\sphat$. 
Then 
$\Gamma^E(\lambda)= 
e^{\widehat{\mathbf{f}}^-_E(\lambda)}
e^{\widehat{\mathbf{f}}^+_E(\lambda)}
$. 
Let us recall the quadratic form $W$ defined by 
\eqref{W_form_1} and \eqref{W_form_2}. The following proposition is straightforward to prove. We leave it as an exercise. 
\begin{proposition}\label{prop:W}
The quadratic form $W$ has the following properties.
\begin{enumerate}
\item[a)]
We have
$W_{kl}=\sum_{i=1}^{k+1} (-1)^{i+1} S_{k+1-i}^T(t) S_{l+i}(t).$
\item[b)]
We have  $W_{kl}^T=W_{lk}.$
\item[c)]
We have 
$W(\mathbf{q},\mathbf{q})=
\Omega([S(t,z)\mathbf{q}(z)]_+,S(t,z)\mathbf{q}(z)).
\qed$
\end{enumerate}
\end{proposition}
Combining formula \eqref{VOS} and Proposition \ref{prop:W} we get the following key result due to Givental \cite{Giv2003}.
\begin{proposition}\label{prop:S-conj}
Put $\mathbf{f}_E(t,\lambda,z):=S(t,z) \mathbf{f}_E(\lambda,z)$ and $\Gamma^E(t,\lambda):=
e^{\widehat{\mathbf{f}}^-_E(t,\lambda)}
e^{\widehat{\mathbf{f}}^+_E(t,\lambda)}
$. The following formula holds:
\ben
\Gamma^E(\lambda) \widehat{S}(t)^{-1} = 
e^{\frac{1}{2}\, W(
\mathbf{f}^+_E(\lambda,z)
\mathbf{f}^+_E(\lambda,z))}\,  
\widehat{S}(t)^{-1}
\Gamma^E(t,\lambda),
\een
where $W$ is the quadratic form defined by \eqref{W_form_1} and \eqref{W_form_2}.
\end{proposition} 

\medskip

Using Proposition \ref{prop:S-conj}, we will obtain a formula for the bilinear operator $\Omega_{\rm aff}(t,\lambda)$ in terms of the vertex operators $\Gamma^E(t,\lambda)$ and the currents 
$\phi_E(t,\lambda):=\partial_\lambda \widehat{\mathbf{f}}_E(t,\lambda)$. 
Let us define the operator $L(t,\lambda)$ by 
\beq\label{t-vir}
L(\lambda ) \,
\Big(
\widehat{S}(t)^{-1}\otimes \widehat{S}(t)^{-1} 
\Big) = 
\Big(
\widehat{S}(t)^{-1}\otimes \widehat{S}(t)^{-1} 
\Big)\,
L(t,\lambda).
\eeq
We will discuss $L(t,\lambda)$ in the next section. Recalling Proposition \ref{prop:S-conj}, we get that the ancestor bilinear operator has the following form: 
\beq\label{affine_casimir_anc}
\Omega_{\rm aff}(t,\lambda):=
\sum_{E\in \Phi_{\rm aff}} b_E(t,\lambda) \Gamma^E(t,\lambda)\otimes
\Gamma^{-E}(t,\lambda)-
\Big(
\sum_{F\in \Phi_{\rm aff}^{\rm im}} 
b_F(t,\lambda) \Gamma^F(t,\lambda)\otimes \Gamma^{-F}(t,\lambda)
\Big)
L(t,\lambda),
\eeq
where the coefficients $b_E(t,\lambda)$ for $E\in \Phi_{\rm aff}$ and  $b_F(t,\lambda)$ for $F\in \Phi^{\rm im}_{\rm aff}$ can be described as follows. 
\begin{proposition}\label{prop:be}
Let us define the {\em phase factors} $B^{E,F}(t,\lambda,\mu)$ by the following identity:
\ben
\Gamma^E(t,\lambda)\Gamma^F(t,\mu)= B^{E,F}(t,\lambda,\mu) 
:\Gamma^E(t,\lambda)\Gamma^F(t,\mu):\ , \quad 
E,F\in K^0(\PP^1_a).
\een
Then the coefficient
\beq\label{b_Et}
b_E(t,\lambda) = 
b_E(\lambda) \, 
e^{W(\mathbf{f}_E^+(\lambda,z), \mathbf{f}_E^+(\lambda,z))}=
\lim_{\lambda'\to\lambda} 
\frac{B^{E,E}(t,\lambda',\lambda)}{(\lambda'-\lambda)^{(E|E)}}.
\eeq
\end{proposition}
\proof
By definition, the coefficient $b_E(t,\lambda)$ is defined by the following identity:
\ben
b_E(\lambda) \Gamma^E(\lambda)\otimes \Gamma^{-E}(\lambda) \, 
(\widehat{S}^{-1}\otimes \widehat{S}^{-1}) = 
(\widehat{S}^{-1}\otimes \widehat{S}^{-1})\, 
b_E(t,\lambda) \Gamma^E(t,\lambda)\otimes \Gamma^{-E}(t,\lambda). 
\een
The first identity is a direct corollary of Proposition \ref{prop:S-conj}. To prove the second one, we start with the product formula
\ben
\Gamma^E(\lambda)\Gamma^F(\mu)= 
B^{E,F}(\lambda,\mu) 
:\Gamma^E(\lambda)\Gamma^F(\mu):
\een
Let us conjugate the above formula by $\widehat{S}^{-1}$. Using Proposition \ref{prop:S-conj} we get 
\ben
e^{\tfrac{1}{2} (
W(\mathbf{f}_E^+,\mathbf{f}_E^+) + 
W(\mathbf{f}_F^+,\mathbf{f}_F^+))
}
B^{E,F}(t,\lambda,\mu) = 
B^{E,F}(\lambda,\mu)
e^{\tfrac{1}{2} 
W(\mathbf{f}_E^++\mathbf{f}_F^+,\mathbf{f}_E^++\mathbf{f}_F^+)
}
\een
where we suppressed the arguments of $\mathbf{f}_E^+(\lambda,z)$ and $\mathbf{f}_F^+(\mu,z)$. The above formula implies that 
\beq\label{phase-da}
B^{E,F}(t,\lambda,\mu)= 
B^{E,F}(\lambda,\mu)
e^{W(\mathbf{f}_E^+(\lambda,z),\mathbf{f}_F^+(\mu,z))}.
\eeq
Using the above formula we get 
\ben
\lim_{\lambda'\to\lambda} 
\frac{B^{E,E}(t,\lambda',\lambda)}{(\lambda'-\lambda)^{(E|E)}}=
b_E(\lambda) 
e^{W(\mathbf{f}_E^+(\lambda,z),\mathbf{f}_F^+(\lambda,z))}=
b_E(t,\lambda).
\qed
\een

\subsection{Borcherds n-products}\label{sec:Bnp}
Following Bakalov--Milanov (see \cite{BM2013}, Proposition 3.2) for any pair of fields $A(\lambda)$  , $B(\lambda)$ that are local with respect to each other, we define their $n$-product  
\beq\label{n-product}
A(\lambda)_{(n)}B(\lambda) := \operatorname{Res}_{\lambda'=\lambda} 
A(\lambda')B(\lambda) (\lambda'-\lambda)^n d\lambda',
\eeq
where the composition $A(\lambda')B(\lambda)$, thanks to locality, is interpreted via its Laurent series expansion at $\lambda'=\lambda$ and the residue is computed formally as the coefficient in front of $(\lambda'-\lambda)^{-1}$. The product \eqref{n-product} is a {\em Borcherds $n$-product}, that is, it satisfies Borcherds axioms of a vertex algebra --  see \cite{Bo1986}, Section 4, identities (i)--(v). We would like to use definition \eqref{n-product} when $A(\lambda)$ and $B(\lambda)$ are operator series of the type $\partial_\lambda\widehat{\mathbf{f}}_\alpha(\lambda)$. Such operator series define linear maps
\ben
A,B:
\CC_\hbar [\![q_0,q_1+1,q_2,\dots]\!] \to 
\mathcal{K}_\hbar [\![q_0,q_1+1,q_2,\dots]\!],
\een
where $\mathcal{K}=\CC(\!(\lambda^{-1/l})\!)$ and for any field 
$k$ we denote by $k_\hbar:=k(\!(\hbar^{1/2})\!)$ the field of formal Laurent series in $\hbar^{1/2}$ with coefficients in $k$. The composition $A(\lambda_1)B(\lambda_2)$ is a formal power series in $q_0,q_1+1,q_2,\dots$ with coefficients in $\CC(\!(\lambda_1^{-1/l})\!)(\!(\lambda_2^{-1/l})\!)(\!(\hbar^{1/2})\!)$. The standard definition of locality works, that is, we say that $A$ and $B$ are local to each other if there exists an integer $N\geq 0$ such that 
\ben
(\lambda_1-\lambda_2)^N A(\lambda_1)B(\lambda_2) = 
(\lambda_1-\lambda_2)^N B(\lambda_2)A(\lambda_1).
\een
If $A(\lambda)=\partial_\lambda \widehat{\mathbf{f}}_\alpha(\lambda)$ and 
$B(\lambda)=\partial_\lambda \widehat{\mathbf{f}}_\beta(\lambda)$, then using Proposition \ref{prop:B_perp}, we get 
\ben
A(\lambda_1) B(\lambda_2) = 
\frac{1}{l^2}\sum_{k=1}^l 
(\tau^k(\alpha)|\beta) \, 
\frac{\xi^k}{\left(
\lambda_1^{1/l}-\xi^k \lambda_2^{1/l}\right)^2}\  +\ 
:A(\lambda_1) B(\lambda_2):.
\een
The above sum is symmetric with respect to switching the pairs $(\alpha,\lambda_1)$ and $(\beta,\lambda_2)$. This proves that the fields are local to each other. In particular, we may define their $n$-product by \eqref{n-product}. 
\begin{remark}
Our definition of a field can be compared to the standard one (see \cite{FB2001,Kac1998}) as follows. 
Suppose that $V=\oplus_{d=0}^\infty V_d$ is a graded vector space where the homogeneous components $V_d$ are finite dimensional complex vector spaces. By definition, a field on $V$ is a linear map 
\beq\label{field-A}
A: V \to V(\!(z)\!)=
\bigoplus_{d=0}^\infty \CC(\!(z)\!)\otimes V_d.
\eeq
Let us define topologies on $V$ and $V(\!(z)\!)$ such that a basis of open neighborhoods of $0$ is given respectively by $\oplus_{i\geq d}V_i$ ($d\geq 0$) and $\oplus_{i\geq d} \CC(\!(z)\!) \otimes V_i$ ($d\geq 0$). In our terminology, a field on $\overline{V}$ is any linear map
\ben
A:\overline{V}\to \overline{V(\!(z)\!)} = 
\prod_{d=0}^\infty \CC(\!(z)\!)\otimes V_d
\een
where $\overline{V}$ and $\overline{V(\!(z)\!)}$ are the completions with respect to the above topologies.  
\qed
\end{remark}
Let us recall the Virasoro term $L(\lambda)$ defined in Section \ref{sec:bilo}. We will prove that $L(\lambda)$ can be understood in terms of the $(-1)$-product. 

\begin{proposition}\label{prop:coset-vir}
The generating series $L_\perp(\lambda)$ can be computed by the following formula: 
\ben
L_\perp(\lambda):= \frac{1}{2}\, \sum_{i=1}^{N-2} 
\Big(
\phi_{\alpha_i}(\lambda)\otimes 1- 1\otimes \phi_{\alpha_i}(\lambda)
\Big)_{(-1)} \Big(
\phi_{\beta_i}(\lambda)\otimes 1- 1\otimes \phi_{\beta_i}(\lambda)
\Big).
\een
\end{proposition}
\proof
By definition the $(-1)$-product of 
$\phi_{\alpha_i}\otimes 1-1\otimes \phi_{\alpha_i}$ and 
$\phi_{\beta_i}\otimes 1-1\otimes \phi_{\beta_i}$ is 
\ben
\operatorname{Res}_{\lambda'=\lambda} 
\frac{d\lambda'}{\lambda'-\lambda} \Big(
\phi_{\alpha_i}(\lambda')\otimes 1-1\otimes \phi_{\alpha_i}(\lambda')
\Big) 
\Big(
\phi_{\beta_i}(\lambda)\otimes 1-1\otimes \phi_{\beta_i}(\lambda)
\Big). 
\een
Let us recall the following identity:
\ben
[\phi^+_{\alpha_i}(\lambda'),\phi^-_{\beta_i}(\lambda)]=
\partial_{\lambda'}\partial_{\lambda} \Omega(
\mathbf{f}^+_{\alpha_i}(\lambda',z),
\mathbf{f}^-_{\beta_i}(\lambda,z))
\een
where $\phi_a^{\pm}(\lambda):=(\partial_\lambda \mathbf{f}^\pm_a(\lambda,z))\sphat$. 
Using the above formula and the normally ordered product we can rewrite the Borcherds $(-1)$-product as a sum of the following two residues:
\ben
\operatorname{Res}_{\lambda'=\lambda} 
\frac{d\lambda'}{\lambda'-\lambda} :\Big(
\phi_{\alpha_i}(\lambda')\otimes 1-1\otimes \phi_{\alpha_i}(\lambda')
\Big) 
\Big(
\phi_{\beta_i}(\lambda)\otimes 1-1\otimes \phi_{\beta_i}(\lambda)
\Big):
\een
and 
\beq\label{BP-res}
\operatorname{Res}_{\lambda'=\lambda} 
\frac{d\lambda'}{\lambda'-\lambda} 
2\partial_{\lambda'}\partial_{\lambda} \Omega(
\mathbf{f}^+_{\alpha_i}(\lambda',z),
\mathbf{f}^-_{\beta_i}(\lambda,z)). 
\eeq
The first residue coincides with the $i$th summand in \eqref{vir-field_1}. To complete the proof we need only to prove that 
\beq\label{cvir-res}
2\lambda^2
\operatorname{Res}_{\lambda'=\lambda} 
\frac{d\lambda'}{\lambda'-\lambda} 
\sum_{i=1}^{N-2}
\partial_{\lambda'}\partial_{\lambda} \Omega(
\mathbf{f}^+_{\alpha_i}(\lambda',z),
\mathbf{f}^-_{\beta_i}(\lambda,z)) = 
\operatorname{tr}\Big(
\frac{1}{4} +\theta \theta^T\Big)=
\sum_{j=1}^3\sum_{p=1}^{a_j-1} 
\frac{p}{a_j}\Big(1-\frac{p}{a_j}\Big).
\eeq
To this end, recalling formula \eqref{phase_B}, we get that
$ 
\partial_{\lambda'}\partial_{\lambda} \Omega(
\mathbf{f}^+_{\alpha_i}(\lambda',z),
\mathbf{f}^-_{\beta_i}(\lambda,z))=
\partial_{\lambda'}\partial_{\lambda} \, 
\log B_{\alpha_i,\beta_i}(\lambda',\lambda) 
$ 
coincides with the following sum: 
\ben
\frac{1}{l^2}\, 
\sum_{k=0}^{l-1} \xi^k
(\tau^k(\alpha_i)|\beta_i) \, 
\frac{\lambda^{\tfrac{1}{l}-1} (\lambda')^{\tfrac{1}{l}-1}}{
\Big(
\lambda^{\tfrac{1}{l}} -\xi^k (\lambda')^{\tfrac{1}{l}} \Big)^2},
\een
where $\xi=2^{2\pi\ii/l}.$
The contribution of the summand with $k=0$ to the residue \eqref{BP-res} is
\beq\label{k=0}
\frac{2}{l^2}\, (\alpha_i|\beta_i) \, 
\operatorname{Res}_{\lambda'=\lambda} 
\frac{d\lambda'}{\lambda'-\lambda} \,
\frac{\lambda^{\tfrac{1}{l}-1} (\lambda')^{\tfrac{1}{l}-1}}{
\Big(
\lambda^{\tfrac{1}{l}} - (\lambda')^{\tfrac{1}{l}} \Big)^2}=
\frac{2}{l} \operatorname{Res}_{x=0} \frac{dx}{((1+x)^l-1)x^2}\, \lambda^{-2} =
\frac{l^2-1}{6 l^2}\, \lambda^{-2},
\eeq
where we changed the variables $\lambda'=(1+x)^l \lambda$. 
The contribution of the remaining terms is 
\beq\label{k>0}
\frac{2}{l^2}\,
\sum_{k=1}^{l-1} 
(\tau^k(\alpha_i)|\beta_i) \, 
\frac{\xi^k\lambda^{-2}}{
(1 -\xi^k)^2} = 
\frac{2}{l^2}\,
\sum_{k=1}^{l-1}  
\frac{\xi^{k(r+1)}\lambda^{-2}}{
(1 -\xi^k)^2} ,
\eeq
where we used that $\Psi(\alpha_i)=\phi_{j,p}$ for some $(j,p)$ with $1\leq j\leq 3$ and $1\leq p\leq a_j-1$, we put $r=(a_j-p)l/a_j$, and finally using Proposition \ref{prop:sd}, part c), we transformed
$
\tau^k(\alpha_i)=
e^{-2\pi \ii k p/a_j} \alpha_i = 
\xi^{k r}\alpha_i. 
$ 
It remains only to compute the following sums:
$\tfrac{2}{l^2}\sum_{k=1}^{l-1} \frac{\xi^{k(r+1)}}{(1-\xi^k)^2}.$ 
It is convenient to consider the following generating series:
\ben
S_r(x):=\frac{2}{l^2} \sum_{k=1}^{l-1} \frac{\xi^{k r} }{1-x\xi^k}=
\frac{2}{l^2} \Big(
\frac{l x^{l-r}}{1-x^l} -\frac{1}{1-x}\Big),
\een
where for the second equality we used Lemma \ref{le:identity}.
We have 
\ben
S_r'(x)= \frac{2}{l^2}\, 
\frac{l(l-r) x^{l-r-1} + lr x^{2l-r-1} - (1+x+\cdots+x^{l-1})^2}{
(x-1)^2 (1+x+\cdots+x^{l-1})^2
}.
\een
Using the L'H\^opital's rule, we get that the limit of $S_r'(x)$ when $x\to 1$ is  
\ben
\frac{1}{l^4} \left(
l(l-r)(l-r-1)(l-r-2) + 
lr(2l-r-1)(2l-r-2) - 
2\Big( \Big(\sum_{i=1}^{l-1} i\Big)^2 + 
l \sum_{i=1}^{l-1} i(i-1) \Big)\right).
\een
The above expression, after some straightforward computations, simplifies significantly and we get 
\ben
\frac{2}{l^2} \sum_{k=1}^{l-1} \frac{\xi^{k(r+1)}}{(1-\xi^k)^2}=
\lim_{x\to 1} S_r'(x) = 
\frac{1}{6 l^2} \Big(1-l^2 + 6rl -6 r^2\Big)=
\frac{1-l^2}{6l^2} + 
\frac{p}{a_j} \Big( 1-\frac{p}{a_j}\Big),
\een
where recall that $r=(a_j-p)l/a_j$.
Substituting this formula in \eqref{k>0} and recalling also \eqref{k=0}, we finally get 
\ben
2
\operatorname{Res}_{\lambda'=\lambda} 
\frac{d\lambda'}{\lambda'-\lambda} 
\partial_{\lambda'}\partial_{\lambda} \Omega(
\mathbf{f}^+_{\alpha_i}(\lambda',z),
\mathbf{f}^-_{\beta_i}(\lambda,z)) = 
\frac{p}{a_j} \Big( 1-\frac{p}{a_j}\Big)\, \lambda^{-2}.
\qed
\een

\begin{proposition}\label{prop:coset-vir-0}
Let $\{\alpha_i\}_{i=N-1}^N$ and 
$\{\beta_i\}_{i=N-1}^N$ be dual bases of $K_0(\PP^1,\CC)$ 
with respect to the Euler pairing. Then
\ben
L_0(\lambda)= \tfrac{1}{2}
\sum_{i=N-1,N}
\operatorname{Res}_{\nu=0} \frac{d\nu}{\nu(q-q^{-1})}
\Big(
\phi^\nu_{\alpha_i}(\lambda)\otimes 1- 1\otimes 
\phi^\nu_{\alpha_i}(\lambda)\Big)_{(-1)}
\Big(
\phi^{-\nu}_{\beta_i}(\lambda)\otimes 1- 1\otimes 
\phi^{-\nu}_{\beta_i}(\lambda)\Big).
\een  
\end{proposition}
\proof
We may assume that $\alpha_{N-1}=\delta/l$, $\alpha_N=L-1$, $\beta_{N-1}=L-1$, and $\beta_N=-\delta/l$. We have to prove that the $(-1)$-products on the RHS coincide with the normally ordered products. Just like in the proof of the previous proposition, after recalling the definition of the Borcherds $(-1)$-product, we get that we have to prove that 
\beq\label{BP-res-0}
\sum_{i=N-1,N}
\operatorname{Res}_{\nu=0} \tfrac{d\nu}{\nu(e^{\pi\ii\nu}-e^{-\pi\ii \nu})}
\operatorname{Res}_{\lambda'=\lambda} 
\frac{d\lambda'}{\lambda'-\lambda} 
2\partial_{\lambda'}\partial_{\lambda} \Omega(
\mathbf{f}^{\nu,+}_{\alpha_i}(\lambda',z),
\mathbf{f}^{-\nu,-}_{\beta_i}(\lambda,z)) = 0. 
\eeq
Using Proposition \ref{prop:cp}, we get 
\ben
\mathbf{f}^{\nu,+}_{\delta/l}(\lambda,z)=
\sum_{k=0}^\infty \Big(
\frac{\lambda^{-\nu-k}}{\Gamma(1-\nu-k)} -
\frac{\lambda^{-\nu-k-1}}{\Gamma(-\nu-k)}(2\pi\ii +\log Q) P
\Big) (-z)^k
\een
and 
\ben
\mathbf{f}^{-\nu,-}_{L-1}(\lambda,z)=
\sum_{k=0}^\infty 
2\pi\ii\, \frac{\lambda^{\nu+k}}{\Gamma(\nu+k+1)}\, P\, (-z)^{-k-1},
\een
where we used that $\operatorname{deg}(\delta/l)=-1$, $\operatorname{rk}(\delta/l)=1$, $\chi_{j,p}(\delta/l)=0$ and $\operatorname{deg}(L-1)=1$, $\operatorname{rk}(L-1)=0$, $\chi_{j,p}(L-1)=0$. The propagator becomes 
\ben
\Omega(
\mathbf{f}^{\nu,+}_{\delta/l}(\lambda',z),
\mathbf{f}^{-\nu,-}_{L-1}(\lambda,z))= 
2\pi\ii\, \sum_{k=0}^\infty (-1)^{k+1} \, 
\frac{(\lambda')^{-\nu-k}}{\Gamma(1-\nu-k)}\, 
\frac{\lambda^{\nu+k}}{\Gamma(1+\nu+k)}.
\een
Recall the reflection formula for the gamma function
\ben
\frac{2\pi\ii}{\Gamma(1-\nu-k) \Gamma(1+\nu+k)} =
\frac{(-1)^k}{\nu+k}\, (e^{\pi\ii\nu}-e^{-\pi\ii\nu}),
\een
we get 
\ben
\partial_{\lambda'}
\Omega(
\mathbf{f}^{\nu,+}_{\delta/l}(\lambda',z),
\mathbf{f}^{-\nu,-}_{L-1}(\lambda,z))= 
(e^{\pi\ii\nu}-e^{-\pi\ii\nu})\,
\sum_{k=0}^\infty  
(\lambda')^{-\nu-k-1}\, \lambda^{\nu+k} =
(e^{\pi\ii\nu}-e^{-\pi\ii\nu})\, \lambda^\nu\, (\lambda')^{-\nu}\, 
\frac{1}{\lambda'-\lambda}
\een
and 
\ben
\operatorname{Res}_{\lambda'=\lambda} 
\frac{d\lambda'}{\lambda'-\lambda} 
2\partial_\lambda\partial_{\lambda'}
\Omega(
\mathbf{f}^{\nu,+}_{\delta/l}(\lambda',z),
\mathbf{f}^{-\nu,-}_{L-1}(\lambda,z)) =
(e^{\pi\ii\nu}-e^{-\pi\ii\nu})\, 
(\nu-\nu^2)\, \lambda^{-2}. 
\een
Substituting the above formula into \eqref{BP-res-0} we get that the summand with $i=N-1$ is $0$. Similar computation shows that the other summand $i=N$ is also 0.
\qed

Using Propositions \ref{prop:coset-vir} and \ref{prop:coset-vir-0} we can give the following interpretation of $L(\lambda)$. Let us introduce the following bilinear pairing on $K(\PP^1_a,\CC)$: 
\beq\label{hp}
h_\nu (E,F)= q\langle E, F\rangle + 
q^{-1}\langle F,E\rangle ,
\eeq
where $q=e^{\pi\ii\nu}$. Note that $h_0(E,F)=(E|F)$ is the intersection pairing. 
\begin{proposition}\label{prop:vir_tp}
Let $\{\alpha_i\}_{i=1}^N$ be a basis of $K(\PP^1,\CC)$ and $\{\alpha^j\}_{j=1}^N$ be the dual basis with respect to the pairing \eqref{hp}: $h_\nu(\alpha_i,\alpha^j)=\delta_{ij}$. Then 
\ben
L(\lambda)= \frac{1}{2} \sum_{i=1}^N 
\operatorname{Res}_{\nu=0} \frac{d\nu}{\nu}
\Big(
\phi^\nu_{\alpha_i}(\lambda)\otimes 1- 1\otimes \phi^\nu_{\alpha_i}(\lambda)
\Big)_{(-1)} \Big(
\phi^{-\nu}_{\alpha^i}(\lambda)\otimes 1- 1\otimes \phi^{-\nu}_{\alpha^i}(\lambda)
\Big).
\een
\end{proposition}
\proof
Let us express the RHS of the above formula in terms of non-twisted periods. To begin with, note that the decomposition $K(\PP^1,\CC)=K_\perp(\PP^1,\CC)\oplus K_0(\PP^1,\CC)$ is orthogonal with respect to the Euler pairing, that is, $\langle \alpha,\beta\rangle =-\langle \beta,\alpha\rangle= 0$ for $\alpha \in K_\perp(\PP^1,\CC)$ and $\beta\in K_0(\PP^1,\CC)$. This follows from Propositions \ref{prop:rad} and \ref{prop:sd}, that is, 
\ben
\langle \alpha,\beta\rangle = ((1-\tau) (1-\tau)^{-1}\alpha |\beta) = 0.
\een
Let $\{\alpha_i\}_{i=1}^{N-2}$ be a basis of $K_\perp(\PP^1_a,\CC)$. For example, let us choose 
$\{\alpha_i\}_{i=1}^{N-2}$ to be the eigen-basis of $\tau$ corresponding via the isomorphism $\Psi$ to 
$\{\phi_{i,p}\}_{1\leq i\leq 3, 1\leq p\leq a_i-1}$. 
Furthermore, let $\alpha_{N-1}:=\delta/l$ and $\alpha_N:=L-1$ be a basis of $K_0(\PP^1,\CC)$. 
The dual basis 
$\{\alpha^j(\nu)\}_{j=1}^{N}$ is uniquely determined by the equations
\ben
h_\nu(\alpha_i,\alpha^j(\nu))= \delta_{i,j},\quad 1\leq i,j\leq N,
\een
where we wrote $\alpha^j=\alpha^j(\nu)$ to emphasize that $\alpha^j$ depends on the parameter $\nu$. 
Since the restriction of $h_0$ to $K_\perp$ is non-degenerate, we get that $\alpha^j(\nu)$ $1\leq j\leq N-2$ depend analytically on $\nu$. In particular, $\beta_j:=\alpha^j(0)$ $(1\leq j\leq N-2)$ form a basis to $K_\perp(\PP^1,\CC)$ dual to $\{\alpha_i\}_{i=1}^{N-2}$ with respect to the intersection pairing $(\ |\ )$. The series $\phi_{E}^\nu(\lambda)$ is also analytic in $\nu$. Therefore, the first $N-2$ terms in the series $L(\lambda)$ have a simple pole at $\nu=0$. The corresponding residues are straightforward to compute. We get that the first $N-2$ terms of $L(\lambda)$ have the following contribution:
\ben
\frac{1}{2}\, \sum_{i=1}^{N-2} 
\Big(
\phi_{\alpha_i}(\lambda)\otimes 1- 1\otimes \phi_{\alpha_i}(\lambda)
\Big)_{(-1)} \Big(
\phi_{\beta_i}(\lambda)\otimes 1- 1\otimes \phi_{\beta_i}(\lambda)
\Big)=L_\perp(\lambda),
\een
where the equality follows from Proposition \ref{prop:coset-vir}. 
Furthermore, as we already pointed out in the introduction, the restriction of the Euler pairing to $K_0(\PP^1_a,\CC)$ is a symplectic form and $\{\delta/l, L-1\}$ is a Darboux basis: 
\ben
\langle \delta/l, L-1\rangle = -\langle L-1,\delta/l\rangle =1.
\een
We get that $\alpha^{N-1}(\nu):= \frac{1}{q-q^{-1}} \, (L-1)$, 
$\alpha^N(\nu):= -\frac{1}{q-q^{-1}} \,\delta/l$ is a basis of $K_0(\PP^1,\CC)$ dual to 
$\alpha_{N-1}=\delta/l$, $\alpha_N=L-1$. The remaining two terms can be written as follows:
\ben
\frac{1}{2}
\sum_{i=N-1,N}
\operatorname{Res}_{\nu=0} \frac{d\nu}{\nu(q-q^{-1})}
\Big(
\phi^\nu_{\alpha_i}(\lambda)\otimes 1- 1\otimes 
\phi^\nu_{\alpha_i}(\lambda)\Big)_{(-1)}
\Big(
\phi^{-\nu}_{\beta_i}(\lambda)\otimes 1- 1\otimes 
\phi^{-\nu}_{\beta_i}(\lambda)\Big),
\een
where $\beta_{N-1}:= L-1$ and $\beta_N:= -\delta/l$. Note 
$\{\alpha_i\}_{i=N-1}^N$ and 
$\{\beta_i\}_{i=N-1}^N$ are dual bases of $K_0(\PP^1,\CC)$ 
with respect to the Euler pairing.  According to Proposition \ref{prop:coset-vir-0} the above formula coincides with $L_0(\lambda)$. 
\qed

\subsection{Virasoro operators for the ancestors}
Let us introduce a $\nu$-deformation of the currents $\phi_\alpha(t,\lambda)=\partial_\lambda \widehat{\mathbf{f}}_{\alpha}(t,\lambda).$ 
Put
\ben
\mathbf{f}^\nu_\alpha(\lambda,z) = 
\sum_{m\in \ZZ} I^{(m+\nu)}_\alpha(\lambda)
(-z)^{m},
\een
$\mathbf{f}^\nu_\alpha(t,\lambda,z):=S(t,z)\mathbf{f}^\nu_\alpha(\lambda,z)$, and
$\phi^\nu_\alpha(t,\lambda):=\partial_\lambda \widehat{\mathbf{f}}^\nu_{\alpha}(t,\lambda).$ 
\begin{proposition}\label{prop:anc_cv}
Let $\{\alpha_i\}_{i=1}^N$ be a basis of $K(\PP^1_a,\CC)$ and $\{\alpha^j\}_{j=1}^N$ be the dual basis with respect to the pairing \eqref{hp}. 
The following formula holds:
\ben
L(t,\lambda)= \frac{1}{2}\sum_{i=1}^{N}
\operatorname{Res}_{\nu=0} \frac{d\nu}{\nu}\, 
\Big(
\phi^\nu_{\alpha_i}(t,\lambda)\otimes 1 - 
1 \otimes \phi^\nu_{\alpha_i}(t,\lambda)\Big)_{(-1)}
\Big(
\phi^{-\nu}_{\alpha^i}(t,\lambda)\otimes 1 - 
1 \otimes \phi^{-\nu}_{\alpha^i}(t,\lambda)\Big).
\een
\end{proposition}
\proof
By definition, 
\ben
L(\lambda) \, (\widehat{S}^{-1}\otimes \widehat{S}^{-1}) = 
(\widehat{S}^{-1}\otimes \widehat{S}^{-1}) \, 
L(t,\lambda).
\een
The infinitesimal version of Proposition \ref{prop:S-conj} yields
\ben
\phi^\nu_\alpha(\lambda) \, \widehat{S}^{-1}(t) = 
\widehat{S}^{-1}(t) \, \phi^\nu_\alpha(t,\lambda).
\een
The operator $\widehat{S}^{-1}$ is compatible with the Borcherd's products:
\ben
\phi^\nu_\alpha(\lambda)_{(n)}\phi^{-\nu}_\beta(\lambda) \, 
\widehat{S}^{-1}(t) = 
\widehat{S}^{-1}(t) \, 
\phi^\nu_\alpha(t,\lambda)_{(n)}\phi^{-\nu}_\beta(t,\lambda). 
\een
The formula follows from Proposition \ref{prop:vir_tp}.
\qed

\medskip

Finally, using the above proposition, let us derive a formula for the Virasoro operator $L(t,\lambda)$ similar to the formula that we used to define $L(\lambda)$. 
Suppose that $\{\alpha_i\}_{i=1}^{N-2}$ is a basis of $K_\perp(\PP^1_a,\CC)$ and let $\alpha_{N-1}:=\delta/l$, $\alpha_N:=L-1$ be a basis of $K_0(\PP^1_a,\CC)$. 
Let us split $L(t,\lambda)=L_\perp(t,\lambda)+L_0(t,\lambda)$ where $L_\perp(t,\lambda)$ (resp. $L_0(t,\lambda)$) is obtained from the formula for $L(t,\lambda)$ in Proposition \ref{prop:anc_cv} by truncating from the sum the last $2$ (resp. the first $N-2$) terms.
Recalling the definition of the Borcherd's $(-1)$-product we get 
\begin{align*}
    L_\perp(t,\lambda) = & 
    \frac{1}{2}\sum_{i=1}^{N-2}
:\Big(
\phi_{\alpha_i}(t,\lambda)\otimes 1 - 
1 \otimes \phi_{\alpha_i}(t,\lambda)\Big)\,
\Big(
\phi_{\beta_i}(t,\lambda)\otimes 1 - 
1 \otimes \phi_{\beta_i}(t,\lambda)\Big): + \\
&
+ 
\operatorname{Res}_{\lambda'=\lambda} 
\frac{d\lambda'}{\lambda'-\lambda}\, 
\sum_{i=1}^{N-2}
\partial_{\lambda'}\partial_{\lambda}\, 
\Omega(
\mathbf{f}_{\alpha_i}^+(t,\lambda',z),
\mathbf{f}_{\beta_i}^-(t,\lambda,z)),  
\end{align*}
where $\{\beta_i\}_{i=1}^{N-2}$ is the basis of $K_\perp(\PP^1_a,\CC)$ dual to $\{\alpha_i\}_{i=1}^{N-2}$ with respect to the intersection pairing.
Note that formula \eqref{phase-da} can be written equivalently as
\ben
\Omega(
\mathbf{f}_E^+(t,\lambda,z),
\mathbf{f}_F^-(t,\mu,z)) = 
\Omega(
\mathbf{f}_E^+(\lambda,z),
\mathbf{f}_F^-(\mu,z)) + 
W(\mathbf{f}_E^+(\lambda,z),
\mathbf{f}_F^+(\mu,z)).
\een
Let us recall that in the proof of Proposition \ref{prop:coset-vir} we derived the following formula (see formula \eqref{cvir-res}):
\ben
\operatorname{Res}_{\lambda'=\lambda} 
\frac{d\lambda'}{\lambda'-\lambda}\, 
\sum_{i=1}^{N-2}
\partial_{\lambda'}\partial_{\lambda}\, 
\Omega(
\mathbf{f}_{\alpha_i}^+(\lambda',z),
\mathbf{f}_{\beta_i}^-(\lambda,z)) = 
\frac{1}{2}\, \operatorname{tr}\Big(
\frac{1}{4}+\theta\,\theta^T\Big)\lambda^{-2}.
\een
Since the quadratic form  $W(\mathbf{f}_E^+(\lambda,z),
\mathbf{f}_F^-(\mu,z))$ does not have a pole at $\lambda=\mu$, the formula for $L_\perp(t,\lambda)$ takes the following form:
\begin{align}
\notag
    L_\perp(t,\lambda) = & 
    \frac{1}{2}\sum_{i=1}^{N-2}
:\Big(
\phi_{\alpha_i}(t,\lambda)\otimes 1 - 
1 \otimes \phi_{\alpha_i}(t,\lambda)\Big)\,
\Big(
\phi_{\beta_i}(t,\lambda)\otimes 1 - 
1 \otimes \phi_{\beta_i}(t,\lambda)\Big): + \\
\label{cos-vir-perp}
&
+ 
\frac{1}{2}\, \operatorname{tr}\Big(
\frac{1}{4}+\theta\,\theta^T\Big)\lambda^{-2} + 
\sum_{i=1}^{N-2} W(
\partial_\lambda
\mathbf{f}_{\alpha_i}^+(\lambda,z), 
\partial_\lambda
\mathbf{f}_{\beta_i}^+(\lambda,z)).
\end{align}
Similar computation yields
\beq\label{cos-vir-0}
L_0(t,\lambda) = \operatorname{Res}_{\nu=0}
\frac{d\nu}{\nu(e^{\pi\ii\nu}-e^{-\pi\ii\nu})} \, 
L_0^\nu(t,\lambda)=
\frac{1}{2\pi\ii}\, 
\left.\frac{d}{d\nu} \, L_0^\nu(t,\lambda)\right|_{\nu=0}\ ,
\eeq
where 
\begin{align}
\label{L0nu}
L_0^\nu(t,\lambda) = & 
\frac{1}{2}\sum_{i=N-1,N}
:\Big(
\phi^\nu_{\alpha_i}(t,\lambda)\otimes 1 - 
1 \otimes \phi^\nu_{\alpha_i}(t,\lambda)\Big)\,
\Big(
\phi^{-\nu}_{\beta_i}(t,\lambda)\otimes 1 - 
1 \otimes \phi^{-\nu}_{\beta_i}(t,\lambda)\Big):
\end{align}
where $\{\beta_{N-1}:=L-1,\beta_{N}:=-\delta/l\}$ is the basis of $K_0(\PP^1_a,\CC)$ dual to $\{\alpha_{N-1}=\delta/l,\alpha_N=L-1\}$ with respect to the Euler pairing. 
Note that if we follow the derivation of \eqref{cos-vir-perp}, then we should have an extra term in formula \eqref{L0nu} of the form 
\ben
\sum_{i=N-1,N-2} W(
\partial_\lambda
\mathbf{f}_{\alpha_i}^{\nu,+}(\lambda,z), 
\partial_\lambda
\mathbf{f}_{\beta_i}^{-\nu,+}(\lambda,z)).
\een
However, this term does not contribute to \eqref{cos-vir-0}. Indeed, differentiating with respect to $\nu$ and setting $\nu=0$ yields an expression that involves either the series
$\partial_\lambda\mathbf{f}_{\alpha_i}^+(\lambda,z)$ or the series
$\partial_\lambda\mathbf{f}_{\beta_i}^+(\lambda,z)$ and they both vanish for $i=N-1,N$.

\section{Analyticity of the ancestor HQEs}\label{sec:aa_hqe}
Let us introduce the following multi-index notation. Suppose that 
\ben
m: \ZZ_{\geq 0}\times \mathcal{B}\to \ZZ_{\geq 0},\quad (k,a)\mapsto m_{k,a}
\een
is a function with finite support, that is, $m_{k,a}\neq 0$ only for finitely many pairs $(k,a)$. For brevity, we will simply say that $m$ is a {\em finitely supported function} and by this we always mean that $m$ has the above form. Let us introduce the following monomials: 
\ben
(\mathbf{q}(z)+z)^m := 
\prod_{k=0}^\infty \prod_{a\in \mathcal{B}} 
(q_{k,a}+\delta_{(k,a),(1,0)})^{m_{k,a}},
\een
where $0\in \mathcal{B}$ denotes the pair $(0,0)$, that is, $\phi_0=\phi_{0,0}=1$. Let us also fix the following notation:
\ben
|m|=\sum_{k,a} m_{k,a},\quad 
m!= \prod_{k,a} m_{k,a}!\, .
\een
The total ancestor potential has the following form:
\beq\label{tap}
\A_t(\hbar,\mathbf{q})=
\sum_{g\in \ZZ} \sum_m 
a^{(g)}_m(t) \, \hbar^{g-1} \, \frac{(\mathbf{q}(z)+z)^m}{m!}\, ,
\eeq
where the coefficients $a^{(g)}_m$ are holomorphic functions in $t\in M$ and the second sum is over all finitely supported functions $m:\ZZ_{\geq 0}\times \mathcal{B}\to \ZZ_{\geq 0}$.
The series \eqref{tap} has the following property: if the coefficient $a^{(g)}_m\neq 0$, then the following inequality holds:
\ben
\sum_{k=0}^\infty \sum_{a\in \mathcal{B}} k m_{k,a} \leq 3g-3 + |m|.
\een
The inequality is easy to prove by degree reasoning: the LHS corresponds to the total degree of the $\psi$-classes in the definition of $a^{(g)}_m$ and the RHS is the dimension of the moduli space over which the $\psi$-classes must be integrated.  
In general, series satisfying such property are said to be {\em tame} (see \cite{Giv2003}). 

\subsection{Ancestor polynomials} \label{sec:apoly}
The total ancestor potential of any target manifold has the following property. Suppose that $\Gamma^{\pm}=:e^{\pm\widehat{\mathbf{f}}}:$ is an arbitrary vertex operator where $\mathbf{f}\in H[\![z,z^{-1}]\!]$. The action of the bilinear operator $\Gamma^+\otimes \Gamma^-$  on $\A_t\otimes \A_t$ is convergent in the formal topology of the ancestor Fock space and the following formula holds:
\beq\label{A-polynomial}
\Gamma^+\otimes \Gamma^-\, 
(\A_t\otimes \A_t) = 
\sum_{g\in \tfrac{1}{2}\ZZ}\, 
\sum_{m',m''} \, A^{(g)}_{m',m''}(\mathbf{f}) \,
\hbar^{g-1}\,
\frac{(\mathbf{q}'(z)+z)^{m'}}{m'!}\,
\frac{(\mathbf{q}''(z)+z)^{m''}}{m''!}
\eeq
where the coefficients $A^{(g)}_{m',m''}(\mathbf{f})$ are polynomials in $\mathbf{f}$. Here the second sum is over all finitely supported functions $m'$ and $m''$. This observation is due to Givental -- see \cite{Giv2001}. We will refer to the polynomials $A^{(g)}_{m',m''}$ as the {\em ancestor polynomials}. There is an additional new feature in the case of the ancestor HQEs of $\PP^1_a$, that is, we have a substitution that eliminates all variables $q'_{k,0,0}$ ($k\geq 0$). Our goal is to prove that up to a certain simple exponential factor the above polynomial property of the total ancestor potential is preserved. 
\begin{lemma}\label{le:suba}
    The substitution in the ancestor HQEs of $\PP^1_a$ is equivalent to the following system of identities:
    \begin{align*}
    [S^T(t,-z)(\mathbf{q}'(z)-\mathbf{q}''(z))]_{0,0,0} & =m\sqrt{\hbar}, \\
    [S^T(t,-z)(\mathbf{q}'(z)-\mathbf{q}''(z))]_{k,0,0} & =0,\quad k>0,
    \end{align*}
    where $[f]_{k,0,0}$ for $f\in H[\![z]\!]$ denotes the coefficient in front of $\phi_{0,0} z^k$.  
\end{lemma}
\proof
The substitution at question is by definition 
\beq\label{subst}
2\pi\ii m \sqrt{\hbar}=
\Omega(\mathbf{f}_{L-1}(t,\lambda,z), \mathbf{q}'(z)-\mathbf{q}''(z)) =
\Omega(\mathbf{f}_{L-1}(\lambda,z), S^T(t,-z) (\mathbf{q}'(z)-\mathbf{q}''(z)) )
\eeq
where we used that $\mathbf{f}_{L-1}(t,\lambda,z)=S(t,z)\mathbf{f}_{L-1}(\lambda,z)$ and that $S(t,z)$ is a symplectic transformation. Since 
\ben
\mathbf{f}_{L-1}(\lambda,z) = 2\pi\ii\, \sum_{k=0}^\infty \frac{\lambda^k}{k!} P(-z)^{-k-1},
\een
the identity \eqref{subst} is equivalent to 
\ben
2\pi\ii m \sqrt{\hbar}= 2\pi\ii \, 
\sum_{k=0}^\infty 
\frac{\lambda^k}{k!} [S^T(t,-z) (\mathbf{q}'(z)-\mathbf{q}''(z))]_{k,0,0}. 
\een
The statement of the lemma follows easily by comparing the coefficients in front of $\lambda$.
\qed

The formulas for the substitution from Lemma \ref{le:suba} has the following form: $q'_{k,0}= q_{k,0}''+m\sqrt{\hbar}\, \delta_{k,0} + f_{k}(\mathbf{q}',\mathbf{q}'')$ where
\beq\label{suba_0}
f_{k}(\mathbf{q}',\mathbf{q}'') = 
\sum_{l=1}^\infty  (-1)^{l+1} [S_l^T(t)(q'_{k+l}-q''_{k+l})]_{0,0},
\eeq
where $[v]_{0,0}:=(v,\phi^{0,0})$ for $v\in H$ denotes the $(0,0)$-coordinate of $v$. The RHS of formula \eqref{suba_0} might still involve variables $q'_{k+l,0}$ with $l>0$. However, since 
$q'_{k+l,0}=q_{k+l,0}''+f_{k+l}(\mathbf{q}',\mathbf{q}'')$, 
by iterating the substitution \eqref{suba_0}, we get that 
\beq\label{suba}
f_{k}(\mathbf{q}',\mathbf{q}'') = 
\sum_{l>k}
\sum_{a\in \mathcal{B}}
\Big(
f'_{k,l,a}(t) q'_{l,a}+ f''_{k,l,a}(t)q''_{l,a}
\Big),
\eeq
where the coefficients $f'_{k,l,a}(t)$ and $f''_{k,l,a}(t)$ are holomorphic functions in $t$ and $f'_{k,l,0}=0$.

\begin{proposition}\label{prop:wave}
    Let $\Gamma^{\pm}=:e^{\pm \widehat{f}}:$ be a vertex operator where $\mathbf{f}(z)=\sum_{n\in \ZZ} I^{(n)} \, (-z)^n$ and the coefficients $I^{(n)}\in H$ are {\em formal vector variables}. The formal power series 
    \ben
    \left.
    \Gamma^+\otimes \Gamma^-\, (\mathcal{A}_t\otimes \mathcal{A}_t)
    \right|_{q'_{k,0,0}-q_{k,0,0}''=m\sqrt{\hbar}\, \delta_{k,0} + f_{k}(\mathbf{q}',\mathbf{q}'')}
    \een
    has the following form:
    \beq\label{A-poly}
    e^{m(I^{(-1)},\phi_0)}\,
    \sum_{g\in \tfrac{1}{2}\ZZ}\, \sum_{m',m''} 
    A^{(g)}_{m',m''}(\mathbf{f}) \, \hbar^{g-1} \, 
    \frac{(\mathbf{q}'(z)+z)^{m'}}{(m')!} \, 
    \frac{(\mathbf{q}''(z)+z)^{m''}}{(m'')!}
    \eeq
    where the second sum is over all functions $m',m'':\ZZ_{\geq 0}\times \mathcal{B}\to \ZZ_{\geq 0}$ with finite support such that $m'_{k,0,0}=0$ and the coefficients $A_{m',m''}^{(g)}(\mathbf{f})$ are polynomials in the vector variables 
    $\{I^{(n)}\}_{n\in \ZZ}$.  
\end{proposition}
\proof
Let us recall the expansion \eqref{tap}. Then we have
\ben
\Gamma^+\, \A_t(\hbar,\mathbf{q})= 
e^{\sum_{k=0}^\infty (I^{(-k-1)},q_k)/\sqrt{\hbar}} \, 
\sum_{g\in \ZZ} \sum_m 
a^{(g)}_m \, \hbar^{g-1} \, \frac{(\mathbf{q}(z)-\sqrt{\hbar}\mathbf{f}^+(z)+z)^m}{m!}.
\een
Let us introduce the notation $(\mathbf{f}^-(z))^{m}$ via the following identity:
\beq\label{vob}
e^{\sum_{k=0}^\infty (I^{(-k-1)},q_k+\delta_{k,1})/\sqrt{\hbar}} = 
\sum_{m} \hbar^{-|m|/2}\, (\mathbf{f}^-(z))^{m}\, 
\frac{(\mathbf{q}(z)+z)^m}{m!},
\eeq
where $|m|:=\sum_{k,a} m_{k,a}$ and $q_k+\delta_{k,1}$ is the dilaton shift. In other words, 
\ben
(\mathbf{f}^-(z))^{m} :=
\prod_{k=0}^\infty \prod_{a\in \mathcal{B}} 
(I^{(-k-1)},\phi_a)^{m_{k,a}}.
\een
This is a monomial in the variables $I=(I^{(n)})_{n<0}$. Eventually, we would like to replace $q_{0,0,0}'$ with $q_{0,0,0}''+m\sqrt{\hbar} +\cdots$. For that reason the term in the exponent in \eqref{vob} containing $q_{0,0,0}$ plays a distinguished role and it is better to keep it in the exponent. In other words, let us introduce the following expansion:
\beq\label{vob^000}
e^{\sum_{k=0}^\infty (I^{(-k-1)},q_k+\delta_{k,1})/\sqrt{\hbar}} = 
\sum_{m} \hbar^{-|m|/2}\,  
e^{(I^{(-1)},\phi_0) q_{0,0,0}/\sqrt{\hbar}}\,
(\mathbf{f}^-(z))^{m}\,
\frac{(\mathbf{q}(z)+z)^m}{m!}
\eeq
where the sum is over all finitely supported functions $m:\ZZ_{\geq 0}\times \mathcal{B}\to \ZZ_{\geq 0}$ such that $m_{0,0,0}=0$. 
Let us decompose the vertex operators as follows
\ben
\Gamma^\pm=:
\widetilde{\Gamma}^\pm\,
e^{\pm (I^{(-1)},\phi_{0})q_{0,0,0}/
\sqrt{\hbar}}\, 
e^{\mp (I^{(-2)},1)/\sqrt{\hbar}}.
\een
The action of the vertex operator $\widetilde{\Gamma}^+\, \A_t(\hbar,\mathbf{q})$ can be written in the following form:  
\ben
\sum_{
\substack{g\in \ZZ \\ m,m^+,m^-} } 
a^{(g)}_{m+m^+} \, \hbar^{g-1+(|m^+|-|m^-|)/2} \, 
\frac{(\mathbf{q}(z)+z)^m}{m!}\,
\frac{(\mathbf{q}(z)+z)^{m^-}}{m^-!}\,
\frac{(-\mathbf{f}^+(z))^{m^+}}{m^+!}\,
\frac{(\mathbf{f}^-(z))^{m^-}}{m^-!}\,
\een
where the summation over $m^-$ involves only finitely supported functions $m^-$ such that $m^-_{0,0,0}=0$. 
We claim that if we fix $G:=g+(|m^+|-|m^-|)/2$ and $M=m+m^-$, then there are only finitely many choices for $g$, $m$, $m^\pm$ such that the coefficient $a^{(g)}_{m+m^+}$ is not $0$. Indeed, since $M=m+m^-$ is fixed, there are only finitely many choices for $m$ and $m^-$. Since $g+|m^+|/2=G+|m^-|/2$, we see that $g$ is uniquely determined by $m^+$ and $m^-$. It remains only to estimate $m^+$. The tameness of the total ancestor potential implies that if $a^{(g)}_{m+m^+}\neq 0$, then
\ben
0\leq
\sum_{k,a} k (m_{k,a}+m^+_{k,a}) \leq 3g-3 + |m^+|+|m|=
3G-3+\frac{3}{2} |m^-|-\frac{1}{2} |m^+|+|m|.
\een
Since $m$ and $m^-$ for fixed $G$ are bounded, the above estimate shows that $|m^+|$ is bounded, that is, there are only finitely many choices for $m^+$. This proves that we have an expansion 
\beq\label{tvoa}
\widetilde{
\Gamma}^+\, \A_t(\hbar,\mathbf{q})= 
\sum_{G\in \tfrac{1}{2}\ZZ} \sum_{M} 
\gamma^{(G)}_{M}(\mathbf{f}) \, \hbar^{G-1} \, 
\frac{(\mathbf{q}(z)+z)^M}{M!}
\eeq
where the second sum is over all finitely supported functions $M$ and 
\beq\label{gamma}
\gamma^{(G)}_{M}(\mathbf{f}):=\sum_{g, m,m^\pm} 
a^{(g)}_{m+m^+}\,  
\binom{M}{m^-}\, 
\frac{(-\mathbf{f}^+(z))^{m^+}}{m^+!}\,
\frac{(\mathbf{f}^-(z))^{m^-}}{m^-!}\ ,
\eeq
where the sum is over all $g,m,m^\pm$ with $m^-_{0,0,0}=0$ such that $G=g+(|m^+|-|m^-|)/2$ and $M=m+m^-$. This is a finite sum so the coefficient $\gamma^{(G)}_M$ is a polynomial in $I$. For later purposes we will need the following estimate: if $\gamma^{(G)}_M\neq 0$ then
\begin{align}
\label{inf-G}
G\geq \frac{1}{3} \, M_{0,0,0} -\frac{1}{2}\, |M|.
\end{align}
Indeed, let us recall formula \eqref{gamma}. Thanks to the string equation the number of insertions of $1$ in the ancestor correlator can not exceed $(3+\mbox{the degree of the psi classes})$. This yields the following inequality
\ben
m_{0,0,0}+m_{0,0,0}^+\leq 3 +\sum_{k,a} k(m_{k,a}+m^+_{k,a}).
\een
By tameness, that is the degree of the psi-classes does not exceed the dimension of the moduli space, we get that the RHS of the above inequality does not exceed 
$3g+|m|+|m^+|.$ Since $M=m+m^-$ and $m^-_{0,0,0}=0$ we can estimate 
$M_{0,0,0}=m_{0,0,0}=m_{0,0,0}+m^+_{0,0,0}-m_{0,0,0}^+$ in the following way: 
\ben
M_{0,0,0}\leq 3g + |m|+|m^+|-m_{0,0,0}^+ =
3G-\frac{|m^+|}{2} +\frac{3|m^-|}{2}+|m|-m_{0,0,0}^+
\leq 
3G +\frac{3|m^-|}{2}+|m|.
\een
where we used that $g=G-(|m^+|-|m^-|)/2$. Finally, since 
\ben
\frac{3|m^-|}{2}+|m|=
\frac{|m^-|}{2}+|M|\leq \frac{3|M|}{2},
\een
combining with the previous inequality we get  
\eqref{inf-G}. 

Using formula \eqref{tvoa} we get the following expansion:
\beq\label{bsub}
\widetilde{\Gamma}^+\otimes 
\widetilde{\Gamma}^-\, 
(\A_t\otimes \A_t)=
\sum_{g\in \tfrac{1}{2}\ZZ}\sum_{m',m''} 
c^{(g)}_{m',m''}(\mathbf{f}) \,
\hbar^{g-1}
\frac{(\mathbf{q}'(z)+z)^{m'}}{m'!}\,
\frac{(\mathbf{q}''(z)+z)^{m''}}{m''!}\, ,
\eeq
where 
\beq\label{cgmm}
c^{(g)}_{m',m''} (\mathbf{f})=
\sum_{g=G'+G''-1} 
\gamma^{G'}_{m'}(\mathbf{f})
\gamma^{G''}_{m''}(-\mathbf{f}).
\eeq
Thanks to the estimate \eqref{inf-G}, for the non-zero terms in the above sum both $G'$ and $G''$ are bounded from below. Therefore, the sum must be finite, i.e., the coefficient $c^{(g)}_{m',m''}$ is a polynomial in $I$. 

Let us substitute 
$q'_{k,0,0}=
m\sqrt{\hbar}\delta_{k,0}+f_k(\mathbf{q}',\mathbf{q}'')$ in \eqref{bsub}. Using the binomial formula we can express the powers of $q'_{k,0,0}$ as formal power series of the following form: if $k>0$ then
\ben
\frac{(q'_{k,0,0})^{m'_{k,0,0}}}{m'_{k,0,0}!}=
\sum_{n'_k,n''_k}\, \square_{n_k',n_k''}(t)\,
\frac{(\mathbf{q}'(z)+z)^{n'_k}}{n_k'!}
\frac{(\mathbf{q}''(z)+z)^{n''_k}}{n_k''!}
\een
where $\square_{k,n'_k,n''_k}$ is a holomorphic function in $t$ whose explicit value is not important for now, the sum is over all finitely supported functions $n'_k,n''_k:\ZZ_{\geq 0}\times \mathcal{B}\to \ZZ_{\geq 0}$ such that 
\beq\label{mk00}
n'_{k}(i,0,0)=0\quad \forall i\geq 0,\quad 
|n'_k|+|n''_k|= m'_{k,0,0}.
\eeq
Similarly, if $k=0$, then we have 
\ben
\frac{(q'_{0,0,0})^{m'_{0,0,0}}}{m'_{0,0,0}!}=
\sum_{\nu, n'_0,n''_0}\, 
\square_{\nu,n_0',n_0''}(t)\,
\hbar^{\nu/2}\,
\frac{(\mathbf{q}'(z)+z)^{n'_0}}{n_0'!}\,
\frac{(\mathbf{q}''(z)+z)^{n''_0}}{n_0''!}\, ,
\een
where the sum is over all $\nu\in \ZZ_{\geq 0}$ and all finitely supported functions $n'_0,n''_0:\ZZ_{\geq 0}\times \mathcal{B}\to \ZZ_{\geq 0}$ such that 
\beq\label{m000}
n'_{0}(i,0,0)=0\quad \forall i\geq 0,\quad 
\nu+|n'_0|+|n''_0|= m'_{0,0,0}.
\eeq
Therefore, after the substitution 
$q'_{k,0,0}= q_{k,0,0}''+
m\sqrt{\hbar}\delta_{k,0}+f_k(\mathbf{q}',\mathbf{q}'')$,
the RHS of \eqref{bsub} takes the following form: 
\beq\label{tA-poly}
\widetilde{\Gamma}^+\otimes 
\widetilde{\Gamma}^-\,
(\A_t\otimes \A_t)=
\sum_{g\in \tfrac{1}{2}\ZZ}\sum_{M',M''} 
\widetilde{A}^{(G)}_{M',M''}(\mathbf{f}) \,
\hbar^{G-1}\,
\frac{(\mathbf{q}'(z)+z)^{M'}}{M'!}\,
\frac{(\mathbf{q}''(z)+z)^{M''}}{M''!}\, ,
\eeq
where the coefficient 
$\widetilde{A}^{(G)}_{M',M''}(\mathbf{f})$ has the following form:
\ben
\widetilde{A}^{(G)}_{M',M''}(\mathbf{f})=\sum_{g,\nu, m',m'', n',n''} 
\square_{\nu,n',n''}(t)\, 
c^{(g)}_{m',m''}(\mathbf{f})\, ,
\een
where the sum is over all integers $g,\nu\in \ZZ_{\geq 0}$, all finitely supported functions $m',m''$, and all sequences of finitely supported functions $n'=(n'_k)_{k\geq 0}$ and $n''=(n''_k)_{k\geq 0}$ such that \eqref{mk00}--\eqref{m000} and the following constraints hold:
\begin{align}
\label{G}
G& = g+\frac{\nu}{2}\\
\label{M'}
M' & = \widetilde{m}'+n_0'+n_1'+\cdots\\
\label{M''}
M'' & = m''+n_0''+n_1''+\cdots
\end{align}
where $\widetilde{m}'_{k,a}=m_{k,a}$ for $a\neq 0$ 
and  $\widetilde{m}'_{k,0}=0$. We have to prove that there are finitely $g,\nu,m',m'',n',n''$ satisfying the constraints \eqref{mk00}--\eqref{M''} such that $c^{(g)}_{m',m''}\neq 0$. Since $M'$ and $M''$ are fixed, recalling the constraints \eqref{M'} and \eqref{M''}, we get immediately that there are finitely many $\widetilde{m}',m'',n',$ and $n''$. Furthermore, the value of $m'_{k,0,0}$ for $k>0$ is fixed by \eqref{mk00}.
It remains only to prove that there are finitely many choices for $g,\nu$ and $m'_{0,0,0}.$ Using \eqref{m000} and \eqref{G} we get
\beq\label{gG}
g+\frac{m_{0,0,0}'}{2} = G +\frac{1}{2}(|n_0'|+|n_0''|).
\eeq
On the other hand, recalling the definition of $c^{(g)}_{m',m''}$, that is, \eqref{cgmm} and the estimate \eqref{inf-G}, we get
\ben
g=G'+G''-1\geq 
\frac{1}{3}(m'_{0,0,0}+m_{0,0,0}'')
-\frac{1}{2}(|m'|+|m''|)-1.
\een
Note that $|m'|=m_{0,0,0}'+|\widetilde{m}'|$. Therefore, the above inequality is equivalent to 
\ben
g+\frac{m_{0,0,0}'}{2}\geq 
\frac{1}{3}(m'_{0,0,0}+m_{0,0,0}'')
-\frac{1}{2}(|\widetilde{m}'|+|m''|)-1.
\een
The LHS of the above equality, thanks to \eqref{gG}, can be expressed in terms of $G$, $n'_0$, and $n''_0$. Solving the resulting inequality for $m_{0,0,0}'$, we get the following upper estimate: 
\ben
m_{0,0,0}'\leq 3G+3 + 
\frac{3}{2}(
|\widetilde{m}'|+|m''|+|n_0'|+|n_0''|)-m_{0,0,0}''.
\een
All terms on the RHS of the above inequality are either fixed or bounded (for fixed $G,M',M''$). We conclude that $m_{0,0,0}'$ is also bounded. Then \eqref{gG} implies that $g$ is bounded. Finally, thanks to \eqref{G} the integer $\nu$ is also bounded. We get that $\widetilde{A}^{(G)}_{M',M''}$ is a polynomial in $I$ as claimed. 

Under the substitution $q_{0,0,0}'=q_{0,0,0}''+m\sqrt{\hbar} + 
f_0(\mathbf{q}',\mathbf{q}'')$ we have
\ben
e^{(I^{(-1)},\phi_0)(q_{0,0,0}'-q_{0,0,0}'')/\sqrt{\hbar}}=
e^{(I^{(-1)},\phi_0) m} \, e^{f_0(\mathbf{q}',\mathbf{q}'')/\sqrt{\hbar}}. 
\een
The above expression, after exponentiation of the expansion \eqref{suba} for $f_0(\mathbf{q}',\mathbf{q}'')$, clearly has the form \eqref{A-poly}. To finish the proof it remains only to notice that 
\ben
\Gamma^+\otimes \Gamma^-\, (\A_t\otimes \A_t) =
e^{(I^{(-1)},\phi_0)(q_{0,0,0}'-q_{0,0,0}'')/\sqrt{\hbar}}\, 
\widetilde{\Gamma}^+\otimes 
\widetilde{\Gamma}^-\, (\A_t\otimes \A_t)
\een
and that the form of the expansion \eqref{A-poly} does not change if we multiply it by \eqref{tA-poly}.
\qed


Let us recall the series $\mathbf{f}_\alpha(t,\lambda,z)$ introduced in 
Proposition \ref{prop:S-conj}, that is,  
\ben
\mathbf{f}_\alpha(t,\lambda,z):=
S(t,z)\, \mathbf{f}_\alpha(\lambda,z) =:
\sum_{m\in \ZZ} I^{(m)}_\alpha(t,\lambda) (-z)^m,\quad \alpha\in K(\PP^1,\CC),
\een
where 
\beq\label{periods}
I^{(m)}_\alpha(t,\lambda)=\sum_{k=0}^\infty 
(-1)^k S_k(t) I^{(m+k)}_\alpha(\lambda)
\eeq
are the so-called {\em periods} of the Frobenius manifold.
Let us investigate the dependence of the ancestor HQEs \eqref{hqe-anc} on the periods \eqref{periods}. 
As an immediate corollary of Proposition \ref{prop:wave} we get the following corollary.
\begin{corollary}\label{cor:wave}
Let $\alpha\in \Phi$ be any reflection vector, i.e., real elliptic root. Then 
the formal power series 
\ben
\left.
\Gamma^{\alpha}(t,\lambda)\otimes 
\Gamma^{-\alpha}(t,\lambda) \, 
(\A_t\otimes \A_t)\right|_{
\widehat{\mathbf{f}}_{L-1}(t,\lambda)\otimes 1 -1\otimes
\widehat{\mathbf{f}}_{L-1}(t,\lambda) = 2\pi \ii\, m
}
\een
has an expansion of the following form:
\ben
e^{m\, (I^{(-1)}_\alpha(t,\lambda),1)}
\sum_{g\in \tfrac{1}{2}\ZZ}\,
\sum_{m',m''}\, 
A^{g}_{m',m''}(\mathbf{f}_\alpha)
\hbar^{g-1}\,
\frac{(\mathbf{q}'(z)+z)^{m'}}{m'!}\,
\frac{(\mathbf{q}''(z)+z)^{m''}}{m''!},
\een
where the coefficients $A^{g}_{m',m''}$ 
are polynomials in $I_\alpha^{(k)}(t,\lambda)$ $(k\in \ZZ)$. \qed
\end{corollary}

\begin{proposition}\label{prop:coset-poly}
Suppose that $\varphi=n\delta$ is an imaginary affine root. Then the formal power series 
\beq\label{hqe-im}
\left.
\Gamma^{\varphi}(t,\lambda)\otimes 
\Gamma^{-\varphi}(t,\lambda) \, 
L(t,\lambda) \, (\A_t\otimes \A_t)\right|_{
\widehat{\mathbf{f}}_{L-1}(t,\lambda)\otimes 1 -1\otimes
\widehat{\mathbf{f}}_{L-1}(t,\lambda) = 2\pi \ii\, m
}
\eeq
has an expansion of the following form:
\ben
(Qe^{t_{01}})^{-mnl}
\sum_{g\in \tfrac{1}{2}\ZZ}\,
\sum_{m',m''}\, 
L^{g}_{m',m''}(
\mathbf{f}_\varphi,
\mathbf{f}_{\alpha_1},
\dots,
\mathbf{f}_{\alpha_N},C)
\hbar^{g-1}\,
\frac{(\mathbf{q}'(z)+z)^{m'}}{m'!}\,
\frac{(\mathbf{q}''(z)+z)^{m''}}{m''!},
\een
where the coefficients $L^{g}_{m',m''}$ 
are polynomials in $I_\varphi^{(k)}(t,\lambda)$ 
$(k\in \ZZ)$, 
$I_{\alpha_i}^{(k)}(t,\lambda)$ 
$(1\leq i\leq N$, $k\in \ZZ)$, 
$\tfrac{d}{d\nu} I^{(k+\nu)}_{\alpha_i}(t,\lambda)|_{\nu=0}$ 
$(i=N-1,N$, $k\in \ZZ)$, and the central constant (see \eqref{cos-vir-perp})
\beq\label{vir_cc}
C(t,\lambda):=
\frac{1}{2}\, \operatorname{tr}\Big(
\frac{1}{4}+\theta\,\theta^T\Big)\lambda^{-2} + 
\sum_{i=1}^{N-2} W(
\partial_\lambda
\mathbf{f}_{\alpha_i}^+(\lambda,z), 
\partial_\lambda
\mathbf{f}_{\beta_i}^+(\lambda,z)).
\eeq 
\end{proposition}
\proof
Note that we have the following identity which is some variation of the Wick formula:
\ben
e^{\widehat{\mathbf{f}}} :
\widehat{\mathbf{g}}_1 
\widehat{\mathbf{g}}_2:
= 
:e^{\widehat{\mathbf{f}}}\ 
\widehat{\mathbf{g}}_1 
\widehat{\mathbf{g}}_2: + 
\Omega(\mathbf{f}^+,\mathbf{g}_1^-)\, 
:e^{\widehat{\mathbf{f}}}\ 
\widehat{\mathbf{g}}_2: + 
\Omega(\mathbf{f}^+,\mathbf{g}_2^-)\, 
:e^{\widehat{\mathbf{f}}}\
\widehat{\mathbf{g}}_1: + 
\Omega(\mathbf{f}^+,\mathbf{g}_1^-)\, 
\Omega(\mathbf{f}^+,\mathbf{g}_2^-)\, 
e^{\widehat{\mathbf{f}}}.
\een
Let us split $L(t,\lambda)=L_\perp(t,\lambda)+L_0(t,\lambda)$. We claim that the statement of the proposition holds when $L(t,\lambda)$ in \eqref{hqe-im} is replaced by either $L_\perp(t,\lambda)$ or $L_0(t,\lambda)$. The proof in both cases is similar so let us consider only the case when $L$ is replaced by $L_\perp$. Let us recall formula \eqref{cos-vir-perp}. We apply the above Wick formula with
\begin{align*}
\mathbf{f} & =
\mathbf{f}_\varphi(t,\lambda,z)\otimes 1 -1\otimes
\mathbf{f}_\varphi(t,\lambda,z),\\ \mathbf{g}_1 & =\partial_\lambda 
\mathbf{f}_{\alpha_i}(t,\lambda,z)\otimes 1 -1\otimes
\partial_\lambda
\mathbf{f}_{\alpha_i}(t,\lambda,z), \\ 
\mathbf{g}_2 & =\partial_\lambda 
\mathbf{f}_{\beta_i}(t,\lambda,z)\otimes 1 - 1\otimes 
\partial_\lambda
\mathbf{f}_{\beta_i}(t,\lambda,z).
\end{align*}
We get a sum of 4 normally ordered terms. Let us discuss the first one, that is, 
\ben
: e^{\widehat{\mathbf{f}}_\varphi}\otimes 
e^{\widehat{-\mathbf{f}}_\varphi}\, 
(\phi_{\alpha_i}\otimes 1-1\otimes \phi_{\alpha_i})\,
(\phi_{\beta_i}\otimes 1-1\otimes \phi_{\beta_i}):
\een
The above expression can be interpreted as the coefficient in front of $\epsilon_1\epsilon_2$ of the following vertex operator:
\ben
e^{
\widehat{\mathbf{f}}_\varphi + \epsilon_1 
\phi_{\alpha_i}+ \epsilon_2 
\phi_{\beta_i}}\otimes 
e^{
-\widehat{\mathbf{f}}_\varphi - \epsilon_1 
\phi_{\alpha_i}- \epsilon_2 
\phi_{\beta_i}}.
\een
Let us apply the above vertex operator to $\A_t\otimes \A_t$ and substitute $\widehat{\mathbf{f}}_{L-1}(t,\lambda)\otimes 1 -1\otimes
\widehat{\mathbf{f}}_{L-1}(t,\lambda) = 2\pi \ii\, m$. Recalling Proposition \ref{prop:wave} we get 
\ben
e^{m\, (I^{(-1)}_{
\varphi+\epsilon_1
\alpha_i+\epsilon_2
\beta_i},1)}\, 
\sum_{g,m',m''}
A^{(g)}(
\mathbf{f}_{\varphi}+ \epsilon_1
\partial_\lambda\mathbf{f}_{\alpha_i}+ \epsilon_2
\partial_\lambda\mathbf{f}_{\beta_i})
\hbar^{g-1} \, 
\frac{(\mathbf{q}'(z)+z)^{m'}}{(m')!} \, 
\frac{(\mathbf{q}''(z)+z)^{m''}}{(m'')!}.
\een
According to Lemma \ref{le:I-1}, b), $e^{m(I^{(-1)}_\varphi(t,\lambda),1)} = (Qe^{t_{01}})^{-mnl}$. Extracting the coefficient in front of $\epsilon_1\epsilon_2$ we get an expansion in which the coefficients will have the required polynomial dependence. 

The analysis in the remaining 3 cases is similar. First, note that the propagators 
\ben
\Omega(\mathbf{f}^+,\mathbf{g}_1^-) =  
2\Omega(\mathbf{f}_\varphi^+(t,\lambda,z), 
\partial_\lambda 
\mathbf{f}_{\alpha_i}^+(t,\lambda,z)) = 
-2nl\,(I^{(0)}_{\alpha_i}(t,\lambda),1)
\een
and 
\ben
\Omega(\mathbf{f}^+,\mathbf{g}_2^-) =  
2\Omega(\mathbf{f}_\varphi^+(t,\lambda,z), 
\partial_\lambda 
\mathbf{f}_{\beta_i}^+(t,\lambda,z)) = 
-2nl\,(I^{(0)}_{\beta_i}(t,\lambda),1)
\een
have the required polynomial dependence. It remains only to note that $e^{\widehat{\mathbf{f}}}\ 
\widehat{\mathbf{g}}_2$ and 
$e^{\widehat{\mathbf{f}}}\ 
\widehat{\mathbf{g}}_1$
are the coefficients in front of $\epsilon$ of respectively 
$e^{
\widehat{\mathbf{f}}_\varphi + \epsilon 
\phi_{\beta_i}}\otimes 
e^{
-\widehat{\mathbf{f}}_\varphi - \epsilon 
\phi_{\beta_i}}
$ 
and 
$e^{
\widehat{\mathbf{f}}_\varphi + \epsilon 
\phi_{\alpha_i}}\otimes 
e^{
-\widehat{\mathbf{f}}_\varphi - \epsilon 
\phi_{\alpha_i}}.
$ 
To complete the proof it remains only to recall Proposition \ref{prop:wave}.
\qed 

\subsection{Period vectors}
The vector valued functions \eqref{periods} are called the {\em periods} of the Frobenius manifold. Apriori, they are formal Laurent series in $\lambda^{-1/l}$. The series is however convergent because one can check that $I^{(m)}_\alpha(t,\lambda)$ is a solution to the following connection:
\begin{align}
\label{ssc-1}
\nabla^{(m)}_{\partial/\partial t_a} & = 
\frac{\partial}{\partial t_a} + 
(\lambda -E\bullet)^{-1} \phi_a\bullet 
\Big(\theta-m+\tfrac{1}{2}\Big) \\
\label{ssc-2}
\nabla^{(m)}_{\partial/\partial \lambda} & = 
\frac{\partial}{\partial \lambda} - 
(\lambda -E\bullet)^{-1} 
\Big(\theta-m+\tfrac{1}{2}\Big).
\end{align}
This is the so-called {\em second structure} connection of the Frobenius manifold. It is a connection on the trivial bundle $(M\times \CC)'\times H$ where 
\ben
(M\times \CC)':=\{ (t,\lambda)\in M\times \CC\ |\ 
\operatorname{det}(\lambda-E\bullet)\neq 0\}.
\een
Since $\lambda=\infty$ is a regular singular point for \eqref{ssc-2}, we get that the Laurent series $I^{(m)}_E(t,\lambda)$ is convergent. Using the differential equations \eqref{ssc-1}--\eqref{ssc-2}, we extend $I^{(m)}_E(t,\lambda)$ to a multivalued analytic function on $(M\times \CC)'$. In order to keep track of the different branches, we fix a reference point $(t^\circ,\lambda^\circ)\in (M\times \CC)'$ such that the canonical coordinates $u_i(t^\circ)\neq u_j(t^\circ)$ for $i\neq j$ and $\lambda^\circ $ is a positive real number such that $|u_i(t^\circ)|<\lambda^\circ$ for all $1\leq i\leq N$.  

Suppose that $t\in M$ is a generic semi-simple point such that the canonical coordinates are pairwise distinct: $u_i(t)\neq u_j(t)$ for $i\neq j$.
Similarly to formula \eqref{periods}, we can use Givental's R-matrix to construct solutions to the second structure connection near $\lambda=u_i(t)$. Namely, let us define
\ben
I^{(m)}_i(t,\lambda)= \sqrt{2\pi} 
\sum_{k=0}^\infty (-1)^k\, 
R_k(t) e_i(t) \, 
\frac{(\lambda-u_i)^{k-m-\tfrac{1}{2}}}{
\Gamma(k-m+\tfrac{1}{2})
}\, ,
\quad m\in \ZZ,
\een
where $R_k(t)\in \operatorname{End}(H)$ is the operator defined in Section \ref{sec:hgr} and 
$e_i(t)=
\sum_{a\in \mathcal{B}}\, 
\phi_a \, \Psi_{ai}(t)\in H
$ 
are the normalized idempotents of the quantum cup product:
\ben
(e_i(t),e_j(t)) = \delta_{i,j},\quad 
e_i(t)\bullet e_j(t) = \sqrt{\Delta_i(t)} \, \delta_{i,j}\, e_j(t).
\een
It is straightforward to check that $I^{(m)}_i(t,\lambda)$ is a solution to the second structure connection $\nabla^{(m)}$. In particular, if we fix a reference path $C_i$ from $(t^\circ,\lambda^\circ)$ to a neighborhood of $(t,u_i(t))$, then there exists a unique vector $\varphi_i\in H$ such that 
\ben
I^{(m)}_i(t,\lambda) = I^{(m)}_{\varphi_i}(t,\lambda).
\een
Such a vector is called a {\em reflection vector}. Furthermore, we have Saito's formula
\beq\label{saf}
(I^{(0)}_\alpha(t,\lambda),
(\lambda-E\bullet)
I^{(0)}_\beta(t,\lambda)) = (\alpha|\beta),\quad \forall \alpha,\beta\in H.
\eeq
Using the above formula, one can prove easily that 
the analytic continuation along the simple loop corresponding to the reference path $C_i$ transforms $I^{(m)}_\alpha(t,\lambda)$ into $I^{(m)}_{r_{\varphi_i}(\alpha)}(t,\lambda)$ where 
\beq\label{refl}
r_{\varphi_i}(\alpha):= \alpha-(\alpha|\varphi_i) \varphi_i,\quad \alpha\in H.
\eeq
In general, analytic continuation defines a map 
\ben
r:\pi_1((M\times \CC)',(t^\circ,\lambda^\circ))\to 
\operatorname{GL}(H),\quad 
C\mapsto r_C
\een
known as the {\em monodromy representation}.  
\begin{proposition}\label{prop:rv}
The set of all reflection vectors coincides with the real elliptic root system $\Phi$.    
\end{proposition}
\proof
The complete proof of this proposition is rather involved. The statement is essentially equivalent to the so-called {\em refined Dubrovin conjecture} (see \cite{CM2024}). Let us just outline the main steps and give references where further details can be found. It is known that under the mirror symmetry between $\PP^1_a$ and the simple elliptic singularity $f=x_1^{a_1}+x_2^{a_2}+x_3^{a_3}-\tfrac{1}{Q}x_1x_2x_3$, the gamma integral structure of Iritani corresponds to the Milnor lattice of $f$, that is, $K^0(\PP^1_a)\cong H_2(f^{-1}(1);\ZZ)$. This can be proved in the same way as Theorem 12 in \cite{MST2016} (see also Theorems 6.9 and 6.10 in \cite{Iri2011}). By mirror symmetry, the periods $I^{(m)}_\alpha(t,\lambda)$ can be identified with period integrals on the Milnor fibers of the miniversal deformation $f_t$ of $f$. Moreover, up to a sign $(\ |\ )$  coincides with the intersection pairing of the Milnor lattice. Since under mirror symmetry the second structure connection becomes the Gauss-Manin connection, by using Picard--Lefschetz theory (see \cite{AGuV}), we conclude that the set of all reflection vectors corresponds to the set of vanishing cycles. On the other hand, it is easy to check that by using reflections \eqref{refl} with $\varphi_i$ corresponding to the vertices of the quiver on Figure \ref{fig:ep}, any element of $\Phi$ can be transformed into an element of the set $\{1,\gamma_{j,p}(1\leq j\leq 3, 1\leq p\leq a_j-1), L\}$. To complete the proof, it remains only to prove that there exists a set of vanishing cycles corresponding to the vertices of the quiver on Figure \ref{fig:ep} such that the intersection pairing of the vanishing cycles coincides with the intersection pairing of the corresponding vertices of the quiver. The existence of such vanishing cycles follows easily from the results of Gabrielov (see \cite{Ga1974}). \qed

\subsection{Twisted periods}
Formula \eqref{periods} makes sense also for any $m\in \CC$. We will be interested in the case when $m\in \nu+\ZZ$ where $\nu\in (-\tfrac{1}{2},\tfrac{1}{2})$. The corresponding vector valued functions $I^{(m)}_\alpha(t,\lambda)$ will be called {\em $\nu$-twisted} or simply {\em twisted} periods. They were first introduced by Dubrovin in \cite{Du2004}. In the settings of singularity theory, the twisted periods were studied by Givental in \cite{Giv1988}. We follow the exposition of \cite{CM2024}.

The discussion from the previous section remains the same. The only difference is that the reflection vectors and the monodromy representation will depend on $\nu$. Let us  introduce the bilinear pairing 
\beq\label{comp_r}
h_\nu(\alpha,\beta) := (
I^{(\nu)}_\alpha(t,\lambda), (\lambda-E\bullet)
I^{(-\nu)}_\beta(t,\lambda) ) = 
q\langle \alpha,\beta\rangle + q^{-1}\langle\beta,\alpha\rangle ,
\eeq
where $q:=e^{\pi\ii \nu}$ and for the second equality we refer to Theorem 1.4.1 in \cite{CM2024}. Note that the function $(\lambda-u_i)^{k-m-\tfrac{1}{2}}$ involved in the definition of a reflection vector is multi-valued: the analytic continuation around $\lambda=u_i$ changes the value by a factor of $-q^{-2}$. Therefore, the reflection vector $\varphi_i(\nu)$ corresponding to a reference path $C_i$ is uniquely determined up to a factor in $(-q^{-2})^\ZZ$. The analytic continuation along the simple closed loop corresponding to the reference path $C_i$ transforms $I^{(m)}_\alpha(t,\lambda)$ into $I^{(m)}_{r^\nu_{\varphi_i}(\alpha)}(t,\lambda)$ where
\ben
r^\nu_{\varphi_i}(\alpha)= \alpha - q^{-1} h_\nu(\alpha,\varphi_i(-\nu))\, \varphi_i(\nu).
\een
Let us point out that $h_\nu(\varphi_i(\nu),\varphi_i(-\nu))=q+q^{-1}$. We refer to Sections 2.2 and 2.3 in \cite{CM2024} for further details and proves. 

\subsection{Phase form, propagators, and phase series}
\label{sec:ppp}
Let us introduce the following formal Laurent series
\begin{align}
\nonumber
\Omega_{\alpha,\beta}(t,\lambda_1,\lambda_2)  & =
\Omega(\mathbf{f}_\alpha^+(t,\lambda_1,z), \mathbf{f}_\beta(t,\lambda_2,z))
=\\
\label{propagator}
& 
=
\sum_{k=0}^\infty (-1)^{k+1}
(I^{(k)}_\alpha(t,\lambda_1), I^{(-1-k)}_\beta(t,\lambda_2)).
\end{align}
Following physics terminology we refer to $\Omega_{\alpha,\beta}(t,\lambda_1,\lambda_2)$ as {\em propagators}. We always assume that $\lambda_2$ is sufficiently close to $\lambda_1$ so that the reference path for $I^{(-1-k)}_\beta(t,\lambda_2)$ is obtained by composing the reference path for $I^{(k)}_\alpha(t,\lambda_1)$ and the straight segment from $(t,\lambda_1)$ to $(t,\lambda_2)$. 
The propagator is interpreted as a formal Laurent series in $\lambda_1^{-1/l}$. On the other hand, using that $I^{(k)}_\alpha(t,\lambda_1)$ and $I^{(-1-k)}_\beta(t,\lambda_2)$ are solutions to the second structure connection, it is straightforward to prove (see \cite{MilSa}, Section 3.4) that 
\beq\label{d-prop}
d\Omega_{\alpha,\beta}(t,\lambda_1,\lambda_2)=
\mathcal{W}_{\alpha,\beta}(t,\lambda_1,\lambda_2),
\eeq
where $d$ on the LHS is the de Rham differential of the complex manifold $M\times \CC^2$ and 
\begin{align}
\label{phase_form}
\mathcal{W}_{\alpha,\beta}(t,\lambda_1,\lambda_2) = & 
\sum_{i\in \mathcal{B}}  (
I^{(0)}_\alpha(t,\lambda_1), \phi_i \bullet 
I^{(0)}_\beta(t,\lambda_2))\, dt_i + \\
\nonumber
&
+ (
(\lambda_1-E\bullet)
I^{(0)}_\alpha(t,\lambda_1),
I^{(0)}_\beta(t,\lambda_2))\,
   \frac{d(\lambda_1-\lambda_2)}{\lambda_1-\lambda_2} +\\
\nonumber
&
- (
I^{(0)}_\alpha(t,\lambda_1),
I^{(0)}_\beta(t,\lambda_2)
) d\lambda_1
\end{align}
is a multi-valued analytic 1-form on $M\times \CC^2$ with singularities along the hypersurface 
\ben
\{(t,\lambda_1,\lambda_2)\ |\ 
(\lambda_1-\lambda_2)\,
\operatorname{det}(\lambda_1-E\bullet_t) \, 
\operatorname{det}(\lambda_2-E\bullet_t)  =0\}\quad \subset \quad 
M\times \CC^2,
\een
where we wrote $E\bullet_t$ to emphasize that the quantum cup product depends on the parameter $t\in M$. In particular, we have 
\beq\label{prop_der}
\partial_{\lambda_1}
\Omega_{\alpha,\beta}(t,\lambda_1,\lambda_2) =
\frac{1}{\lambda_1-\lambda_2}\, 
\left(
I^{(0)}_\alpha(t,\lambda_1), (\lambda_2-E\bullet) 
I^{(0)}_\beta(t,\lambda_2) \right).
\eeq
The above formula implies that the propagator 
$\Omega_{\alpha,\beta}(t,\lambda_1,\lambda_2)$ is convergent as a Laurent series in $\lambda_1^{-1/l}$ for all $(t,\lambda_1,\lambda_2)$ such that $|\lambda_1|>|\lambda_2|$ and $|\lambda_1|>r(t)$ where 
\ben
r(t):=\operatorname{max}(|u_1(t)|,\dots,|u_N(t)|),\quad
t\in M,
\een
where $u_1(t),\dots,u_N(t)$ is the set of all eigenvalues of $E\bullet_t$, each repeated as many times as its multiplicity. Let $\epsilon>0$ be a sufficiently small positive real number, that is, we require that the disk with center 1 and radius $\epsilon$ does not contain non-trivial $l$-roots of $1$ where $l$ is the order of the automorphism $\tau$. Let us introduce the following open subsets of $M\times \CC^2$:
\begin{align*}
(M\times \CC^2)'_\epsilon & := 
\{(t,\lambda_1,\lambda_2)\ :\ 
|\lambda_1-\lambda_2|<\epsilon\, 
\operatorname{min}\{ 
|\lambda_1-u_i(t)|,\, 
|\lambda_2-u_i(t)|\, (1\leq i\leq N)\}, \\
(M\times \CC^2)''_\epsilon & := 
\{(t,\lambda_1,\lambda_2)\in 
(M\times \CC^2)'_\epsilon\ :\ \lambda_1\neq \lambda_2\}, \\
(M\times \CC^2)^+_\epsilon & := 
\{(t,\lambda_1,\lambda_2)\ :\ 
|\lambda_1-\lambda_2|< \epsilon  
(|\lambda_2| -r(t)) < \epsilon 
(|\lambda_1|-r(t))\}. 
\end{align*}
These domains will appear quite frequently in what follows. Let us discuss shortly their structure and make few suggestions on how to think about them.
The domain $(M\times \CC^2)'_\epsilon$ is obtained by thickening the $\lambda$-direction of $(M\times \CC)'$. Namely, we have a diagonal embedding 
\ben
(M\times \CC)'\to (M\times \CC^2)'_\epsilon,\quad
(t,\lambda)\mapsto (t,\lambda,\lambda)
\een
and the domain $(M\times \CC^2)'_\epsilon$ is a tubular neighborhood of $(M\times \CC)'$ in $M\times \CC^2$. It is also useful to consider the projection map
\beq\label{d_proj}
(M\times \CC^2)'_\epsilon\to (M\times \CC)',\quad
(t,\lambda_1,\lambda_2)\mapsto (t,\lambda_1).
\eeq
This realizes  $(M\times \CC^2)'_\epsilon$ as the total space of a disk fibration over $(M\times \CC)'$: the fiber over $(t,\lambda_1)$ is the intersection of the following $2N$ disks (we suppress the dependence of $u_i$ on $t$)
\ben
\{\lambda_2\in \CC\ |\ 
|\lambda_2-\lambda_1|<\epsilon |\lambda_1-u_i|\},\quad
1\leq i\leq N
\een
and 
\beq\label{disk_i}
\{\lambda_2\in \CC\ |\ 
|\lambda_2-\lambda_1|<\epsilon |\lambda_2-u_i|\},\quad
1\leq i\leq N.
\eeq
Note that the inequality in \eqref{disk_i} defines a disk with center $\tfrac{\lambda_1-\epsilon^2u_i}{1-\epsilon^2}$ and radius 
$\tfrac{\epsilon}{1-\epsilon^2}\, |\lambda_1-u_i|$. The diagonal embedding is a section of the fibration \eqref{d_proj}. The domain $(M\times \CC^2)''_\epsilon$ is obtained from $(M\times \CC^2)'_\epsilon$ by removing the diagonal, that is, in terms of \eqref{d_proj}, it is a fibration of punctured disks. Finally, for the 3rd domain $(M\times \CC^2)^+_\epsilon$, using the triangle inequality, we have 
\ben
|\lambda_i|-r(t) \leq 
|\lambda_i|-|u_j(t)|\leq 
||\lambda_i|-|u_j(t)||\leq 
|\lambda_i-u_j(t)|.
\een
We get the following inclusions:
\ben
(M\times \CC^2)^+_\epsilon\subset 
(M\times \CC^2)''_\epsilon\subset
(M\times \CC^2)'_\epsilon\ .
\een

Let us return to our description of the analytic properties of the propagators. 
Using the triangle inequality, it is straightforward to prove that for every point $(t,\lambda_1,\lambda_2)\in (M\times \CC^2)'_\epsilon$, the line segment between $(t,\lambda_1)$ and $(t,\lambda_2)$ in $M\times \CC$ does not intersect the discriminant: $u_i(t)\neq s\lambda_1 +(1-s)\lambda_2$ for all $s\in [0,1]$ and for all $i\in[1,N]$. The propagator series is convergent for all $(t,\lambda_1,\lambda_2)\in (M\times \CC^2)^+_\epsilon$ and the phase form $\mathcal{W}_{\alpha,\beta}$ is a multi-valued analytic 1-form in the domain $(M\times \CC^2)''_\epsilon$. Therefore, the propagator series extends analytically along any path in $(M\times \CC^2)''_\epsilon$ by the following formula:
\ben
\Omega_{\alpha,\beta}(t,\lambda_1,\lambda_2):= 
\Omega_{\alpha,\beta}(t^\circ,\lambda^\circ_1,\lambda^\circ_2) + 
\int_{(t^\circ,\lambda^\circ_1,\lambda^\circ_2)}^{(t,\lambda_1,\lambda_2)} 
\mathcal{W}_{\alpha,\beta},
\een
where $(t^\circ,\lambda^\circ_1,\lambda^\circ_2)\in (M\times \CC^2)^+_\epsilon$ and the integration path is contained in $(M\times \CC^2)''_\epsilon$. 

\subsubsection{Paths with a transverse direction}
Suppose that $C\subset \CC$ is a smooth oriented path without self-intersections. Let us fix a parametrization of $C$, that is, a smooth embedding $\Lambda:[0,1]\to \CC$ compatible with the orientation. Put
\ben
T^+(C):= \bigcup_{s\in C} \RR_{\geq 0}\, \Lambda'(s).
\een
In other words, $T^+(C)$ is the subcone of $\CC$ consisting of all {\em positively} oriented tangent vectors to $C$. Clearly, $T^+(C)$ is independent of the choice of the parametrization $\Lambda$. We will say that $\delta\in \CC$ is a {\em transverse direction}
for $C$ if $\delta\notin T^+(C)$. Note that the set of all transverse directions for $C$ coincides with
the complement $\CC\setminus T^+(C)$.
\begin{remark}
The notion of transverse is slightly abused here. A transverse direction $\delta$ is allowed to be a tangent vector to $C$ (at some point) as long as it has the
opposite orientation.
\qed
\end{remark}
Let $u_1,\dots,u_N$ be a set of pairwise distinct complex
numbers and $\Lambda_0\in \CC$ be a reference point such that $|\Lambda_0|>|u_j|$ for all $1\leq j\leq N$.
Recall that a simple loop 
$L\subset \CC\setminus{\{u_1,\dots,u_N\}}$ around
$u_j$ based at $\Lambda_0$ is a path consisting of two pieces: a path $C$ from $\Lambda_0$ to some point
$u_j+\delta$ and a circle with radius $|\delta|$ and center
$u_j$, where $|\delta|$ is required to be so small that the points $u_i$ ($i\neq j$) are outside the circle.
The simple loop $L$ is defined as the path traversed by a point traveling along $C$ from $\Lambda_0$ to $u_j+\delta$, going around
$u_j$ in counter-clockwise direction along the
circle, and finally returning back to $\Lambda_0$ along the path $C$. We will refer to $C$ as the {\em tail} of the simple loop $L$.
\begin{definition}\label{def:sl_direction}
Let $L$ be a simple loop around $u_i$ based at $\Lambda_0$. We say that $L$ is a simple loop {\em approaching $u_i$ in a transverse direction} $\delta$ if the following conditions are satisfied:
\begin{enumerate}
\item[(i)]
The tail $C$ of the simple loop $L$ ends at a point
$u_i+\delta$ such that
$\delta$ is a transverse direction for $C$.
\item[(ii)] We have $|\Lambda_0|>|\lambda-\delta|$ for all
$\lambda\in C$.
\qed
\end{enumerate}
\end{definition}
A simple loop $L$ approaching $u_i$ in a transverse direction can be constructed as follows. Let $[u_i,\Lambda_0]$ be the line segment
between $u_i$ and $\Lambda_0$. Let $\delta\in \mathbb{C}$ be such that
$u_i+\delta\in [u_i,\Lambda_0]$ and $|\delta|<|u_j-u_i|$ for all
$j\neq i$. Note that since $|u_i|<|\Lambda_0|$, we have
$|\lambda-\delta|<|\Lambda_0|$ for all $\lambda\in
[u_i+\delta,\Lambda_0]$. If the line segment $[u_i+\delta,\Lambda_0]$ does not contain any of the points $u_j$, then we can simply take the tail of our loop to be $C= [u_i+\delta,\Lambda_0]$. Otherwise, for each $u_j\in [u_i+\delta,\Lambda_0]$, we cut a small piece from the line segment around $u_j$ and replace it with a half-circle avoiding $u_j$. The resulting path $C$ has all the required properties provided the pieces that we have removed are sufficiently small. Note that in particular, the fundamental group
$\pi_1(\CC\setminus{\{u_1,\dots,u_N\}},\Lambda_0)$ is
generated by simple loops with transverse directions. Therefore, it is sufficient to compute the periods of the phase form along simple loops that admit a transverse direction.

\subsubsection{Connection formula and phase periods}
\label{sec:conf}
In this subsection we let $t\in M$ be a generic semi-simple point such that the canonical coordinates $\{u_i(t)\}_{i=1}^N$ are pairwise distinct. Suppose that $\lambda_2$ is very close to $u_i(t)$ and that $\beta$ is a reflection vector. Then the period vectors $I^{(-k-1)}_{\beta}(t,\lambda_2)$ have order of vanishing at $\lambda_2=u_i(t)$ at least $k+\tfrac{1}{2}$. The propagator series \eqref{propagator} can be interpreted as a formal Laurent series  
$\Omega^i_{\alpha,\beta}(t,\lambda_1,\lambda_2)$ in 
$(\lambda_2-u_i(t))^{1/2}$. 
The same argument as above implies that the series is convergent in the domain 
$(M\times \CC^2)^i_\epsilon\subset M\times \CC^2$ 
defined by the following inequalities:
\ben
\Big\{
(t,\lambda_1,\lambda_2)\ :\ 
|\lambda_1-\lambda_2|<
\epsilon(|\lambda_2-u_i(t)|)<
\epsilon (|\lambda_1-u_i(t)|)< \epsilon(|u_j(t)-u_i(t))|\quad
\forall j\neq i\Big\}.
\een
We would like to compare the values of the Laurent series $\Omega_{\alpha,\beta}(t,\Lambda_1,\Lambda_2)$ at a point $(t,\Lambda_1,\Lambda_2)\in (M\times \CC^2)^+_\epsilon$ and 
$\Omega^i_{\alpha,\beta}(t,\lambda_1,\lambda_2)$ at a point $(t,\lambda_1,\lambda_2)\in (M\times \CC^2)^i_\epsilon$. 
\begin{proposition}\label{prop:conf}
 Suppose that $\Lambda_2\in \CC$ satisfies $|\Lambda_2|>r(t)$. 
 Let $C\subset \CC\setminus\{u_1(t),\dots,u_N(t)\}$ be a path with a transverse direction $\delta$ connecting $\Lambda_2$ and $\lambda_2:=u_i(t)+\delta$ and suppose that 
 \ben
 |\delta|<|u_j(t)-u_i(t)|\quad \forall j\neq i.
 \een
 There exists a complex number $\Lambda_1$ sufficiently close to $\Lambda_2$ such that 
\begin{enumerate}
\item[(i)]
The path 
$
\widehat{C}:=\{(t,\Lambda_1-\Lambda_2+\lambda,\lambda)\ |\ \lambda\in C\}$ is contained in $(M\times \CC^2)''_\epsilon$. 
\item[(ii)]
The end points $(t,\Lambda_1,\Lambda_2)$ and $(t,\lambda_1,\lambda_2)$ of $\widehat{C}$ belong respectively to $(M\times \CC^2)^+_\epsilon$ and $(M\times \CC^2)^i_\epsilon$.
\end{enumerate}
For every $\Lambda_1$ satisfying conditions (i) and (ii), 
the following formula holds:
\ben
\Omega^i_{\alpha,\beta}(t,\lambda_1,\lambda_2) =
\Omega_{\alpha,\beta}(t,\Lambda_1,\Lambda_2) + 
\int_{\widehat{C}} \mathcal{W}_{\alpha,\beta},\quad
\forall\ \alpha\in H,
\een
where $\beta\in H$ is the reflection vector corresponding to the composition of $\{t\}\times C$ and the reference path of $\Omega_{\alpha,\beta}(t,\Lambda_1,\Lambda_2).$
\qed
\end{proposition}

Let $L$ be a simple loop in $\CC$ around $u_i(t)$ that approaches $u_i(t)$ in a transverse direction.
We will be interested in loops $L$ whose base point $\Lambda_2$
is such that
$|\Lambda_2|>|u_j(t)|$ for all $1\leq j\leq N$. If $\Lambda_1$
is sufficiently
close to $\Lambda_2$, then the path 
\ben
\widehat{L}:\  (t,\Lambda_1-\Lambda_2+\lambda,\lambda),\quad
\lambda \in L
\een
belongs to $(M\times \CC^2)''_\epsilon$ and hence the integral of the phase form is well defined. Let us fix a reference path in $(M\times \CC^2)^+_\epsilon$ so that the values of the propagators 
$\Omega_{\alpha,\beta}(t,\Lambda_1,\Lambda_2)$ are uniquely
fixed. Let $\varphi$ be the reflection vector corresponding to the reference path and the simple loop
$\{t\}\times L\subset (M\times \CC)'$.  
\begin{proposition}\label{prop:ph_period}
Under the above notation the following formula holds:
\beq\label{discr_period}
\Omega_{\alpha,\beta}(t,\Lambda_1,\Lambda_2) - 
\Omega_{w(\alpha),w(\beta)}(t,\Lambda_1,\Lambda_2) + 
\int_{\widehat{L}} \mathcal{W}_{\alpha,\beta}  = -
2\pi\ii\, (\alpha|\varphi)\, \langle \varphi,\beta\rangle,
\eeq
where $w=r_\varphi$ is the reflection corresponding
to $\varphi$.
\end{proposition}
The proofs of Propositions \ref{prop:conf} and \ref{prop:ph_period} are based on the so-called Painleve property of a semi-simple Frobenius manifold. In the settings of singularity theory, the proof of Proposition \ref{prop:ph_period} can be found in \cite{Mi2019} (see Theorem 2.3). For the most general settings we refer to Theorems 3.23 and 3.28 in \cite{MilSa}. 

\subsubsection{Phase series}\label{sec:phase}
Let $W_{kl}(t)\in\operatorname{End}(H)$ be the linear operators defined by \eqref{W_form_2}. Motivated by Proposition \ref{prop:S-conj}, let us introduce the following series
\ben
W_{\alpha,\beta}(t,\lambda_1,\lambda_2)=
W(
\mathbf{f}^+_\alpha(\lambda_1),
\mathbf{f}^+_\beta(\lambda_2)) =
\sum_{k,l=0}^\infty (-1)^{k+l} (W_{kl}(t)
I^{(l)}_\alpha(\lambda_1), 
I^{(k)}_\beta(\lambda_2)).
\een
We will refer to the above series as the {\em phase series}. Using the properties of $W_{kl}(t)$ (see Proposition \ref{prop:W}), it is straightforward to check that 
\beq\label{prop-phase}
\Omega_{\alpha,\beta}(t,\lambda_1,\lambda_2)=
\Omega_{\alpha,\beta}(\lambda_1,\lambda_2) + 
W_{\alpha,\beta}(t,\lambda_1,\lambda_2).
\eeq
On the other hand, we have an explicit formula for 
$\Omega_{\alpha,\beta}(\lambda_1,\lambda_2)=\log B^{\alpha,\beta}(\lambda_1,\lambda_2)$ -- see formula \eqref{phase_B}. In other words,
\begin{align}
\label{calp}
\Omega_{\alpha,\beta}(\lambda_1,\lambda_2) = & 
-2\pi\ii\, 
\operatorname{rk}(\alpha)
\operatorname{deg}(\beta) + \, 
\operatorname{rk}(\alpha)
\operatorname{rk}(\beta)\,\log Q +\, \\
\notag
& 
(\alpha|\beta)\, 
\log \Big( 1- (\lambda_2/\lambda_1)^{1/l}  \Big) + \,
\sum_{k=1}^{l-1}
(\tau^k(\alpha)|\beta)
\log \Big( 1- \xi^k (\lambda_2/\lambda_1)^{1/l}\Big).
\end{align}
The phase series is apriori a formal Laurent series in $\lambda_1^{-1/l}$. However, thanks to \eqref{prop-phase}, it is convergent and it defines a multi-valued analytic function on $(M\times \CC^2)^+_\epsilon$. In fact, the phase series extends analytically across the diagonal $\lambda_1=\lambda_2$. Indeed, using \eqref{prop-phase}, \eqref{calp}, and \eqref{prop_der} we get 
\ben
\partial_{\lambda_1} W_{\alpha,\beta}(t,\lambda_1,\lambda_2) = 
\frac{1}{\lambda_1-\lambda_2}\Big(I^{(0)}_{\alpha}(t,\lambda_1),(\lambda_2-E\bullet)I^{(0)}_\beta(t,\lambda_2)\Big) - 
\frac{1}{l}\sum_{k=0}^{l-1} 
\frac{(\tau^k(\alpha)|\beta)}{\lambda_1^{1/l}-\xi^k \lambda_2^{1/l}} 
\xi^k \lambda_1^{-1}\lambda_2^{1/l}.
\een
Let us determine the disk of convergence of the RHS as a Laurent series in $\lambda_1^{-1/l}$. Clearly, we must have $|\lambda_1|>|u_i(t)|$ for all $i$. The RHS has a pole at $\lambda_1=\lambda_2$ of order 1. However, using Saito's formula \eqref{saf}, it is straightforward to see that the residue is $0$. In other words, the RHS is analytic at $\lambda_1=\lambda_2$, so $\lambda_2$ is inside the disk of convergence. Since differentiation does not change the disk of convergence, the same must be true for $W_{\alpha,\beta}(t,\lambda_1,\lambda_2)$. 
We get that $W_{\alpha,\beta}(t,\lambda_1,\lambda_2)$ is a multi-valued analytic function on the following domain:
\ben
(M\times \CC^2)^\infty_\epsilon:=
\{ (t,\lambda_1,\lambda_2)\ |\ 
|\lambda_1-\lambda_2|<\epsilon 
(|\lambda_i|-r(t)),\ i=1,2\}.  
\een
The analytic extension of $W_{\alpha,\beta}(t,\lambda_1,\lambda_2)$ can be and will be done via formula \eqref{prop-phase}. We need only to clarify in what sense $\Omega_{\alpha,\beta}(\lambda_1,\lambda_2)$ can be viewed as a function on $(M\times \CC)''_\epsilon$. To begin with, the function 
$\Omega_{\alpha,\beta}(\lambda_1,\lambda_2)$ 
is naturally a multi-valued analytic function on the domain
\beq\label{cal_d}
\{(\lambda_1,\lambda_2)\in \CC^2\ |\ 
0<
|\lambda_1-\lambda_2|<\epsilon |\lambda_i|,\ i=1,2\}.
\eeq
Namely, let us fix a reference point $(\lambda_1^\circ,\lambda_2^\circ)$ in the domain \eqref{cal_d}. For example, put $\lambda_1^\circ=\lambda^\circ$ and $\lambda_2^\circ = (1-\tfrac{\epsilon^2}{2})\lambda^\circ$ where $\lambda^\circ$ is a sufficiently large positive real number -- same as in our choice $(t^\circ,\lambda^\circ)$ for a reference point of $(M\times \CC)'$. 
The value of $\Omega_{\alpha,\beta}(\lambda_1^\circ,\lambda_2^\circ)$ is defined via the principal branch of the logarithm. The value at other points depends on the choice of a reference path. Thanks to our choice of $\epsilon$, all logarithmic terms $\log (1-\xi^k (\lambda_2/\lambda_1)^{1/l})$ for $1\leq k\leq l-1$ are in fact single valued in the domain \eqref{cal_d}. 
Note that if $(t,\lambda_1,\lambda_2)\in (M\times \CC^2)^\infty_\epsilon$ and $\lambda_1\neq \lambda_2$, then $(\lambda_1,\lambda_2)$ belongs to the domain \eqref{cal_d}. Therefore, we can define $\Omega_{\alpha,\beta}(\lambda_1,\lambda_2)$, that is, we can view $\Omega_{\alpha,\beta}(\lambda_1,\lambda_2)$ as a multi-valued analytic function on $(M\times \CC^2)^\infty_\epsilon\setminus{\{\lambda_1\neq \lambda_2\}}$. 

\medskip

\begin{proposition}\label{prop:pan}
a) The phase series $W_{\alpha,\beta}(t,\lambda_1,\lambda_2)$ extends analytically along any path in the domain
\beq\label{phase_d}
\{ (t,\lambda_1,\lambda_2)\in (M\times \CC^2)'_\epsilon\ |\ 
|\lambda_1-\lambda_2|<\epsilon |\lambda_i|,\ i=1,2\}.
\eeq

b) The phase series is symmetric: 
\ben
W_{\alpha,\beta}(t,\lambda_1,\lambda_2)=
W_{\beta,\alpha}(t,\lambda_2,\lambda_1). 
\een

c) The coefficient $b_\alpha(t,\lambda)$ of the ancestor HQEs is a multi-valued analytic function on 
$(M\times \CC)'$ compatible with the monodromy representation: the analytic continuation of $b_{\alpha}(t,\lambda)$ along a closed loop $C$ is $b_{r_C(\alpha)}(t,\lambda)$. 

d) The central constant \eqref{vir_cc} can be expressed explicitly in terms of the periods as follows:
\ben
C(t,\lambda)=\frac{1}{2} \sum_{i=1}^{N-2}(
I^{(1)}_{\alpha_i}(t,\lambda), (\lambda-E\bullet)
I^{(1)}_{\beta_i}(t,\lambda)).
\een
Moreover, $C(t,\lambda)$ is a single valued analytic function on $(M\times \CC)'$.
\end{proposition}
\proof
a) This is a direct consequence of identity \eqref{prop-phase} and the analyticity of the phase series $W_{\alpha,\beta}(t,\lambda_1,\lambda_2)$ at $\lambda_1=\lambda_2$. If we consider only points $(t,\lambda_1,\lambda_2)\in (M\times \CC^2)''_\epsilon$ such that $(\lambda_1,\lambda_2)$ belongs to the domain \eqref{cal_d}, then the analytic extension exist because both $\Omega_{\alpha,\beta}(t,\lambda_1,\lambda_2)$ and $\Omega_{\alpha,\beta}(\lambda_1,\lambda_2)$ can be extended. It remains only to check that the difference $\Omega_{\alpha,\beta}(t,\lambda_1,\lambda_2)-
\Omega_{\alpha,\beta}(\lambda_1,\lambda_2)$ does not have singularities along the diagonal $\lambda_1=\lambda_2$. The singular term in $\Omega_{\alpha,\beta}(\lambda_1,\lambda_2)$ is $(\alpha|\beta)\log (\lambda_1-\lambda_2)$ while the singularity of $\Omega_{\alpha,\beta}(t,\lambda_1,\lambda_2)$ is concentrated in an integral of the form $\int_{(t,\lambda_1',\lambda_2')}^{(t,\lambda_1,\lambda_2)}
\mathcal{W}_{\alpha,\beta}$ where the integration path is in some small neighborhood of a point on the diagonal $\lambda_1=\lambda_2$. We have to check that  
\beq\label{log_sing}
\int_{(t,\lambda_1',\lambda_2')}^{(t,\lambda_1,\lambda_2)} \mathcal{W}_{\alpha,\beta}-
(\alpha|\beta)\log (\lambda_1-\lambda_2)
\eeq
is holomorphic at $\lambda_1=\lambda_2$. This is a local computation. Put $u:=\lambda_1-\lambda_2$ and $\lambda:=\lambda_2$. The singular term in the phase form has the form 
\ben
\frac{du}{u} \, ( (\lambda+u-E\bullet)
I_\alpha^{(0)}(t,\lambda+u),
I^{(0)}_\beta(t,\lambda) ) .
\een
Expanding the integrand into Taylor series at $u=0$, we get 
\ben
\frac{du}{u} \,\Big(
(\alpha|\beta) + u\, \Big(
I^{(0)}_\alpha(t,\lambda)+(\lambda-E\bullet) I^{(1)}_\alpha(t,\lambda)\, ,\, 
I^{(0)}_\beta(t,\lambda)\Big) + O(u^2) \, \Big),
\een
where we used Saito's formula 
$(\alpha|\beta)=((\lambda-E\bullet) I^{(0)}_\alpha(t,\lambda),
I^{(0)}_\beta(t,\lambda))$. Integrating, we get  $(\alpha|\beta)\log u +$ regular terms in $u$, that is, the integral in \eqref{log_sing} has a logarithmic singularity which cancels out with the logarithmic term in \eqref{log_sing}. 

b) 
This is a direct corollary of Proposition \ref{prop:W}, part b). 

c) According to Proposition \ref{prop:be}, we have $b_\alpha(t,\lambda)=b_\alpha(\lambda) e^{W_{\alpha,\alpha}(t,\lambda,\lambda)}$.  The fact that $b_\alpha(t,\lambda)$ extends analytically along any path in $(M\times \CC)'$ is a corollary of a). It remains to prove that the coefficient is compatible with the monodromy representation. More generally, let us prove that 
$
B^{\alpha,\beta}(t,\Lambda_1,\Lambda_2)= e^{
\Omega_{\alpha,\beta}(t,\Lambda_1,\Lambda_2)
}$
is compatible with the monodromy representation for all $\alpha,\beta\in \Phi$. Since the fundamental group $\pi_1((M\times \CC)',(t^\circ,\lambda^\circ))$ is generated by loops with a transversal direction, it is sufficient to prove the compatibility for loops $L$ for which we can apply the formula from Proposition \ref{prop:ph_period}. The analytic continuation along $\widehat{L}=\{(t,\Lambda_1-\Lambda_2+\lambda)\ |\ \lambda\in L\}$ transforms 
$B^{\alpha,\beta}(t,\Lambda_1,\Lambda_2)$ into 
\ben
B^{\alpha,\beta}(t,\Lambda_1,\Lambda_2) 
e^{\int_{\widehat{L}}\mathcal{W}_{\alpha,\beta}} =
B^{r_\varphi(\alpha),r_\varphi(\beta)}(t,\Lambda_1,\Lambda_2)
e^{-2\pi \ii  (\alpha|\varphi)\, 
\langle \varphi,\beta\rangle }=
B^{r_\varphi(\alpha),r_\varphi(\beta)}(t,\Lambda_1,\Lambda_2)
\een
where we used that if $\alpha,\beta,\varphi\in \Phi$, then  $\langle \varphi,\beta\rangle\in \ZZ$ and $(\alpha|\varphi)\in \ZZ$. 

d)  Using the formula for $L(t,\lambda)$ in Proposition \ref{prop:anc_cv}, we get that the central constant is 
\ben
C(t,\lambda)= \sum_{i=1}^{N-2}
\operatorname{Res}_{\lambda'=\lambda}
\partial_{\lambda'}\partial_\lambda
\Omega_{\alpha_i,\beta_i}(t,\lambda',\lambda)\, 
\frac{d\lambda'}{\lambda'-\lambda}\, .
\een
Recalling formula \eqref{prop_der} we get 
\ben
C(t,\lambda)= \sum_{i=1}^{N-2}
\operatorname{Res}_{\lambda'=\lambda}
\partial_\lambda \Big(
\frac{1}{\lambda'-\lambda}\, (
I^{(0)}_{\alpha_i}(t,\lambda'), (\lambda-E\bullet)
I^{(0)}_{\beta_i}(t,\lambda) )\Big)\, 
\frac{d\lambda'}{\lambda'-\lambda}.
\een
Let us expand the following pairing into Taylor series at $\lambda'=\lambda$: 
\ben
(
I^{(0)}_{\alpha}(t,\lambda'), (\lambda-E\bullet)
I^{(0)}_{\beta}(t,\lambda)) =
(\alpha|\beta) +
\sum_{k=1}^\infty 
C^{(k)}_{\alpha,\beta}(t,\lambda)\, 
\frac{(\lambda'-\lambda)^k}{k!}\, ,
\een
where 
\ben
C^{(k)}_{\alpha,\beta}(t,\lambda)= (
I^{(k)}_\alpha(t,\lambda), 
(\lambda-E\bullet)
I^{(0)}_{\beta}(t,\lambda) ).
\een
Straightforward computation yields
\beq\label{C-res}
\operatorname{Res}_{\lambda'=\lambda}
\partial_\lambda \Big(
\frac{1}{\lambda'-\lambda}\, (
I^{(0)}_{\alpha}(t,\lambda'), (\lambda-E\bullet)
I^{(0)}_{\beta}(t,\lambda) )\Big)\, 
\frac{d\lambda'}{\lambda'-\lambda}= 
\partial_\lambda C^{(1)}_{\alpha,\beta}(t,\lambda) -
\frac{1}{2}
C^{(2)}_{\alpha,\beta}(t,\lambda).
\eeq
Using the differential equations of the second structure connection we transform $C_{\alpha,\beta}^{(k)}$ as follows:
\ben
C^{(k)}_{\alpha,\beta}(t,\lambda) = 
((\lambda-E\bullet) I^{(k)}_\alpha(t,\lambda), 
I^{(0)}_\beta(t,\lambda) ) =
((\theta+\tfrac{1}{2}-k) I^{(k-1)}_\alpha(t,\lambda), 
I^{(0)}_\beta(t,\lambda) )\, .
\een
Using the above formula we transform \eqref{C-res} into
\ben
((\theta-\tfrac{1}{2}) I^{(1)}_\alpha(t,\lambda), 
I^{(0)}_\beta(t,\lambda) ) +
((\theta-\tfrac{1}{2}) I^{(0)}_\alpha(t,\lambda), 
I^{(1)}_\beta(t,\lambda) )- 
\frac{1}{2}\,
((\theta-\tfrac{3}{2}) I^{(1)}_\alpha(t,\lambda), 
I^{(0)}_\beta(t,\lambda) ).
\een
Combining the 1st and the 3rd terms we get
\ben
\frac{1}{2}\,
((\theta+\tfrac{1}{2}) I^{(1)}_\alpha(t,\lambda), 
I^{(0)}_\beta(t,\lambda) ) + 
((\theta-\tfrac{1}{2}) I^{(0)}_\alpha(t,\lambda), 
I^{(1)}_\beta(t,\lambda) ).
\een
Using that $\theta^T=-\theta$ and that $(\theta-\tfrac{1}{2})
I^{(0)}(t,\lambda) = (\lambda-E\bullet) I^{(1)}(t,\lambda)$, it is straightforward to find that the above pairings add up to  
$\tfrac{1}{2}(I^{(1)}_\alpha(t,\lambda),(\lambda-E\bullet)I^{(1)}_\beta(t,\lambda))$. 
\qed

\subsection{Analytic interpretation of the ancestor HQEs}
\label{sec:ai_hqe}
We are in position to obtain an analytic interpretation of the ancestor HQEs. Let us expand 
\beq\label{Omega-an}
\Omega_{\rm aff}(t,\lambda)\,
(\A_t\otimes \A_t)|_{
\widehat{\mathbf{f}}_{L-1}(t,\lambda)\otimes 1 -
1\otimes \widehat{\mathbf{f}}_{L-1}(t,\lambda) =
2\pi\ii m }
\eeq
into a formal power series of the following form:
\ben
\sum_{g\in \tfrac{1}{2}\ZZ} \,
\Omega^{(g)}_{m',m''}(t,\lambda) \, 
\hbar^{g-1}
\frac{(\mathbf{q}'(z)+z)^{m'}}{m'!} 
\frac{(\mathbf{q}''(z)+z)^{m''}}{m''!}. 
\een
We would like to prove that $\Omega^{(g)}_{m',m''}(t,\lambda)$ can be expressed analytically via the ancestor polynomials $A^{(g)}_{m',m''}$, $\Lambda^{(g)}_{m',m''}$, and the  following theta series:
\beq\label{theta-a}
\theta_a(\tau,z):=
\sum_{n\in \ZZ} 
e^{\pi\ii \tau (n^2+2an)+2\pi \ii\, z\, (n+a) }=
e^{-\pi \ii \tau a^2}\theta_{a,0}(\tau,z),
\eeq
where $a:=-\tfrac{m}{2l}$ and $\theta_{a,b}(\tau,z)$ is the so-called theta function with characteristics. The above theta series is convergent and it defines an analytic function in two variables for all $(\tau,z)\in \mathbb{H}\times \CC$ where $\mathbb{H}:=\{\tau\in \CC\ |\ \operatorname{Im}(\tau)>0\}$. For more details we refer to Chapter I of Mumford's book \cite{Mu}.  

\medskip

Let us expand the following part of the ancestor HQEs  
\ben
\Big(
\sum_{\alpha\in \Phi_{\rm aff}}
b_\alpha(t,\lambda) 
\Gamma^{\alpha}(t,\lambda)\otimes 
\Gamma^{\alpha}(t,\lambda)\Big)\,
(\A_t\otimes \A_t)|_{
\widehat{\mathbf{f}}_{L-1}(t,\lambda)\otimes 1 -
1\otimes \widehat{\mathbf{f}}_{L-1}(t,\lambda) =
2\pi\ii m }
\een
into formal power series of the following form:
\ben
\sum_{g\in \tfrac{1}{2}\ZZ} \,
\Theta^{(g)}_{m',m''}(t,\lambda) \, 
\hbar^{g-1}
\frac{(\mathbf{q}'(z)+z)^{m'}}{m'!} 
\frac{(\mathbf{q}''(z)+z)^{m''}}{m''!} 
\een
Recalling Corollary \ref{cor:wave} we get  
\ben
\Theta^{(g)}_{m',m''}(t,\lambda)=
\sum_{\alpha\in \Phi_{\rm aff}}
b_\alpha(t,\lambda)\,
e^{m\, (I^{(-1)}_\alpha(t,\lambda),1)}\, 
A^{(g)}_{m',m''}(\mathbf{f}_\alpha(t,\lambda)).
\een
 The affine roots have the form $\alpha=\beta+n\delta$ where $\beta\in \Phi_\perp$ and $n\in \ZZ$. The above sum can be reorganized as follows:
\ben
\sum_{\beta\in \Phi_\perp}\,
\left(
\sum_{n\in \ZZ}
b_{\beta+n\delta}(t,\lambda)\,
e^{m\, (I^{(-1)}_{\beta+n\delta}(t,\lambda),1)}\, 
A^{(g)}_{m',m''}(
\mathbf{f}_\beta(t,\lambda)+n \mathbf{f}_\delta(t,\lambda))
\right).
\een
\begin{lemma}\label{le:I-1}
The following formulas hold.
\begin{enumerate}
\item[a)]
$I^{(-1)}_\delta(\lambda)=l \, 
(\lambda-(2\pi\ii +\log Q)P)$.
\item[b)]
$(I^{(-1)}_\delta(t,\lambda),1)=-l \, 
(2\pi\ii +t_{01}+\log Q)$.
\item[c)]
$W(\mathbf{f}^+_\delta,\mathbf{f}^+_\delta)=t_{01}l^2 $.
\item[d)]
$W(\mathbf{f}^+_\beta,\mathbf{f}^+_\delta)= l\,\Big(
(I^{(-1)}_\beta(\lambda),1) - 
(I^{(-1)}_\beta(t,\lambda),1)\Big).
$
\end{enumerate}
\end{lemma}
\proof
Note that $\operatorname{rk}(\delta)=l$ and $\operatorname{deg}(\delta)=-l$. Part a) is an immediate corollary of Proposition \ref{prop:cp} -- note that $\chi_{j,p}(\delta)=0.$ For part b), first note that by part a) $I^{(0)}_\delta(\lambda)=l$ and $I^{(n)}_\delta(\lambda)=0$ for $n>0$. Then we get
\ben
I^{(-1)}_\delta(t,\lambda) = 
I^{(-1)}_\delta(\lambda)-
S_1(t) I^{(0)}_\delta(\lambda) =
I^{(-1)}_\delta(\lambda) - l t,
\een
where we used that $S_1(t)1 = t$. The above formula, written in more details, takes the following form:
\beq\label{per_delta}
I^{(-1)}_\delta(t,\lambda) = l\Big(
(\lambda-t_{00})\phi_{0,0} -
(2\pi\ii + t_{01}+\log Q)\phi_{0,1} -
\sum_{j=1}^3 \sum_{p=1}^{a_j-1} t_{j,p}\phi_{j,p}
\Big).
\eeq
The formula in part b) is obtained by extracting the $\phi_{0,1}$-component of the above formula. 
Finally, for part c), since $\mathbf{f}^+_\delta(\lambda) = l$, we get
\ben
W(\mathbf{f}^+_\delta,\mathbf{f}^+_\delta) = 
l^2 W(1,1)= l^2 (W_{0,0}(t) 1,1)=
l^2 (S_1(t) 1,1) = t_{01} l^2.
\een
It remains to prove d). Since $\mathbf{f}^+_\delta(\lambda) = l$, we get
\ben
W(\mathbf{f}^+_\beta,\mathbf{f}^+_\delta) = 
l\sum_{k=0}^\infty (-1)^k(W_{0,k} I^{(k)}(\lambda),1)=
l\sum_{k=0}^\infty (-1)^k(S_{k+1} I^{(k)}(\lambda),1)
\een
where we used that $W_{0,k}=S_{k+1}$ -- see Proposition \ref{prop:W}, a). On the other hand, by definition 
\ben
I^{(-1)}_\beta(t,\lambda) = 
I^{(-1)}_\beta(\lambda) + 
\sum_{k=1}^\infty 
(-1)^k S_k(t) I^{(-1+k)}_\beta(\lambda) =
I^{(-1)}_\beta(\lambda) -
\sum_{k=0}^\infty 
(-1)^k S_{k+1}(t) I^{(k)}_\beta(\lambda).
\een
The formula in part d) follows easily from the last two identities.
\qed

\medskip

Combining Lemma \ref{le:I-1}, Proposition \ref{prop:be}, and Lemma \ref{le:bE-splitting} we find 
\ben
b_{\beta+n\delta}(t,\lambda)= 
b_{\beta+n\delta}(\lambda) e^{W(
\mathbf{f}_{\beta+n\delta}^+,
\mathbf{f}_{\beta+n\delta}^+)}=
b_\beta(t,\lambda)\, 
Q^{n^2l^2 +2\operatorname{rk}(\beta)\, nl}\, 
e^{t_{01}\, n^2l^2}\, 
e^{2n W(
\mathbf{f}_{\beta}^+,
\mathbf{f}_{\delta}^+)}\, .
\een
According to Proposition \ref{prop:cp}, $(I^{(-1)}_\beta(\lambda),1)=2\pi\ii\, 
\operatorname{deg}(\beta)-
\operatorname{rk}(\beta)\,\log Q$. Using this formula and 
Lemma \ref{le:I-1}, d), we get
\ben
e^{2n W(
\mathbf{f}_{\beta}^+,
\mathbf{f}_{\delta}^+)} = 
e^{2nl 
\Big(2\pi\ii\,  \operatorname{deg}(\beta)-
\operatorname{rk}(\beta)\,\log Q\Big) -
2nl\, (I^{(-1)}_\beta(t,\lambda),1)}=
Q^{-2nl\, \operatorname{rk}(\beta)}
e^{ -
2nl\, (I^{(-1)}_\beta(t,\lambda),1)},
\een
where we used that $l\, \operatorname{deg}(\beta)\in \ZZ.$ We get the following formula:
\beq\label{delta-period}
b_{\beta+n\delta}(t,\lambda)= 
b_\beta(t,\lambda)\, 
(Qe^{t_{01}})^{n^2l^2}\, 
e^{-2nl\, (I^{(-1)}_\beta(t,\lambda),1)  }\, .
\eeq
Using Lemma \ref{le:I-1}, b), we get 
\ben
e^{m\, (I^{(-1)}_{\beta+n\delta}(t,\lambda),1)  } =
e^{m\, (I_{\beta}(t,\lambda),1)}\,
(Qe^{t_{01})^{-mnl}}.
\een
Combining these computations we get:
\ben
\Theta^{(g)}_{m',m''}(t,\lambda)=
\sum_{\beta\in \Phi_\perp}\,
b_\beta(t,\lambda)\, 
\left(
\sum_{n\in \ZZ}
(Qe^{t_{01}})^{n^2l^2-mnl}
e^{(m-2nl)\, (I^{(-1)}_{\beta}(t,\lambda),1)}\, 
A^{(g)}_{m',m''}(
\mathbf{f}_\beta+n \mathbf{f}_\delta
\right),
\een
where for simplicity we suppressed the dependence of $\mathbf{f}_\beta$ and $\mathbf{f}_\delta$ on $t$ and $\lambda$.
In order to make the connection with the theta series \eqref{theta-a} transparent, let us put 
\beq\label{theta_v}
Qe^{t_{01}} =: e^{\pi\ii \tau/ l^2},\quad 
(I^{(-1)}_\beta(t,\lambda),1)=: -\pi\ii z_\beta /l,\quad
a:= -m/2l.
\eeq
Then the above expression becomes 
\ben
\sum_{\beta\in \Phi_\perp}\,
b_\beta(t,\lambda)\, 
\left(
\sum_{n\in \ZZ}
e^{\pi\ii \tau (n^2+2an) + 2\pi i z_\beta (n+a)}\, 
A^{(g)}_{m',m''}(
\mathbf{f}_\beta+n \mathbf{f}_\delta)
\right).
\een
It is clear that the coefficient $\Theta^{(g)}_{m',m''}(t,\lambda)$ can be expressed as a polynomial in $I^{(n)}_\beta(t,\lambda)$, $I^{(n)}_\delta(t,\lambda)$, and the derivatives $\theta^k_a(\tau,z):=\partial_z^k \theta_a(\tau,z)$ of the theta function. Indeed, let us expand the ancestor polynomial in the following way:
\ben
A^{(g)}_{m',m''}(\mathbf{f}_\beta + n \mathbf{f}_\delta)=
\sum_{k\geq 0} 
A^{(g)}_{m',m'',k}(\mathbf{f}_\beta,\mathbf{f}_\delta) 
(2\pi\ii (n+a))^k,
\een
where $A^{(g)}_{m',m'',k}=0$ for $k\gg 0$. Then 
\beq\label{co-Theta}
\Theta^{(g)}_{m',m''}(t,\lambda)=
\sum_{\beta\in \Phi_\perp}
\sum_{k=0}^\infty
b_\beta(t,\lambda)\, 
\theta^k_a(\tau,z_\beta)\,
A^{(g)}_{m',m'',k}(
\mathbf{f}_\beta(t,\lambda),
\mathbf{f}_\delta(t,\lambda) ),
\eeq
where $\tau$ and $z_\beta$ are multi-valued analytic functions on $(M\times \CC)'$ defined by \eqref{theta_v}.

\medskip

It remains to analyze the Virasoro term in the HQEs, that is, 
\ben
\sum_{n\in \ZZ}
\left.
b_{n\delta}(t,\lambda)
\Gamma^{n\delta}(t,\lambda)\otimes 
\Gamma^{-n\delta}(t,\lambda) \, 
L(t,\lambda) \, (\A_t\otimes \A_t)\right|_{
\widehat{\mathbf{f}}_{L-1}(t,\lambda)\otimes 1 -1\otimes
\widehat{\mathbf{f}}_{L-1}(t,\lambda) = 2\pi \ii\, m
} \ .
\een
Note that 
\ben
b_{n\delta}(t,\lambda)=
b_{n\delta}(\lambda) \, 
e^{n^2 W(
\mathbf{f}^+_\delta,
\mathbf{f}^+_\delta )
} = 
(Q e^{t_{01}})^{n^2l^2}
\een
where we used Lemma \ref{le:I-1}, part c) and formula \eqref{bE}.
Recalling Proposition \ref{prop:coset-poly}, we get that the Virasoro term in the ancestor HQEs has an expansion of the form 
\ben
\sum_{g\in \tfrac{1}{2}\ZZ} \,
\Lambda^{(g)}_{m',m''}(t,\lambda) \, 
\hbar^{g-1}
\frac{(\mathbf{q}'(z)+z)^{m'}}{m'!} 
\frac{(\mathbf{q}''(z)+z)^{m''}}{m''!} 
\een
where 
\ben
\Lambda^{(g)}_{m',m''}(t,\lambda)= 
\sum_{n\in \ZZ} (Qe^{t_{01}})^{n^2l^2-mnl}\, 
L^{(g)}_{m',m''}(
n\mathbf{f}_\delta,
\mathbf{f}_{\alpha_1},\dots,
\mathbf{f}_{\alpha_{N-2}}, C).
\een
Again, let us take the Taylor series expansion
\ben
L^{(g)}_{m',m''}(n \mathbf{f}_\delta,
\mathbf{f}_{\alpha_1},\dots,
\mathbf{f}_{\alpha_{N-2}}, C) =
\sum_{k\geq 0} 
L^{(g)}_{m',m'',k}(\mathbf{f}_\delta,
\mathbf{f}_{\alpha_1},\dots,
\mathbf{f}_{\alpha_{N}}, C) 
(2\pi\ii (n+a))^k.
\een
We have $L^{(g)}_{m',m'',k}=0$ for $k\gg 0$ and 
\beq\label{co-Lambda}
\Lambda^{(g)}_{m',m''}(t,\lambda)=
\sum_{k=0}^\infty
\theta^k_a(\tau,0)\,
L^{(g)}_{m',m'',k}(
\mathbf{f}_\delta(t,\lambda),
\mathbf{f}_{\alpha_1}(t,\lambda),\dots,
\mathbf{f}_{\alpha_{N}}(t,\lambda), C ).
\eeq
\begin{proposition}\label{prop:anal-HQE}
The coefficients $\Omega^{(g)}_{m',m''}(t,\lambda)$ have the following properties:
\begin{enumerate}
\item[a)]
$\Omega^{(g)}_{m',m''}(t,\lambda)$ is a single-valued analytic function on $(M\times \CC)'$. 
\item[b)]
The total ancestor potential is a solution to the ancestor HQEs of $\PP^1_a$ if and only if the coefficients $\Omega^{(g)}_{m',m''}(t,\lambda)$ extend holomorphically across the discriminant. 
\end{enumerate}
\end{proposition}
\proof
a)
Let us first point out that according to Proposition \ref{prop:pan} the central constant $C(t,\lambda)$ is a quadratic polynomial in $I^{(1)}_{\alpha_i}$ ($1\leq i\leq N-2)$.
Since the period vectors $I^{(n)}_\alpha(t,\lambda)$ $(a\in H)$ are multi-valued analytic functions on the complement of the discriminant, by using formulas 
\eqref{co-Theta} and \eqref{co-Lambda}, we get that the function $\Omega^{(g)}_{m',m''}(t,\lambda)=
\Theta^{(g)}_{m',m''}(t,\lambda)+
\Lambda^{(g)}_{m',m''}(t,\lambda)$ is also multi-valued analytic. We need only to prove that it is single valued, that is, it is invariant under the analytic continuation along a simple loop around the discriminant. To begin with, note that the Virasoro term $L(t,\lambda)$ is invariant. This follows from Proposition \ref{prop:anc_cv}. Indeed, let us look at the 1-form in the $\nu$-plane whose residue at $\nu=0$ is $L(t,\lambda)$ -- see Proposition \ref{prop:anc_cv}. Using the reflection formula \eqref{comp_r} and the relations:
\ben
\sum_{i=1}^N h_\nu(\alpha_i,\varphi)\alpha^i= 
\sum_{i=1}^N h_\nu(\varphi,\alpha^i)\alpha_i=
\varphi\quad 
h_\nu(\varphi,\varphi)=q+q^{-1},
\een
where $\varphi\in \Phi$ is the reflection vector corresponding to the simple loop,
it is straightforward to check that the $1$-form is invariant under the analytic continuation. The vertex operators $\Gamma^{n\delta}(t,\lambda)$ are invariant under the analytic continuation because $\delta$ is in the radical of the intersection form.  Finally, it remains to prove that the remaining part of $\Omega_{\rm aff}(t,\lambda)$ is also invariant. The coefficients $b_{\beta}(t,\lambda)$ are compatible with the monodromy representation -- see Proposition \ref{prop:pan}, part c),i.e., if $B$ is a simple closed loop around the discriminant, then the analytic continuation of $b_\beta(t,\lambda)$ along $B$ is $b_{\beta-(\beta|\varphi)\varphi}(t,\lambda)$ where $\varphi$ is the reflection vector (or vanishing cycle) corresponding to $B$. Let us write $\varphi=\alpha+n\delta+k(L-1)$ where $\alpha\in \Phi_\perp$ and $k,l\in \ZZ$ are some integers. Then
\ben
\beta-(\beta|\varphi)\varphi = 
\beta-(\beta|\alpha+n\delta)(\alpha+n\delta) -
(\beta|\alpha)\, k (L-1). 
\een
The affine ancestor Casimir operator $\Omega_{\rm aff}(t,\lambda)$ is invariant under the action of the affine Weyl group, that is, under the reflections $\beta\mapsto \beta- (\beta|\alpha+n\delta)(\alpha +n\delta)$, because such reflections permute the terms in the sum over $\Phi$. Therefore, we need only to prove that the HQEs are invariant under the translation $\beta\mapsto \beta+k (L-1)$. It is easy to see that all coefficients are invariant, that is, $b_{\beta+k(L-1)}(t,\lambda)=b_\beta(t,\lambda)$. For the vertex operators, we have 
\ben
\Gamma^{\beta+k(L-1)}(t,\lambda)\otimes 
\Gamma^{-\beta-k(L-1)}(t,\lambda) = 
e^{
(\widehat{\mathbf{f}}_{L-1}\otimes 1- 1\otimes
\widehat{\mathbf{f}}_{L-1})/\sqrt{\hbar} }\, 
\Gamma^{\beta}(t,\lambda)\otimes 
\Gamma^{-\beta}(t,\lambda).
\een
After the substitution 
$\widehat{\mathbf{f}}_{L-1}\otimes 1- 1\otimes
\widehat{\mathbf{f}}_{L-1}) = 2\pi\ii \, m\sqrt{\hbar}$ the RHS coincides with $\Gamma^{\beta}(t,\lambda)\otimes 
\Gamma^{-\beta}(t,\lambda)$. This completes the proof of part a). 

b) By definition, if the ancestor HQEs hold, then $\Omega^{(g)}_{m',m''}(t,\lambda)$ is a polynomial in $\lambda$. The statement that it extends across the discriminant is obvious. We have to prove the converse. Let us fix $t\in M$ generic such that $u_i(t)\neq u_j(t)$ for all $i\neq j$. The period vectors have singularities only at $\lambda=u_i(t)$ ($1\leq i\leq N$) and $\lambda=\infty.$ In order to prove that $\Omega^{(g)}_{m',m''}(t,\lambda)$ are polynomials in $\lambda$ it is sufficient to prove that $\Omega^{(g)}_{m',m''}(t,\lambda)$ has at most polynomial growth at $\lambda=\infty$. Indeed, if we knew that, then we can subtract a polynomial from $\Omega^{(g)}_{m',m''}(t,\lambda)$ such that the resulting function will be bounded at infinity and hence it must be constant by the Liouville's theorem. The period vectors have at most polynomial growth because they are solutions to an ODE with a Fuchsian singularity at $\lambda=\infty$. By formula \eqref{co-Lambda}, we get that $\Lambda^{(g)}_{m',m''}(t,\lambda)$ has at most polynomial growth at infinity. To prove that $\Theta^{(g)}_{m',m''}(t,\lambda)$ has polynomial growth, let us recall formula \eqref{co-Theta}. The only issue could come from the theta functions 
$\theta^k_a(\tau,z_\beta)$. However, 
\ben
z_\beta= -\frac{l}{\pi\ii}\, (I^{(-1)}_\beta(t,\lambda),1) = 
(I^{(-1)}_\beta(\lambda),1) + 
\sum_{k=1}^\infty (-1)^k
(S_k(t)I^{(k-1)}_\beta(\lambda),1).
\een
According to Proposition \ref{prop:cp}, 
\ben
I^{(-1)}_\beta(\lambda)=
\operatorname{rk}(\beta)\,\lambda + (
2\pi\ii \, \operatorname{deg}(\beta) -
\operatorname{rk}(\beta)\, \log Q) P +
\sum_{j=1}^3\sum_{p=1}^{a_j-1} 
\frac{\lambda^{1-\tfrac{p}{a_j}}}{
1-\tfrac{p}{a_j} }\, 
\chi_{j,p}(E) \,\phi_{j,p}.
\een
Using this formula, we get that $z_\beta$ is bounded near $\lambda=\infty$. Hence all coefficients $\theta^k_a(\tau,z_\beta)$ in \eqref{co-Theta} are also bounded. 
\qed

\section{From ancestors to KdV}\label{sec:an_kdv}

Finally, we are in position to prove the main result -- Theorem \ref{thm:HQE_d}. Thanks to Proposition \ref{prop:HQE_a} and Proposition \ref{prop:anal-HQE}, it is sufficient to prove that the coefficients of the formal power series \eqref{Omega-an} are analytic at $\lambda=u_i(t)$ for all $i=1,2,\dots,N$ where $t\in M$ is a generic semi-simple point. Without loss of generality, we may assume that $t=t^\circ$ is the reference point. Let us fix a reference path in $\CC:=\{t \}\times \CC$ to a point $\lambda_i^\circ$ sufficiently close to $u_i :=u_i(t )$ such that the connection formula in Proposition \ref{prop:S-conj} holds, i.e., if we put $\delta_i:=\lambda_i^\circ-u_i $, then $\delta_i$ is a transverse direction for the reference path. Let $\varphi$ be the reflection vector corresponding to the reference path. 

Our proof follows the strategy from Givental-Milanov \cite{GM2005}. Namely, we use the higher-genus reconstruction formula \eqref{hgr} and the fact that the total descendant potential $\mathcal{D}_{\rm pt}(\hbar,\mathbf{q})$ is a tau-function of the KdV hierarchy. We need two formulations of the KdV hierarchy -- one as a Kac--Wakimoto hierarchy and another one as a reduction of the KP-hierarchy. The necessary background will be given in Section \ref{sec:KdV}. The main steps in the proof are as follows. The vertex operators $\Gamma^\alpha(t,\lambda)\otimes \Gamma^{-\alpha}(t,\lambda)$ in \eqref{Omega-an} corresponding to real roots $\alpha$ split into three types depending on whether $(\alpha|\varphi)=0$, $\pm 1$, or $\pm 2$. In the first case they are holomorphic at $\lambda=u_i $  so we can exclude them from our analyzes. The operators in the second group, will be split into pairs: $\alpha$ will be paired with $r_\varphi(\alpha)=\alpha-(\alpha|\varphi)\varphi$. We will prove that the contribution of each pair is holomorphic at $\lambda=u_i $ because it gives an operator conjugate to the bilinear operator of the KP-reduction form of the KdV hierarchy. Finally, the operators in the 3rd group will be combined with the Virasoro term and we will see that the resulting operator is conjugate to the bilinear operator of the Kac--Wakimoto form of the KdV hierarchy.  

\subsection{Two forms of KdV}\label{sec:KdV}

Let us apply our construction of HQEs to the case of a point. 
Now $H=H^*({\rm pt},\CC)=\CC$, the grading operator is $\theta=0$, and the Iritani's map is $\Psi:\ZZ\to \CC$, $\Psi(n)= \sqrt{2\pi} \, n$. The Euler pairing is $\langle m, n\rangle = mn$ while the intersection pairing is $(m|n)=2mn$. The set of roots, that is, 
$\Phi=\{\varphi\in \ZZ\ |\ (\varphi|\varphi)=2\}$, has only two elements $\pm \alpha$ where $\alpha=1$. In other words, the lattice $K^0({\rm pt})=\ZZ$ is identified with the root lattice of type $A_1$. The period vectors are given by 
\ben
I^{(n)}_\alpha(u,\lambda) =\sqrt{2\pi} \, 
\frac{(\lambda-u)^{-n-\tfrac{1}{2}}}{\Gamma(-n+\tfrac{1}{2})},\quad n\in \ZZ.
\een
The propagator series is 
\beq\label{prop-pt}
\Omega_{\rm pt}(u,\lambda_1,\lambda_2):=
\Omega_{\alpha,\alpha}(u,\lambda_1,\lambda_2)= 
2\log \left(
\frac{
\sqrt{\lambda_1-u} - \sqrt{\lambda_2-u} 
}{
\sqrt{\lambda_1-u} + \sqrt{\lambda_2-u} 
}\right).
\eeq
The coefficients of the Casimir operator (see Proposition \ref{prop:be}) become
\beq\label{kw_coef}
b_{\pm\alpha}(u,\lambda)= 
\operatorname{\lim}_{\lambda'\to \lambda} 
\left(
\frac{
\sqrt{\lambda'-u} - 
\sqrt{\lambda-u} }{
\lambda'-\lambda}
\right)^2\, 
\frac{1}{(\sqrt{\lambda'-u} + \sqrt{\lambda-u})^2 } = \frac{1}{16(\lambda-u)^2}.
\eeq
The Virasoro term has the following form 
\ben
L(u,\lambda) = \frac{1}{4}\, 
: \Big(
\phi_\alpha(u,\lambda)\otimes 1 -1\otimes \phi_\alpha(u,\lambda)\Big)^2: + 
C(u,\lambda),
\een
where $\phi_\alpha(u,\lambda):=\partial_\lambda \widehat{\mathbf{f}}_\alpha(u,\lambda)$ and the central constant is 
\ben
C(u,\lambda)= \frac{1}{4} (
I^{(1)}_\alpha(u,\lambda), (\lambda-u) 
I^{(1)}_\alpha(u,\lambda) ) =  
\frac{1}{8(\lambda-u)^2}. 
\een
The HQEs in this case take the form: we say that a formal power series 
$\mathcal{A}\in \CC[q_0](\!(\hbar)\!)[\![q_1+1,\dots]\!]$ is a solution to the ancestor HQEs of a point if 
\beq\label{KW-pt}
  \Omega^{\rm KW} (u,\lambda) \mathcal{A}\otimes \mathcal{A}\quad 
\mbox{ is regular (i.e. polynomial) in }\quad \lambda,
\eeq
where 
\ben
  \Omega^{\rm KW} (u,\lambda) = 
b_\alpha(u,\lambda) 
\Gamma^\alpha(u,\lambda)\otimes 
\Gamma^{-\alpha}(u,\lambda)+
b_{-\alpha}(u,\lambda) 
\Gamma^{-\alpha}(u,\lambda)\otimes 
\Gamma^{\alpha}(u,\lambda) - 
L(u,\lambda). 
\een
Let us compare the HQEs \eqref{KW-pt} with the HQEs of the KdV hierarchy. 
According to Givental, a formal power series $\mathcal{A}$ is a tau-function of the KdV hierarchy if 
\beq\label{KP-pt}
\frac{1}{\sqrt{\lambda-u}} 
\left(
\Gamma^{\alpha/2}(u,\lambda)\otimes 
\Gamma^{-\alpha/2}(u,\lambda) - 
\Gamma^{-\alpha/2}(u,\lambda)\otimes 
\Gamma^{\alpha/2}(u,\lambda)
\right) \,
\mathcal{A}\otimes \mathcal{A}
\eeq
is regular in $\lambda$. This formulation of the KdV hierarchy is obtained as a reduction of the KP hierarchy (see \cite{Giv2003}). 
\begin{proposition}\label{prop:kdv-kw}
Every solution to the HQEs \eqref{KP-pt} is a solution to the HQEs \eqref{KW-pt}.
\end{proposition}
\proof
For simplicity, let us put $u=0$ and $\Gamma^a(0,\lambda)=\Gamma^a(\lambda)$. The same argument works in general too. Put 
\ben
\Omega^{\rm KdV}(\lambda)=
\frac{1}{\sqrt{\lambda}} 
\left(
\Gamma^{\alpha/2}(\lambda)\otimes 
\Gamma^{-\alpha/2}(\lambda) - 
\Gamma^{-\alpha/2}(\lambda)\otimes 
\Gamma^{\alpha/2}(\lambda).
\right) 
\een
Using formula \eqref{prop-pt}, it is straightforward to prove the following product formula for $\Omega^{\rm KdV}$:
\begin{align*}
\Omega^{\rm KdV}(\lambda') \Omega^{\rm KdV}(\lambda) = &
\iota_{\lambda',\lambda}
\frac{\sqrt{\lambda'}-\sqrt{\lambda}}{\sqrt{\lambda'}\sqrt{\lambda} (\sqrt{\lambda'}+\sqrt{\lambda})}
\sum_{\pm } 
:\Gamma^{\pm\alpha/2}(\lambda')\Gamma^{\pm\alpha/2}(\lambda):\otimes 
:\Gamma^{\mp\alpha/2}(\lambda')\Gamma^{\mp\alpha/2}(\lambda): +\\
& - 
\iota_{\lambda',\lambda}
\frac{\sqrt{\lambda'}+\sqrt{\lambda}}{
\sqrt{\lambda'}\sqrt{\lambda}(
\sqrt{\lambda'}-\sqrt{\lambda}) }
\sum_{\pm } 
:\Gamma^{\pm\alpha/2}(\lambda')\Gamma^{\mp\alpha/2}(\lambda):\otimes 
:\Gamma^{\mp\alpha/2}(\lambda')\Gamma^{\pm\alpha/2}(\lambda): \ ,
\end{align*}
where $\iota_{\lambda',\lambda}$ denotes the Laurent series expansion in the domain $|\lambda'|>|\lambda|$. Note that both sums are symmetric in $\sqrt{\lambda}$ and $\sqrt{\lambda'}$. Let us compute the anti-commutator
\beq\label{ancom}
\Omega^{\rm KdV}(\lambda') \Omega^{\rm KdV}(\lambda) + \Omega^{\rm KdV}(\lambda) \Omega^{\rm KdV}(\lambda') .
\eeq
The first line of the above product formula will contain a factor of the following formal delta function:
\ben
\delta(\sqrt{\lambda'},-\sqrt{\lambda}) = 
(\iota_{\lambda',\lambda} - \iota_{\lambda,\lambda'})
\frac{1}{\sqrt{\lambda'}+\sqrt{\lambda}} = 
\sum_{n\in \ZZ} (\sqrt{\lambda'})^{-n-1} (-\sqrt{\lambda})^n.
\een
Using $\delta(\sqrt{\lambda'},-\sqrt{\lambda})F(\sqrt{\lambda'},\sqrt{\lambda})=
F(-\sqrt{\lambda},\sqrt{\lambda})$, we get 
\ben
\delta(\sqrt{\lambda'},-\sqrt{\lambda})\Gamma^{\pm \alpha/2}(\lambda')=
\delta(\sqrt{\lambda'},-\sqrt{\lambda})\Gamma^{\mp \alpha/2}(\lambda).
\een
Now it is clear that the contribution of the first line of the product formula to the anti-commutator \eqref{ancom} is 
\ben
\frac{\sqrt{\lambda'}-\sqrt{\lambda}}{\sqrt{\lambda'} \sqrt{\lambda}} \,
\delta(\sqrt{\lambda'},-\sqrt{\lambda})\, 2 = 
\frac{4}{\sqrt{\lambda}} \, \delta(\sqrt{\lambda'},-\sqrt{\lambda})
\een
Since the second line of the product formula for $\Omega^{\rm KdV}$ is obtained from the first one by switching $\sqrt{\lambda}\mapsto -\sqrt{\lambda}$, its contribution to the anti-commutator must be 
\ben
-\frac{\sqrt{\lambda'}+\sqrt{\lambda}}{\sqrt{\lambda'} \sqrt{\lambda}} \,
\delta(\sqrt{\lambda'},\sqrt{\lambda})\, 2 = 
-\frac{4}{\sqrt{\lambda}} \, \delta(\sqrt{\lambda'},\sqrt{\lambda}).
\een
Finally, since
\ben
\frac{4}{\sqrt{\lambda}} \, \delta(\sqrt{\lambda'},-\sqrt{\lambda}) -
\frac{4}{\sqrt{\lambda}} \, \delta(\sqrt{\lambda'},\sqrt{\lambda}) = 
-8 \delta(\lambda',\lambda),
\een
we get 
\beq\label{acom}
\Omega^{\rm KdV}(\lambda') \Omega^{\rm KdV}(\lambda) + \Omega^{\rm KdV}(\lambda) \Omega^{\rm KdV}(\lambda') = - 8 \delta(\lambda',\lambda).
\eeq
On the other hand, using the above product formula, we get 
\ben
4\Omega^{\rm KW}(\lambda)= 
\Omega^{\rm KdV}(\lambda)_{(-2)} \Omega^{\rm KdV}(\lambda) =
\operatorname{Res}_{\lambda'=\lambda} 
\frac{d\lambda'\Omega^{\rm KdV}(\lambda') \Omega^{\rm KdV}(\lambda)}{(\lambda'-\lambda)^2}.
\een
Let us decompose $\Omega^{\rm KdV}(\lambda)=\sum_{n\in \ZZ} \Omega^{\rm KdV}_{(n)}\lambda^{-n-1}$. We get 
\beq\label{KWA}
4\Omega^{\rm KW}(\lambda)\mathcal{A} \otimes \mathcal{A}= 
\operatorname{Res}_{\lambda'=\lambda} 
\frac{d\lambda'}{(\lambda'-\lambda)^2} 
\sum_{m,n\in \ZZ}^\infty 
(\lambda')^{-m-1} \lambda^{n}
\Omega^{\rm KdV}_{(m)} \Omega^{\rm KdV}_{(-n-1)} \mathcal{A}\otimes \mathcal{A} . 
\eeq
Suppose that $\mathcal{A}$ is a solution to \eqref{KP-pt}. Then 
$\Omega^{\rm KdV}_{(n)}\mathcal{A}\otimes \mathcal{A}=0$ for all $n\geq 0$. The anti-commutation relations \eqref{acom} are equivalent to 
\ben
\Omega^{\rm KdV}_{(m)} \, \Omega^{\rm KdV}_{(-n-1)} + 
\Omega^{\rm KdV}_{(-n-1)}\, \Omega^{\rm KdV}_{(m)} = - 8 \delta_{m,n}.
\een
The sum on the RHS of \eqref{KWA} takes the following form:
\ben
\Big(
-\frac{8}{\lambda'-\lambda} + 
\sum_{m,n=0}^\infty 
(\lambda')^m \lambda^n
\Omega^{\rm KdV}_{(-m-1)} \Omega^{\rm KdV}_{(-n-1)}\Big) 
\mathcal{A}\otimes \mathcal{A} .
\een
Formula \eqref{KWA} becomes 
\ben
4\Omega^{\rm KW}(\lambda)\, (\mathcal{A}\otimes \mathcal{A}) = 
\sum_{m,n=0}^\infty m 
\Omega^{\rm KdV}_{(-m-1)}\Omega^{\rm KdV}_{(-n-1)}\, 
(\mathcal{A}\otimes \mathcal{A})\, 
\lambda^{m+n-1}.
\een
The RHS is clearly regular in $\lambda$. This proves that $\mathcal{A}$ is a solution to the HQEs \eqref{KW-pt}. 
\qed

According to Givental, the Witten conjecture (see \cite{Wi1991}) proved by Kontsevich (see \cite{Ko1992}) can be formulated as follows. The total descendant potential of a point $\mathcal{D}_{\rm pt}(\hbar,\mathbf{q})$ is a solution to the HQEs \eqref{KP-pt} for all $u\in \CC$. Using Proposition \ref{prop:kdv-kw}, we get that $\mathcal{D}_{\rm pt}(\hbar,\mathbf{q})$ is also a solution to the ancestor HQEs of the point \eqref{KW-pt}.

\begin{remark}
The operator $e^{((u_2-u_1)/z)\sphat}$ transforms a solution of the HQEs \eqref{KP-pt} for $u=u_1$ to a solution corresponding to $u=u_2$.  
Thanks to the string equation $(1/z)\sphat\, \mathcal{D}_{\rm pt}(\hbar,\mathbf{q})=0 $, the total descendant potential of a point is a solution to the HQEs \eqref{KP-pt} for all $u\in \CC$. 
\end{remark}

\begin{remark}
In our notation the Kac--Wakimoto form of the KdV hierarchy is given by the following bilinear equations:
\beq\label{KW-a1}
\operatorname{Res} \lambda d\lambda \, 
  \Omega^{\rm KW} (0,\lambda) \mathcal{A}\otimes \mathcal{A} =0.
\eeq 
More precisely, if $\mathcal{A}\in \CC[q_0,q_1,\dots]$, that is, if $\mathcal{A}$ is a polynomial solution to \eqref{KW-a1}, then according to Kac and Wakimoto $\mathcal{A}$ is a tau-function of the KdV hierarchy, that is, $\mathcal{A}$ satisfies \eqref{KP-pt}. Proposition \ref{prop:kdv-kw} implies that $\mathcal{A}$ is a solution to \eqref{KW-pt}. We get that the inverse of Proposition \ref{prop:kdv-kw} is true for polynomial tau-functions. In other words, although \eqref{KW-a1} seems to be just a corollary of \eqref{KW-pt}, in fact when it comes to polynomial solutions, the two system of bilinear equations are equivalent. It will be interesting to find out if \eqref{KW-pt} is equivalent to \eqref{KP-pt} in the case when $\mathcal{A}$ is a formal power series. 
\end{remark}


\subsection{From ancestors to the KdV reduction of KP}
Suppose that $\alpha$ is a real root such that $(\alpha|\varphi)=1$. 
The goal in this subsection is to prove that the coefficients of the formal
power series 
\beq\label{pm1_pair}
\Big(
b_\alpha(t ,\lambda)\Gamma^\alpha(t ,\lambda)\otimes
\Gamma^{-\alpha}(t ,\lambda)+
b_{r_\varphi(\alpha)}(t ,\lambda)
\Gamma^{r_\varphi(\alpha)}(t ,\lambda)\otimes
\Gamma^{-r_\varphi(\alpha)}(t ,\lambda)
\Big) \mathcal{A}_{t }\otimes \mathcal{A}_{t }
\eeq
are holomorphic at $\lambda=u _i$. As we already mentioned above, we have to conjugate vertex operators by $\widehat{R}(t ):=(R(t ,z))\sphat$. 
Let us introduce the vertex operators
\ben
\Gamma_{\rm pt}^{c}(u_i,\lambda)=
e^{c \widehat{\mathbf{f}^-}_{\rm pt}(u_i ,\lambda)}\, 
e^{c \widehat{\mathbf{f}^+}_{\rm pt}(u_i ,\lambda)},\quad 
c\in \CC,
\een
where
\ben
\mathbf{f}_{\rm pt}(u_i ,\lambda,z) = \sqrt{2\pi}
\sum_{n\in \ZZ} \frac{(\lambda-u_i )^{-n-\tfrac{1}{2}}}{
\Gamma(-n+\tfrac{1}{2})} e_i (-z)^n.
\een
These are the vertex operators of the point but there is a small difference coming from quantization
\ben
(e_i (-z)^{-k-1})\sphat = \sum_{a\in \mathcal{B}}
\frac{1}{\sqrt{\Delta_i}}\, 
\frac{\partial u_i}{\partial t_a} \, 
(\phi^a (-z)^{-k-1})\sphat = 
\frac{1}{\sqrt{\Delta_i \hbar}}\, q_{k,a}\, 
\frac{\partial u_i}{\partial t_a} =
\frac{1}{\sqrt{\Delta_i \hbar}}\, q_{k}(u_i).
\een
In other words, the vertex operators $\Gamma_{\rm pt}^{c}(u_i,\lambda)$ are obtained from the vertex operators of the point by rescaling $\hbar\mapsto \hbar \sqrt{\Delta_i}$ and replacing the scalar variables by $q_0(u_i),q_1(u_i),\dots$. Note that this is exactly what we need to do if we want to construct HQEs for the factors $\mathcal{D}_{\rm pt}(\hbar\sqrt{\Delta_i}, \mathbf{q}(u_i))$ in the higher-genus reconstruction formula \eqref{hgr}.  
\begin{proposition}\label{prop:R-conj_vo}
Let $\varphi$ be the reflection vector fixed above. 
\begin{enumerate}
\item[a)] 
The following formula holds:
\ben
\Gamma^{c\varphi}(t,\lambda) \, 
\widehat{R}(t) = 
e^{\tfrac{c^2}{2}\, V(\mathbf{f}^-_\varphi,\mathbf{f}^-_\varphi)}\, 
\widehat{R}(t)\, 
:e^{c\, \widehat{\mathbf{f}}_{\rm pt}(u_i,\lambda)}:\ ,
\een
where $c\in \CC$ is a scalar and 
\ben
V(\mathbf{f}^-_\varphi,\mathbf{f}^-_\varphi) =
\Omega(
(R(t,z)^{-1} \mathbf{f}^-_\varphi(t,\lambda,z))^+, 
(R(t,z)^{-1} \mathbf{f}^-_\varphi(t,\lambda,z))^-). 
\een
\item[b)] 
The formal series $V(\mathbf{f}^-_\varphi,\mathbf{f}^-_\varphi)$ is a convergent power series in $\lambda-u_i$ and the following formula holds:
\ben
V(\mathbf{f}^-_\varphi,\mathbf{f}^-_\varphi) = -
\lim_{\lambda_1\to \lambda_2}\Big(
\Omega^i_{\varphi,\varphi}(t,\lambda_1,\lambda_2)- 
\Omega_{\rm pt}(u_i,\lambda_1,\lambda_2) )\Big),
\een
where $\Omega_{\rm pt}(u_i,\lambda_1,\lambda_2)$ is the propagator series of the point (see \eqref{prop-pt}). 
\end{enumerate}
\end{proposition}
\proof
The formula in part a) follows from \eqref{VOR} and 
$\mathbf{f}_\varphi(t,\lambda,z)= R(t,z)\mathbf{f}_{\rm pt}(u_i,\lambda,z)$.

b) 
We have the following general relation between propagators
\ben
\Omega(f^+,g^-) - \Omega((R^{-1}f)^+,(R^{-1}g)^-)  = - 
\Omega((R^{-1}f^-)^+,(R^{-1}g^-)^-),
\een
where $f,g\in H[z,z^{-1}]$ and $R=1+R_1 z+R_2z^2+\cdots$ is an upper triangular symplectic transformation. The proof is an elementary exercise and we omit it. Let us apply this formula to our settings, that is, $f=\mathbf{f}_\varphi(t,\lambda_1,z)$, 
$g=\mathbf{f}_\varphi(t,\lambda_2,z)$, and 
$R=R(t,z)$. Recall that $\mathbf{f}_\varphi(t,\lambda,z)= R(t,z)\mathbf{f}_{\rm pt}(u_i,\lambda,z)$ and 
$\Omega_{\rm pt}(u_i,\lambda_1,\lambda_2)=\Omega(
\mathbf{f}^+_{\rm pt}(u_i,\lambda_1,z) ,
\mathbf{f}^-_{\rm pt}(u_i,\lambda_2,z) )$.  We get 
\ben
\Omega^i_{\varphi,\varphi}(t,\lambda_1,\lambda_2) -
\Omega_{\rm pt}(u_i,\lambda_1,\lambda_2)
 = - \Omega(
(R^{-1} \mathbf{f}^-_\varphi(t,\lambda_1,z))^+, 
(R^{-1} \mathbf{f}^-_\varphi(t,\lambda_2,z))^-),
\een
where we suppressed the dependence of $R(t,z)$ on $t$ and $z$. 
The LHS is the difference of the propagator series of $\PP^1_a$ and the propagator series of the point. We have to check that this difference extends analytically across the diagonal $\lambda_1=\lambda_2$. 
Let us denote the symplectic pairing on the RHS by $V^i_{\varphi,\varphi}(t,\lambda_1,\lambda_2)$. We have 
\beq\label{phase-ui}
V^i_{\varphi,\varphi}(t,\lambda_1,\lambda_2)=
\sum_{k,l=0}^\infty 
(\widetilde{V}_{kl}(t) 
I^{(-l-1)}_\varphi(t,\lambda_1), 
I^{(-k-1)}_\varphi(t,\lambda_2) ), 
\eeq
where 
$ 
\widetilde{V}_{kl}(t)=
\sum_{i=0}^l (-1)^{i+1} R_{k+i+1}(t) R_{l-i}^T(t).
$
Let us think of $V^i_{\varphi,\varphi}(t,\lambda_1,\lambda_2)$ as a Taylor series in $(\lambda_1-u_i)^{\tfrac{1}{2}}$. Using formula \eqref{prop_der}, we get that the derivative
$\partial_{\lambda_1} V^i_{\varphi,\varphi}(t,\lambda_1,\lambda_2)$ coincides with 
\ben
\frac{1}{\lambda_1-\lambda_2}\Big(
I^{(0)}_{\rm pt}(u_i,\lambda_1),(\lambda_2-u_i) 
I^{(0)}_{\rm pt}(u_i,\lambda_2) \Big) - 
\frac{1}{\lambda_1-\lambda_2}\Big(
I^{(0)}_{\varphi}(t,\lambda_1),(\lambda_2-E\bullet) 
I^{(0)}_{\varphi}(t,\lambda_2) \Big).
\een
Note that each term in the above difference has a pole of order $1$ at $\lambda_1=\lambda_2$ of residue $2$ (thanks to Saito's formula \eqref{saf}). Therefore, the difference, as a series in $(\lambda_1-u_i)^{\tfrac{1}{2}}$ is convergent at $\lambda_1=\lambda_2$. Since differentiation does not change the radius of convergence, we conclude that $\lambda_1=\lambda_2$ is in the disk of convergence of the series $V^i_{\varphi,\varphi}(t,\lambda_1,\lambda_2)$. 
\qed

\medskip
Let us recall the propagator series $\Omega^i_{\alpha,\varphi}(t,\lambda_1,\lambda_2)$ introduced in Section \ref{sec:conf}.
\begin{proposition}\label{prop:fac}
Suppose that $\alpha\in H$ is an arbitrary vector  and let
$\alpha':=\alpha-(\alpha|\varphi)\varphi/2$. 
\begin{enumerate}
\item[a)] 
The propagator series 
$\Omega^i_{\alpha',\varphi}(t,\lambda_1,\lambda_2)$ extends analytically across the diagonal $\lambda_1=\lambda_2$ for $\lambda_1$ and $\lambda_2$ sufficiently close to $u_i(t)$. 
\item[b)] 
The following formula holds:
\ben
\Gamma^\alpha(t ,\lambda) =
e^{-c\, \Omega^i_{\alpha',\varphi}(t ,\lambda,\lambda)}
\Gamma^{\alpha'}(t ,\lambda)
\Gamma^{c\, \varphi}(t ,\lambda),
\een
where $c=(\alpha|\varphi)/2.$
\end{enumerate}
\end{proposition}
\proof
a) Let us estimate the radius of convergence of the propagator series $\Omega_{\alpha',\varphi}^i(t,\lambda_1,\lambda_2)$ as a Laurent series in $(\lambda_2-u_i)^{1/2}$. Using formulas \eqref{d-prop}--\eqref{phase_form}, we get that the derivative with respect to $\lambda_2$ is  
\ben
\partial_{\lambda_2}\, 
\Omega_{\alpha',\varphi}^i(t,\lambda_1,\lambda_2) =
-\frac{1}{\lambda_1-\lambda_2}\, 
((\lambda_1-E\bullet)
I^{(0)}_{\alpha'}(t,\lambda_1),
I^{(0)}_{\alpha'}(t,\lambda_2) ).
\een
The RHS extends analytically across $\lambda_2=\lambda_1$ because thanks to Saito's formula the coefficient in front of $(\lambda_1-\lambda_2)^{-1}$ is $(\alpha'|\varphi)=0$. Since the circle of convergence of the Laurent series must contain a point at which the series does not extend analytically, we conclude that $\lambda_1$ is in the interior of the disk of convergence, i.e., the radius of convergence is $>|\lambda_1-u_i|$. The radius of convergence does not change when we differentiate so we conclude that the propagator series $\Omega_{\alpha',\varphi}^i(t,\lambda_1,\lambda_2)$ also has radius of convergence $>|\lambda_1-u_i|$, that is, the propagator series is convergent at $\lambda_2=\lambda_1$. This is exactly what we had to prove.

b)
We have the following factorization:
\ben
\Gamma^\alpha(t ,\lambda) =
e^{-\Omega(\mathbf{f}_{\alpha'}^+, \mathbf{f}_{\varphi})\,
(\alpha|\varphi)/2}
\Gamma^{\alpha'}(t ,\lambda)
\Gamma^{(\alpha|\varphi)\varphi/2}(t ,\lambda),
\een
where the propagator 
\beq\label{split-prop}
\Omega(\mathbf{f}_{\alpha'}^+(t ,\lambda),
\mathbf{f}_{\varphi}(t ,\lambda))=\sum_{k=0}^\infty (-1)^{k+1} (
I^{(k)}_{\alpha'}(t ,\lambda), 
I^{(-k-1)}_{\varphi}(t ,\lambda))
\eeq
is interpreted via the Laurent
series expansion of the period vectors at $\lambda=u _i$. Here the condition $(\alpha'|\varphi)=0$ is crucial because it forces the period vectors $I^{(k)}_{\alpha'}(t ,\lambda)$ to be holomorphic at $\lambda=u_i $ and therefore the Laurent series is convergent at least in the formal $(\lambda-u_i)$-adic sense. It remains only to notice that the series on the RHS of \eqref{split-prop} is obtained from the propagator series $\Omega^i_{\alpha',\varphi}(t ,\lambda_1,\lambda_2)$ by substituting $\lambda_1=\lambda_2=\lambda$. According to a), the RHS of \eqref{split-prop} is a convergent series and its value coincides with $\Omega^i_{\alpha',\varphi}(t,\lambda,\lambda)$. 
\qed

\medskip

\begin{proposition}\label{prop:loc}
The propagators have the following symmetry:
\ben
\Omega_{\alpha,\beta}(t,\lambda_1,\lambda_2)-
\Omega_{\beta,\alpha}(t,\lambda_2,\lambda_1)= - 
2\pi\ii\langle \alpha,\beta\rangle\, ,
\quad 
\forall \alpha,\beta\in H,
\een
where the reference path for the second propagator on the LHS is obtained from the reference path of the first propagator and a small arc $\gamma$ in the $(\lambda_1,\lambda_2)$-space around the diagonal $\lambda_1=\lambda_2$ connecting the points $(\lambda_1,\lambda_2)$ and $(\lambda_2,\lambda_1)$ such that the vector $\lambda_1-\lambda_2$ rotates anti-clockwise around $0$  as $(\lambda_1,\lambda_2)$ moves along $\gamma$.
\end{proposition}
\proof
Since $\Omega_{\alpha,\beta}(t,\lambda_1,\lambda_2)=
\Omega_{\alpha,\beta}(\lambda_1,\lambda_2)+
W_{\alpha,\beta}(t,\lambda_1,\lambda_2)$ and the phase series is symmetric: 
$
W_{\alpha,\beta}(t,\lambda_1,\lambda_2)=
W_{\beta,\alpha}(t,\lambda_2,\lambda_1),
$
we get 
\ben
\Omega_{\alpha,\beta}(t,\lambda_1,\lambda_2)-
\Omega_{\beta,\alpha}(t,\lambda_2,\lambda_1)=
\Omega_{\alpha,\beta}(\lambda_1,\lambda_2)-
\Omega_{\beta,\alpha}(\lambda_2,\lambda_1).
\een
The above difference can be computed using the explicit formula \eqref{calp}. Indeed, let us rewrite \eqref{calp} as follows:  
\begin{align}
\label{calp_2}
\Omega_{\alpha,\beta}(\lambda_1,\lambda_2) = & 
-2\pi\ii\, 
\operatorname{rk}(\alpha)
\operatorname{deg}(\beta) + \, 
\operatorname{rk}(\alpha)
\operatorname{rk}(\beta)\,\log Q +\, \\
\notag
& 
(\alpha|\beta)\, 
\log \Big( \lambda_1^{1/l}- \lambda_2^{1/l}  \Big) + \,
\sum_{k=1}^{l-1}
(\tau^k(\alpha)|\beta)
\log \Big( \lambda_1^{1/l}- \xi^k \lambda_2^{1/l}\Big),
\end{align}
where we used that the difference between the above formula and \eqref{calp} is
\ben
\frac{1}{l}\sum_{k=0}^{l-1} (\tau^k(\alpha)|\beta)  \log \lambda_1 = 
(\pi_0(\alpha)|\beta)\, \log \lambda_1 = 0. 
\een
Note that with our choice of the arc $\gamma$ connecting $(\lambda_1,\lambda_2)$ and $(\lambda_2,\lambda_1)$, since $\lambda_2^{1/l} - \xi^{-k} \lambda_1^{1/l} = -\xi^{-k} (\lambda_1^{1/l} - \xi^{k} \lambda_2^{1/l})$, we must have 
\ben
\log (\lambda_2^{1/l} - \xi^{-k} \lambda_1^{1/l}) =
\log( \lambda_1^{1/l} - \xi^{k} \lambda_2^{1/l}) + 
\pi \ii -\frac{2\pi\ii\, k}{l}  \, .
\een
We get 
\ben
\Omega_{\alpha,\beta}(\lambda_1,\lambda_2) - 
\Omega_{\beta,\alpha}(\lambda_2,\lambda_1) = 
\frac{2\pi\ii}{l}
\sum_{k=1}^{l-1} k (\tau^k(\alpha)|\beta)
-2\pi\ii \Big(
\operatorname{rk}(\alpha)\, 
\operatorname{deg}(\beta) - 
\operatorname{rk}(\beta)\, 
\operatorname{deg}(\alpha) \Big) \, .
\een
Using the identity 
\ben
\sum_{k=1}^{l-1} k x^k (x-1)^2 = l x^l(x-1)-(x^l-1) x
\een
we get 
\ben
\sum_{k=1}^{l-1} k (\tau^k(\alpha)|\beta) = 
\sum_{k=1}^{l-1} k (\tau^k (\tau-1)^2 (\tau-1)^{-2}\pi_\perp(\alpha)|\beta) = 
l ((\tau-1)^{-1}\pi_\perp(\alpha)|\beta).
\een
Combining the above formula with Proposition \ref{prop:rad}, we get 
\beq\label{prop_c}
\Omega_{\alpha,\beta}(\lambda_1,\lambda_2) - 
\Omega_{\beta,\alpha}(\lambda_2,\lambda_1) = 
-2\pi\ii \langle \pi_\perp(\alpha),\beta\rangle 
-2\pi\ii \Big(
\operatorname{rk}(\alpha)\, 
\operatorname{deg}(\beta) - 
\operatorname{rk}(\beta)\, 
\operatorname{deg}(\alpha) \Big) \, ,
\eeq
where $\pi_\perp: K(\PP^1_a,\CC)\to K_\perp(\PP^1_a,\CC)$ is the projection corresponding to the spectral decomposition 
$K(\PP^1,\CC)= K_0(\PP^1,\CC)\oplus K_\perp(\PP^1,\CC)$ -- see Proposition \ref{prop:sd}. Let us denote by $\pi_0: K(\PP^1_a,\CC)\to K_0(\PP^1_a,\CC)$ the complimentary projection. It is straightforward to check, using Proposition \ref{prop:sd}, part a) and the explicit form for the Euler pairing (see Figure \ref{fig:ep}), that 
\ben
\pi_0(\alpha)= 
\operatorname{rk}(\alpha)\, (\delta/l + L-1) + 
\operatorname{deg}(\alpha) \, (L-1). 
\een
Note that by Proposition \ref{prop:rad}, we have $\langle \pi_0(\alpha),\pi_\perp(\beta)\rangle= 0. $ Therefore, 
\ben
\langle 
\pi_0(\alpha), \beta\rangle = 
\langle 
\pi_0(\alpha), \pi_0(\beta)\rangle = 
\operatorname{rk}(\alpha)\, \operatorname{deg}(\beta) - 
\operatorname{rk}(\beta)\, \operatorname{deg}(\alpha),
\een
where we used that $\langle \delta,L-1\rangle = - \langle L-1,\delta\rangle =l$ and $\langle \delta,\delta\rangle = \langle L-1,L-1\rangle =0$. Comparing the above formula with \eqref{prop_c} we get that the RHS of \eqref{prop_c} coincides with $-2\pi\ii \langle \alpha,\beta\rangle$. This completes the proof of the proposition.
\qed

\medskip

Let us proof the regularity of \eqref{pm1_pair}. According to Proposition \ref{prop:fac} we have
\ben
\Gamma^{\alpha}(t ,\lambda)\otimes 
\Gamma^{-\alpha}(t ,\lambda)=
e^{-
\Omega^i_{\alpha',\varphi}(t ,\lambda,\lambda)}\,
\Big(\Gamma^{\alpha'}(t ,\lambda)\otimes 
\Gamma^{-\alpha'}(t ,\lambda)\Big)\, 
\Big(\Gamma^{\varphi/2}(t ,\lambda)\otimes 
\Gamma^{-\varphi/2}(t ,\lambda)\Big)\, ,
\een
where we used that $(\alpha|\varphi)=1$. Furthermore, recalling Proposition \ref{prop:R-conj_vo}, we get 
\ben
\Big(\Gamma^{\varphi/2}(t ,\lambda)\otimes 
\Gamma^{-\varphi/2}(t ,\lambda)\Big) \, 
\widehat{R}(t )\otimes 
\widehat{R}(t ) = 
e^{\tfrac{1}{4}
V(\mathbf{f}^-_\varphi, \mathbf{f}^-_\varphi)}\, 
\widehat{R}(t )\otimes 
\widehat{R}(t )\, 
\Big(
\Gamma^{1/2}_{\rm pt}(u_i ,\lambda)\otimes 
\Gamma^{-1/2}_{\rm pt}(u_i ,\lambda)\Big).
\een
Using the above formulas we can write the operator
\ben
b_\alpha(t ,\lambda)\, 
\Big(
\Gamma^{\alpha}(t ,\lambda)\otimes 
\Gamma^{-\alpha}(t ,\lambda)\Big)\, 
\Big(\widehat{R}(t )\otimes 
\widehat{R}(t )\Big)
\een
as follows
\ben
\Big(\Gamma^{\alpha'}(t ,\lambda)\otimes 
\Gamma^{-\alpha'}(t ,\lambda)\Big)\, 
\Big(\widehat{R}(t )\otimes 
\widehat{R}(t )\Big)\, 
\Big( c_+(t ,\lambda)\,
\Gamma^{1/2}_{\rm pt}(u_i ,\lambda)\otimes 
\Gamma^{-1/2}_{\rm pt}(u_i ,\lambda)\Big)
\een
where 
\ben
c_+(t ,\lambda) = b_\alpha(t ,\lambda)\,
e^{-
\Omega^i_{\alpha',\varphi}(t ,\lambda,\lambda) + 
\tfrac{1}{4}
V(\mathbf{f}^-_\varphi, \mathbf{f}^-_\varphi)
}\, .
\een
The coefficient $c_+(t ,\lambda)$ can be expressed in terms of the propagators. Namely, put $\lambda_2:=\lambda$. Then the coefficient is obtained by taking the limit as $\lambda_1\to \lambda_2$ of the following function
\beq\label{cp12}
(\lambda_1-\lambda_2)^{-(\alpha|\alpha)}
\exp\Big( 
\Omega_{\alpha,\alpha}(t ,\lambda_1,\lambda_2) - 
\Omega_{\alpha',\varphi}(t ,\lambda_1,\lambda_2) -
\Omega_{\varphi/2,\varphi/2}(t ,\lambda_1,\lambda_2) + 
\tfrac{1}{4}
\Omega_{\rm pt}(u_i ,\lambda_1,\lambda_2)\Big),
\eeq
where we used Propositions \ref{prop:be} and \ref{prop:R-conj_vo} to express $b_\alpha(t ,\lambda)$ and $V(\mathbf{f}^-_\varphi, \mathbf{f}^-_\varphi)$ in terms of propagators. Note that we have also used the connection formula from Proposition \ref{prop:conf} to identify 
\ben
\Omega^i_{\alpha',\varphi}(t ,\lambda_1,\lambda_2)=
\Omega_{\alpha',\varphi}(t ,\lambda_1,\lambda_2)
\een
and 
\ben
\Omega^i_{\varphi/2,\varphi/2}(t ,\lambda_1,\lambda_2) =
\Omega_{\varphi/2,\varphi/2}(t ,\lambda_1,\lambda_2).
\een
Using that $\alpha=\alpha'+\varphi/2$, we get that the sum of the propagators in the exponent in \eqref{cp12} coincides with 
\ben
\Omega_{\alpha',\alpha'}(t ,\lambda_1,\lambda_2)+
\frac{1}{2} (
\Omega_{\varphi,\alpha'}(t ,\lambda_1,\lambda_2) -
\Omega_{\alpha',\varphi}(t ,\lambda_1,\lambda_2)) +
\frac{1}{4}
\Omega_{\rm pt}(u_i ,\lambda_1,\lambda_2).
\een
Using $(\alpha|\alpha)=(\alpha'|\alpha')+(\varphi/2|\varphi/2)$ and Proposition \ref{prop:loc} we get 
\beq\label{cpt}
c_+(t ,\lambda)= 
e^{\pi\ii \langle \alpha',\varphi\rangle}\ 
\lim_{\lambda_1\to \lambda_2} 
\frac{e^{\Omega_{\alpha',\alpha'}(t ,\lambda_1,\lambda_2)}}{
(\lambda_1-\lambda_2)^{(\alpha'|\alpha')}}\ 
\lim_{\lambda_1\to \lambda_2} 
\frac{e^{\tfrac{1}{4}\Omega_{\rm pt}(u_i ,\lambda_1,\lambda_2)}}{
(\lambda_1-\lambda_2)^{(\varphi/2|\varphi/2)}},
\eeq
where $\lambda_2:=\lambda.$ Let us denote the first limit by 
$b_{\alpha'}(t ,\lambda)$. Since the derivative of the propagators are polynomial expressions in the period vectors and the period vectors $I^{(k)}_{\alpha'}(t ,\lambda)$ are analytic at $\lambda=u_i $, we get that $b_{\alpha'}(t ,\lambda)$ is holomorphic at $\lambda=u_i $. 
The second limit can be computed explicitly
\ben
\lim_{\lambda_1\to \lambda_2} 
\frac{e^{\tfrac{1}{4}\Omega_{\rm pt}(u_i ,\lambda_1,\lambda_2)}}{
(\lambda_1-\lambda_2)^{(\varphi/2|\varphi/2)}} = 
\frac{1}{2\sqrt{\lambda-u_i }}\, .
\een
We get 
\ben
c_+(t ,\lambda)= 
b_{\alpha'}(t ,\lambda)\, 
\frac{
e^{\pi\ii \langle \alpha',\varphi\rangle}
}{2\sqrt{\lambda-u_i }}\, .
\een
Similarly, let us repeat the above computation with $r_\varphi(\alpha)=\alpha'-\varphi/2$ instead of $\alpha$. We have to replace everywhere $\varphi$ with $-\varphi$ and also note that 
\ben
\Big(\Gamma^{-\varphi/2}(t ,\lambda)\otimes 
\Gamma^{\varphi/2}(t ,\lambda)\Big) \, 
\widehat{R}(t )\otimes 
\widehat{R}(t ) = 
e^{\tfrac{1}{4}
V(\mathbf{f}^-_\varphi, \mathbf{f}^-_\varphi)}\, 
\widehat{R}(t )\otimes 
\widehat{R}(t )\, 
\Big(
\Gamma^{-1/2}_{\rm pt}(u_i ,\lambda)\otimes 
\Gamma^{1/2}_{\rm pt}(u_i ,\lambda)\Big).
\een
The operator
\ben
b_{r_\varphi(\alpha)}(t ,\lambda)\, 
\Big(
\Gamma^{r_\varphi(\alpha)}(t ,\lambda)\otimes 
\Gamma^{-r_\varphi(\alpha)}(t ,\lambda)\Big)\, 
\Big(\widehat{R}(t )\otimes 
\widehat{R}(t )\Big)
\een
takes the following form
\ben
\Big(\Gamma^{\alpha'}(t ,\lambda)\otimes 
\Gamma^{-\alpha'}(t ,\lambda)\Big)\, 
\Big(\widehat{R}(t )\otimes 
\widehat{R}(t )\Big)\, 
\Big( c_-(t ,\lambda)\,
\Gamma^{-1/2}_{\rm pt}(u_i ,\lambda)\otimes 
\Gamma^{1/2}_{\rm pt}(u_i ,\lambda)\Big)
\een
where 
\ben
c_-(t ,\lambda) = b_{r_\varphi(\alpha)}(t ,\lambda)\,
e^{
\Omega^i_{\alpha',\varphi}(t ,\lambda,\lambda) + 
\tfrac{1}{4}
V(\mathbf{f}^-_\varphi, \mathbf{f}^-_\varphi)
}\, .
\een
The coefficient $c_-(t ,\lambda)$ can be expressed in terms of propagators again. We get the same formula as in \eqref{cpt} except that we have to replace $\varphi$ by $-\varphi$, that is, 
\ben
c_-(t ,\lambda)= 
b_{\alpha'}(t ,\lambda)\, 
\frac{
e^{-\pi\ii \langle \alpha',\varphi\rangle}
}{2\sqrt{\lambda-u_i }}\, .
\een
Since $\mathcal{D}_{\rm pt}$ is a tau-function of KdV, the formal power series 
\begin{align*}
\frac{1}{2\sqrt{\lambda-u_i }}\left(
\Gamma^{1/2}_{\rm pt}(u_i ,\lambda)\otimes 
\Gamma^{-1/2}_{\rm pt}(u_i ,\lambda)-
\Gamma^{-1/2}_{\rm pt}(u_i ,\lambda)\otimes 
\Gamma^{1/2}_{\rm pt}(u_i ,\lambda)\right)\, & 
\mathcal{D}_{\rm pt}(\hbar\sqrt{\Delta_i},\mathbf{q}'(u_i))\times \\
&
\times \mathcal{D}_{\rm pt}(\hbar\sqrt{\Delta_i},\mathbf{q}''(u_i))
\end{align*}
is regular in $\lambda$, that is, it is holomorphic at $\lambda=u_i $. Therefore, proving that \eqref{pm1_pair} is holomorphic at $\lambda=u_i$ is reduced to verifying that $c_+(t ,\lambda)=-c_-(t ,\lambda)$. On the other hand, 
\ben
c_+(t ,\lambda)/c_-(t ,\lambda) = 
e^{2\pi\ii \langle \alpha',\varphi\rangle} = 
e^{2\pi\ii \langle \alpha,\varphi\rangle +
\pi\ii \langle \varphi,\varphi\rangle } =-1
\een
because $\langle \alpha,\varphi\rangle\in \ZZ$ and $\langle \varphi,\varphi\rangle=1$. 

\subsection{From ancestors to the Kac-Wakimoto form of KdV}
Finally, it remains to prove that the restriction of 
\ben
\left(
\sum_{\alpha: (\alpha|\varphi)=\pm 2}
b_\alpha(t,\lambda) 
\Gamma^\alpha(t,\lambda)\otimes 
\Gamma^{-\alpha}(t,\lambda) - 
\sum_{n\in \ZZ} b_{n\delta}(t,\lambda) 
\Gamma^{n\delta}(t,\lambda)\otimes 
\Gamma^{-n\delta}(t,\lambda)\, L(t,\lambda)\right)
\mathcal{A}_t\otimes
\mathcal{A}_t
\een
to $
\widehat{\mathbf{f}}_{L-1}(t,\lambda)\otimes 1 - 
1\otimes \widehat{\mathbf{f}}_{L-1}(t,\lambda) = 2\pi \ii \, m
$ is regular at $\lambda=u_i$ where the sum over $\alpha$ is over all $\alpha\in \Phi_{\rm aff}$ such that $(\alpha|\varphi)=\pm 2$.  Let us choose an integer $k$ such that $\varphi + k (L-1)\in \Phi_{\rm aff}$. 
Let us compare the operator in the brackets with the composition of the following operators:
\ben
\Big(
\Gamma^{k(L-1)}(t,\lambda)\otimes 
\Gamma^{-k(L-1)}(t,\lambda)\Big)\, 
\Big(\sum_{n\in \ZZ} b_{n\delta}(t,\lambda) 
\Gamma^{n\delta}(t,\lambda)\otimes 
\Gamma^{-n\delta}(t,\lambda) \Big)
\een
and 
\ben
b_\varphi(t,\lambda)
\Gamma^{\varphi}(t,\lambda)\otimes 
\Gamma^{-\varphi}(t,\lambda) + 
b_{-\varphi}(t,\lambda)
\Gamma^{-\varphi}(t,\lambda)\otimes 
\Gamma^{\varphi}(t,\lambda)-
L(t,\lambda).
\een
According to formula \eqref{delta-period}, we have 
\ben
b_{\pm\varphi+n\delta}(t,\lambda)=
b_{\pm\varphi}(t,\lambda)\, b_{n\delta}(t,\lambda)\, 
e^{\mp 2nl(I^{(-1)}_\varphi(t,\lambda),1)}.
\een
Note that the exponential factor matches the exponential factor of the following composition
\ben
\Big(
\Gamma^{n\delta}(t,\lambda)\otimes 
\Gamma^{-n\delta}(t,\lambda)\Big) 
\Big(
\Gamma^{\pm\varphi}(t,\lambda)\otimes 
\Gamma^{\mp\varphi}(t,\lambda) \Big)= 
e^{2 \Omega(
\mathbf{f}_{n\delta}^+,
\mathbf{f}_{\pm\varphi}^-)}
\Gamma^{n\delta\pm\varphi}(t,\lambda)\otimes 
\Gamma^{-n\delta\mp\varphi}(t,\lambda). 
\een
Indeed, according to formula \eqref{per_delta} we have 
$\mathbf{f}_{n\delta}^+(t,\lambda,z)= nl \phi_{0,0}$ so the symplectic pairing in the above formula is $-nl (I^{(-1)}_\varphi(t,\lambda),1)$ as claimed. Note that $b_{\alpha+k(L-1)}=b_\alpha$ and that the set of all $\alpha\in \Phi_{\rm aff}$ such that $(\alpha|\varphi)=\pm 2$ coincides with the set $\pm(\varphi + k(L-1)+ n\delta) $ ($n\in \ZZ$). The conclusion is that the difference of the two operators is proportional to $1-\Gamma^{k(L-1)}(t,\lambda)\otimes 
\Gamma^{-k(L-1)}(t,\lambda) $. However the latter vanishes when restricted to $
\widehat{\mathbf{f}}_{L-1}(t,\lambda)\otimes 1 - 
1\otimes \widehat{\mathbf{f}}_{L-1}(t,\lambda) = 2\pi \ii \, m
$ because $\mathbf{f}_{L-1}^+(t,\lambda,z)=0$ so 
\ben
\Gamma^{k(L-1)}(t,\lambda)\otimes 
\Gamma^{-k(L-1)}(t,\lambda) = 
e^{k(
\widehat{\mathbf{f}}_{L-1}(t,\lambda)\otimes 1 - 
1\otimes \widehat{\mathbf{f}}_{L-1}(t,\lambda))
}=e^{2\pi\ii k m}=1.
\een
Since the sum $\sum_{n\in \ZZ}b_{n\delta}(t,\lambda) \Gamma^{n\delta}(t,\lambda)\otimes \Gamma^{-n\delta}(t,\lambda)$ is regular at $\lambda=u_i$, the problem is reduced to proving that 
\beq\label{phi-term}
\Big(
b_\varphi(t,\lambda)
\Gamma^{\varphi}(t,\lambda)\otimes 
\Gamma^{-\varphi}(t,\lambda) + 
b_{-\varphi}(t,\lambda)
\Gamma^{-\varphi}(t,\lambda)\otimes 
\Gamma^{\varphi}(t,\lambda)-
L(t,\lambda)\Big)\, 
\mathcal{A}_t\otimes \mathcal{A}_t
\eeq
is holomorphic at $\lambda=u_i$.

Let us recall that by definition 
\ben
b_\varphi(t,\lambda)=
\lim_{\lambda'\to \lambda } 
\frac{e^{\Omega_{\varphi,\varphi}(t,\lambda',\lambda)}}{
(\lambda'-\lambda)^2 }\, .
\een
Using Proposition \ref{prop:R-conj_vo} and recalling that our choice of the reference path of $\varphi$ is such that the connection formula holds: 
$\Omega_{\varphi,\varphi}(t,\lambda',\lambda)=
\Omega^i_{\varphi,\varphi}(t,\lambda',\lambda)$, we get 
\begin{align}
\label{KW-vo}
\Big(
b_\varphi(t,\lambda)
\Gamma^{\varphi}(t,\lambda)\otimes 
\Gamma^{-\varphi}(t,\lambda) + 
b_{-\varphi}(t,\lambda)
\Gamma^{-\varphi}(t,\lambda)\otimes 
\Gamma^{\varphi}(t,\lambda)
\Big)\, 
\widehat{R}(t)\otimes \widehat{R}(t)  & =  \\ 
\notag
\widehat{R}(t)\otimes \widehat{R}(t)\, 
\Big(
b^+_{\rm pt}(t,\lambda)
\Gamma^{+}_{\rm pt}(u_i,\lambda)\otimes 
\Gamma^{-}_{\rm pt}(u_i,\lambda) + 
b^-_{\rm pt}(t,\lambda)
\Gamma^{-}_{\rm pt}(u_i,\lambda)\otimes 
\Gamma^{+}_{\rm pt}(u_i,\lambda)
\Big)
\end{align}
where $b^{\pm}_{\rm pt}(u_i,\lambda)=\tfrac{1}{16(\lambda-u_i)^2}$ are the coefficients of the bilinear operator of the Kac--Wakimoto form of the KdV hierarchy -- see formula \eqref{kw_coef}.

Finally, it remains only to take care of the Virasoro operator $L(t,\lambda)$. Let us recall the formula for $L(t,\lambda)$ from Proposition \ref{prop:anc_cv}. Let us decompose 
\ben
\alpha_i=\alpha_i'+c_i \varphi(\nu),\quad 
c_i =\frac{1}{q+q^{-1}}\, h_\nu(\alpha_i,\varphi(-\nu))
\een
and 
\ben
\alpha^i={\alpha'}^i + c^i \varphi(-\nu),\quad 
c^i= \frac{1}{q+q^{-1}}\, h_\nu(\varphi, \alpha^i),
\een
where $\varphi(\pm \nu)$ denotes the reflection vector for the $\pm\nu$-twisted periods corresponding to the reference path of $\varphi$. 
This decomposition is chosen in such a way that 
$h_\nu(\alpha_i',\varphi(-\nu))=h_\nu(\varphi,{\alpha'}^i)=0$. Under these conditions the twisted periods $I^{(k+\nu)}_{\alpha'_i}(t,\lambda)$ and $I^{(k-\nu)}_{{\alpha'}^i}(t,\lambda)$ are holomorphic at $\lambda=u_i$ (see \cite{CM2024}, Section 2.3). 
Put 
\ben
\Phi^\nu_\alpha:= \phi^\nu_\alpha(t,\lambda)\otimes 1 - 
1\otimes \phi^{\nu}_\alpha(t,\lambda).
\een
Using the above decompositions we get
\ben
{\Phi^\nu_{\alpha_i}}_{(-1)} \Phi^{-\nu}_{\alpha^i} =
{\Phi^\nu_{\alpha'_i}}_{(-1)} \Phi^{-\nu}_{{\alpha'}^i}+
c_i{\Phi^\nu_{\varphi}}_{(-1)} \Phi^{-\nu}_{{\alpha'}^i}+
c^i{\Phi^\nu_{\alpha'_i}}_{(-1)} \Phi^{-\nu}_{\varphi}+
c_i c^i{\Phi^\nu_{\varphi}}_{(-1)} \Phi^{-\nu}_{\varphi}.
\een
The first term on the RHS is holomorphic at $\lambda=u_i$ so we can exclude it from our analysis. We claim that the second and the third terms vanish after we take the sum over $i$. Indeed, we have 
\ben
\sum_{i=1}^N c_i {\alpha'}^i = 
\frac{1}{q+q^{-1}} 
\sum_{i=1}^N
h_\nu(\alpha_i,\varphi) \alpha^i - 
\sum_{i=1}^N c_i c^i \varphi(-\nu) = 
\Big(
\frac{1}{q+q^{-1}} - 
\sum_{i=1}^N c_i c^i\Big) \, \varphi(-\nu)=0
\een
where for the last equality we used that 
\ben
\sum_{i=1}^N c_i c^i = 
\frac{1}{(q+q^{-1})^2}\, 
\sum_{i=1}^N
h_\nu(\varphi, \alpha^i)
h_\nu(\alpha_i,\varphi)= 
\frac{1}{(q+q^{-1})^2}\, h_\nu(\varphi,\varphi)=
\frac{1}{q+q^{-1}}.
\een
We get that up to terms holomorphic at $\lambda=u_i$ the Virasoro operator $L(t,\lambda)$ coincides with 
\ben
\frac{1}{2}\operatorname{Res}\frac{d\nu}{\nu(q+q^{-1})} 
\Big(
\phi^\nu_\varphi(t,\lambda)\otimes 1 - 
1\otimes \phi^{\nu}_\varphi(t,\lambda)
\Big)_{(-1)} \Big(
\phi^{-\nu}_\varphi(t,\lambda)\otimes 1 - 
1\otimes \phi^{-\nu}_\varphi(t,\lambda)
\Big)
\een
The above 1-form has a simple pole at $\nu=0$ so the residue turns into 
\ben
\frac{1}{4} 
\Big(
\phi_\varphi(t,\lambda)\otimes 1 - 
1\otimes \phi_\varphi(t,\lambda)
\Big)_{(-1)} \Big(
\phi_\varphi(t,\lambda)\otimes 1 - 
1\otimes \phi_\varphi(t,\lambda)
\Big)
\een
Since $\widehat{R}(t)$ is compatible with the Borcherds product and $\phi_\varphi(t,\lambda) \widehat{R}(t) = 
\widehat{R}(t)\phi_{\rm pt}(u_i,\lambda)$, we get that 
\ben
\frac{1}{4} 
\Big(
\phi_\varphi(t,\lambda)\otimes 1 - 
1\otimes \phi_\varphi(t,\lambda)
\Big)_{(-1)} \Big(
\phi_\varphi(t,\lambda)\otimes 1 - 
1\otimes \phi_\varphi(t,\lambda)
\Big)
\widehat{R}(t)\otimes \widehat{R}(t)
\een
coincides with 
$\widehat{R}(t)\otimes \widehat{R}(t)\, L_{\rm pt}(u_i,\lambda)$ where $L_{\rm pt}(u_i,\lambda)$ is the Virasoro term in the bilinear operator of the point. Finally, recalling \eqref{KW-vo} and that $\mathcal{D}_{\rm pt}$ is a solution to \eqref{KW-pt}, we get that \eqref{phi-term} is holomorphic at $\lambda=u_i$.\qed

\appendix
\section{Virasoro operators}\label{sec:viro}
Let us expand $L(\lambda)=\sum_{n\in \ZZ} L_n \lambda^{-n-2}$. We would like to give the explicit formulas for the operators $L_n$. 
It is convenient to make the following (standard) substitution:
\ben
\mathbf{q}'=\mathbf{x}+\mathbf{y},\quad 
\mathbf{q}'' = \mathbf{x}-\mathbf{y},
\een
where $\mathbf{x}=\{x_{k,j,p}\}$ and $\mathbf{y}=\{y_{k,j,p}\}$ are two copies of the dynamical variables $\mathbf{q}$. Since  
\ben
q'_{k,j,p}-q''_{k,j,p} = 2 y_{k,j,p},\quad 
\frac{\partial}{\partial q_{k,j,p}'} - 
\frac{\partial}{\partial q_{k,j,p}''} = 
\frac{\partial}{\partial y_{k,j,p}},
\een
the bilinear operator $\Omega_{\rm aff}(\lambda)$ becomes a differential operator in $\mathbf{y}$.
The Virasoro operators $L_n$ take the following form:
\ben
L_n=L^\perp_n + L_n^0 +
\frac{1}{2}\sum_{j=1}^3\sum_{p=1}^{a_j-1}
\frac{p}{a_j}\Big(1-\frac{p}{a_j}\Big)\, \delta_{n,0},
\een
where 
\begin{align*}
L^\perp_n = & 
\sum_{j=1}^3\sum_{p=1}^{a_j-1} \left( \frac{2}{\hbar} 
\sum_{\substack{k+l=-n-1\\k,l\geq 0}} 
\frac{1}{a_j}
\frac{y_{k,j,p}}{\Big(k-\tfrac{p^*}{a_j}\Big)!} 
\frac{y_{l,j,p^*}}{\Big(l-\tfrac{p}{a_j}\Big)!} + 
2 \sum_{\substack{k-l=-n\\k,l\geq 0}} 
\frac{\Big(l+\tfrac{p}{a_j}\Big)!}{\Big(k-\tfrac{p^*}{a_j}\Big)!}
y_{k,j,p}\frac{\partial}{\partial y_{l,j,p}} +
\right. \\
& 
\hphantom{\sum_{j=1}^3\sum_{p=1}^{a_j-1} () }
\left. +
\frac{\hbar}{2} 
\sum_{\substack{k+l=n-1\\k,l\geq 0}} 
a_j 
\Big(k+\tfrac{p}{a_j}\Big)!
\Big(l+\tfrac{p^*}{a_j}\Big)!
\frac{\partial^2}{\partial y_{k,j,p}\partial y_{l,j,p^*}}\right),
\end{align*}
and 
\begin{align*}
L_n^0 = &
\frac{4}{\hbar}
\sum_{\substack{k+l=-n-1\\k,l\geq 0}} 
(1+l(h_k-h_l))\frac{y_{k,0,1}}{k!} \frac{y_{l,0,0}}{l!} + \\
&+2
\sum_{\substack{k-l=-n\\k,l\geq 0}} \Big(
\frac{l!}{k!}\, k\, y_{k,0,0}\frac{\partial}{\partial y_{l,0,0}} + 
\frac{(l+1)!}{k!}\, 
y_{k,0,1}\frac{\partial}{\partial y_{l,0,1}} \Big).
\end{align*}
The notation in the above formulas is as follows. 
\begin{itemize}
\item 
All sums over $k$ and $l$ are sums over all pairs of integers $(k,l)$ satisfying the constraints under the summation sign. If there are no pairs $(k,l)$ satisfying the constraints, then the sum is by definition $0$. 
\item 
In the definition of $L_n^\perp$ we denoted $p^*:=a_j-p$. More precisely, there is an involution of the basis $\mathcal{B}$ defined by $(j,p)^*=(j,a_j-p)$, so our notation should be understood as $y_{k,j,p^*}:=y_{k,(j,p)^*}=y_{k,j,a_j-p}$. 
\item 
We defined $\alpha!$ for all real numbers $\alpha>-1$ in the following recursive way:
\ben
\alpha! = 
\begin{cases}
1 &  \mbox{ if } -1 < \alpha\leq 0, \\
\alpha\, (\alpha-1)! & \mbox{ if } \alpha > 0 .
\end{cases}
\een
\item 
We denoted by $h_k$ ($k\geq 0$) the {\em Harmonic} numbers defined recursively by 
\ben
h_0=0,\quad h_{k}=h_{k-1}+\frac{1}{k} \quad (k>0).
\een
\end{itemize}
The derivation of the above formulas is a straightforward computation using the definition of $L(\lambda)$ in Section \ref{sec:bilo} and the formulas for the calibrated periods in Proposition \ref{prop:cp}.

\bibliographystyle{plain}
\bibliography{elliptic}

\begin{thebibliography}{10}

\bibitem{AGV2008}
Dan Abramovich, Tom Graber, and Angelo Vistoli.
\newblock {Gromov--Witten theory of Delign--Mumford stacks}.
\newblock {\em Amer. J. Math.}, 130:1337--1398, 2008.

\bibitem{AGuV}
V.~I. Arnold, S.~M. Gusein-Zade, and A.~N. Varchenko.
\newblock {\em Singularities of Differentiable Maps, Volume II: Monodromy and
  Asymptotics of Integrals}.
\newblock Birkh{\"a}user, Boston, MA, 1985.

\bibitem{Ar1989}
Vladimir Arnold.
\newblock {\em Mathematical Methods of Classical Mechanics}, volume~60 of {\em
  Graduate Texts in Mathematics}.
\newblock Springer-Verlag, New York, 2nd edition, 1989.

\bibitem{BM2013}
Bojko Bakalov and Todor Milanov.
\newblock {W-constraints for the total descendant potential of a simple
  singularity}.
\newblock {\em Compositio Math.}, 149:840--888, 2013.

\bibitem{BLS2026}
Xavier Blot, Danilo Lewański, and Sergey Shadrin.
\newblock Beyond descendants: integrable observables for cohomological field
  theories.
\newblock arXiv:2605.22236v1[math.AG], 2026.

\bibitem{Bo1986}
Richard Borcherds.
\newblock {Vertex algebras, Kac--Moody algebras, and the Monster}.
\newblock {\em Proc. Natl. Acad. Sci. USA}, pages 3068--3071, 1986.

\bibitem{BPS2012}
Alexander Buryak, Hessel Posthuma, and Sergey Shadrin.
\newblock {A polynomial bracket for the Dubrovin-Zhang hierarchies}.
\newblock {\em Journal of Differential geometry}, 92:153--185, 2012.

\bibitem{CR2002}
Weimin Chen and Yongbin Ruan.
\newblock {Orbifold Gromov--Witten theory}.
\newblock In Alejandro Adem, Jack Morava, and Yongbin Ruan, editors, {\em
  Orbifolds in Mathematics and Physics}, volume 310 of {\em Contemp. Math.},
  pages 25--85. American Mathematical Society, Providence, RI, 2002.

\bibitem{CM2024}
John~Alexander Cruz~Morales and Todor Milanov.
\newblock Dubrovin conjecture and the second structure connection.
\newblock {\em J. Geom. Phys.}, 228:105931, 2026.

\bibitem{Du2004}
Boris Dubrovin.
\newblock On almost duality for {F}robenius manifolds.
\newblock In {\em Geometry, Topology, and Mathematical Physics}, volume 212 of
  {\em Amer. Math. Soc. Transl. Ser. 2}, pages 75--132. American Mathematical
  Society, Providence, RI, 2004.

\bibitem{DZ1999}
Boris Dubrovin and Youjin Zhang.
\newblock {Frobenius manifolds and Virasoro constraints}.
\newblock {\em Sel. math.}, 5:423–--466, 1999.

\bibitem{DZ2001}
Boris Dubrovin and Youjin Zhang.
\newblock {Normal forms of hierarchies of integrable PDEs, Frobenius manifolds,
  and Gromov--Witten invariants}.
\newblock arXiv:math/0108160v1, 2001.

\bibitem{FB2001}
Edward Frenkel and David Ben-Zvi.
\newblock {\em Vertex Algebras and Algebraic Curves}, volume~88 of {\em
  Mathematical Surveys and Monographs}.
\newblock American Mathematical Society, Providence, RI, 2001.

\bibitem{FGM2010}
Edward Frenkel, Alexander Givental, and Todor Milanov.
\newblock Soliton equations, vertex operators, and simple singularities.
\newblock {\em Functional Analysis and Other Mathematics}, 3(1):47 -- 63, 2010.

\bibitem{Ga1974}
Andrei Gabrielov.
\newblock { Dynkin diagrams of unimodal singularities}.
\newblock {\em Funkcional. Anal. i Prilozen}, 8(3):1–6, 1974.
\newblock Engl. translation in Funct. Anal. Appl. 8 (1974), 192–196.

\bibitem{Giv1988}
Alexander Givental.
\newblock Twisted {P}icard--{L}efschetz formulas.
\newblock {\em Funktsional. Anal. i Prilozhen.}, 22(1):10--18, 1988.

\bibitem{Giv2001}
Alexander Givental.
\newblock Gromov--{W}itten invariants and quantization of quadratic
  hamiltonians.
\newblock {\em Mosc. Math. J.}, 1:551--568, 2001.

\bibitem{Giv2003}
Alexander Givental.
\newblock { $A_{n-1}$ singularities and $n$KdV hierarchies}.
\newblock {\em Mosc. Math.~J.}, 3(2):475--505, 2003.

\bibitem{GM2005}
Alexander Givental and Todor Milanov.
\newblock Simple singularities and integrable hierarchies.
\newblock In Jerrold~E. Marsden and Tudor~S. Ratiu, editors, {\em The Breadth
  of Symplectic and Poisson Geometry}, volume 232 of {\em Progress in
  Mathematics}, pages 173--201. Birkh{\"a}user, Boston, MA, 2005.

\bibitem{Iri2009}
Hiroshi Iritani.
\newblock {An integral structure in quantum cohomology and mirror symmetry for
  toric orbifolds}.
\newblock {\em Adv. Math}, 222(3):1016--1079, 2009.

\bibitem{Iri2011}
Hiroshi Iritani.
\newblock {Quantum cohomology and periods}.
\newblock {\em Ann. Inst. Fourier, Grenoble}, 61(7):2909--2958, 2011.

\bibitem{IMRS}
Hiroshi Iritani, Todor Milanov, Yongbin Ruan, and Yefeng Shen.
\newblock {Gromov–Witten theory of quotients of Fermat Calabi–Yau
  varieties}.
\newblock {\em Memoirs of the American Mathematical Society}, 269(1310):104,
  2021.

\bibitem{Kac1998}
Victor Kac.
\newblock {\em Vertex Algebras for Beginners}, volume~10 of {\em University
  Lecture Series}.
\newblock American Mathematical Society, Providence, RI, 2nd edition, 1998.

\bibitem{KW1987}
Victor Kac and Minoru Wakimoto.
\newblock Exceptional hierarchies of soliton equations.
\newblock In {\em Theta functions—Bowdoin 1987}, volume 49, Part 1 of {\em
  Proceedings of Symposia in Pure Mathematics}, pages 191--237, Providence,
  Rhode Island, 1989. American Mathematical Society.

\bibitem{Ko1992}
Maxim Kontsevich.
\newblock {Intersection theory on the moduli space of curves and the matrix
  Airy function}.
\newblock {\em Commun. Math. Phys.}, 147:1--23, 1992.

\bibitem{LWZ2025}
Si-Qi Liu, Zhe Wang, and Youjin Zhang.
\newblock {Linearization of Virasoro symmetries associated with semisimple
  Frobenius manifolds}.
\newblock {\em Advances in Mathematics}, 460:110046, 2025.

\bibitem{Mi}
Todor Milanov.
\newblock {Hirota Quadratic Equations for the Extended Toda Hierarchy}.
\newblock {\em Duke Math. Journal}, 138(1):161--176, 2007.

\bibitem{Mi2014}
Todor Milanov.
\newblock Analyticity of the total ancestor potential in singularity theory.
\newblock {\em Advances in Mathematics}, 255(1):217--241, 2014.

\bibitem{Mi2019}
Todor Milanov.
\newblock The phase factors in singularity theory.
\newblock In {\em Primitive Forms and Related Topics -- Kavli IPMU 2014},
  volume~83 of {\em Advanced Studies in Pure Mathematics}, pages 295--326.
  Mathematical Society of Japan, Tokyo, 2019.

\bibitem{MiR}
Todor Milanov and Yongbin Ruan.
\newblock {Gromov-Witten theory of elliptic orbifold $P^1$ and quasi-modular
  forms}.
\newblock arXiv:1106.2321, 2011.

\bibitem{MilSa}
Todor Milanov and Kyoji Saito.
\newblock Primitive forms and vertex operators.
\newblock Book in preparation,
  https://member.ipmu.jp/todor.milanov/Research/pf-vo.pdf.

\bibitem{MiSh}
Todor Milanov and Yefeng Shen.
\newblock {The modular group for the total ancestor potential of Fermat simple
  elliptic singularities}.
\newblock {\em Comm. in Numb. Theory and Phys.}, 8(2):329--368, 2014.

\bibitem{MST2016}
Todor Milanov, Yefeng Shen, and Hsian-Hua Tseng.
\newblock {Gromov–Witten theory of Fano orbifold curves, Gamma integral
  structures and ADE-Toda hierarchies}.
\newblock {\em Geometry and Topology}, 20:2135--2218, 2016.

\bibitem{Mu}
David Mumford.
\newblock {\em Tata Lectures on Theta I}, volume~28 of {\em Progress in
  Mathematics}.
\newblock Birkh{\"a}user, Boston, MA, 1994.

\bibitem{Sa}
Kyoji Saito.
\newblock Extended affine root systems.
\newblock {\em Publ. RIMS, Kyoto Univ.}, 21(1):75--179, 1985.

\bibitem{Sato1981}
Mikio Sato.
\newblock Soliton equations as dynamical systems on infinite dimensional
  {G}rassmann manifold.
\newblock {\em S{\=u}rikaisekikenky{\=u}sho K{\=o}ky{\=u}roku}, 439:30--46,
  1981.

\bibitem{Te2012}
Constantin Teleman.
\newblock The structure of 2d semi-simple field theories.
\newblock {\em Invent. Math.}, 188(3):525--588, 2012.

\bibitem{Wi1991}
Edward Witten.
\newblock Two-dimensional gravity and intersection theory on moduli space.
\newblock {\em Surveys in Diff. Geom.}, 1:243--310, 1991.

\end{thebibliography}

\bigskip
\noindent
Kavli IPMU (WPI), UTIAS, The University of Tokyo\\
Kashiwa, Chiba, 277-8583, Japan \\
Email: \texttt{todor.milanov@ipmu.jp}

\end{document}